\documentclass[11pt]{amsart}

\usepackage{comment}
\usepackage{times}
\usepackage{amssymb}
\usepackage{graphicx,xspace}
\usepackage{epsfig}
\usepackage{enumitem}
\usepackage[usenames,dvipsnames,svgnames,table]{xcolor}
\usepackage{gensymb}
\usepackage[T1]{fontenc}
\usepackage[utf8]{inputenc}
\usepackage{bm}
\usepackage{marginnote, enumitem}
\usepackage[config, labelfont={normalsize}]{caption,subfig}

\usepackage{tikz}
\usetikzlibrary{decorations,calc,through,intersections, arrows}
\usetikzlibrary{decorations.pathmorphing, decorations.markings, decorations.pathreplacing}

\usepackage{hyperref}
\usepackage{todonotes}

\usepackage[noabbrev, capitalize]{cleveref}

\newsavebox{\measurebox}

\theoremstyle{plain}

\newtheorem{lemma}{Lemma}[section]
\newtheorem{theorem}[lemma]{Theorem}
\newtheorem{corollary}[lemma]{Corollary}
\newtheorem{proposition}[lemma]{Proposition}

\theoremstyle{remark}
\newtheorem{remark}{Remark}

\newtheorem{definition}[lemma]{Definition}

\newcommand{\qq}{\mathfrak{q}}
\newcommand{\pp}{\mathfrak{p}}

\newcommand{\R}{\mathbb{R}}

\newcommand{\N}{\mathbb{N}}

\newcommand{\q}{\bm{q}}

\newcommand{\M}{\mathcal{M}}

\newcommand{\NN}{\mathcal{N}}
\newcommand{\SSS}{\mathcal{S}}

\newcommand{\Z}{\mathbb{Z}}

\newcommand{\PP}{\mathcal{P}}

\newcommand{\Nhalf}{\frac{1}{2}\N}

\DeclareMathOperator{\dt}{dt}
\DeclareMathOperator{\Cov}{Cov}
\DeclareMathOperator{\suc}{succ}

\author[G.~Chapuy]{Guillaume Chapuy}
 \address{IRIF UMR CNRS 8243, Universit\'e Paris 7, Case 7014, 75205 Paris Cedex 13,
France\newline \indent CRM UMI CNRS 3457, Universit\'e de Montr\'eal, Montréal QC H3C 3J7, Canada
}
\email{guilaume.chapuy@liafa.univ-paris-diderot.fr}

\author[M.~Dołęga]{Maciej Dołęga}
\address{
Wydział Matematyki i Informatyki, 
Uniwersytet im.~Adama Mickiewicza, 
Collegium Mathematicum,
Umultowska 87, 
61-614 Poznań, 
Poland, \newline \indent Instytut Matematyczny,
Uniwersytet Wrocławski,  \mbox{pl.\ Grunwaldzki~2/4,} 50-384
Wrocław, Poland}
\email{maciej.dolega@amu.edu.pl}
\thanks{G.C. and M.D. acknowledge support from {\it Agence Nationale de la Recherche}, grant ANR 12-JS02-001-01 ``Cartaplus''. G.C. acknowledges support from {\it Ville de Paris}, grant ``Émergences 2013, Combinatoire à Paris''.}


\title[A bijection for general rooted maps]
{A bijection for rooted maps on general surfaces}

\begin{document}

\maketitle

\begin{abstract}

We extend the Marcus-Schaeffer bijection between orientable rooted bipartite quadrangulations (equivalently: rooted maps) and orientable labeled one-face maps to the case of \emph{all} surfaces, that is orientable and non-orientable as well. This general construction requires new ideas and is more delicate than the special orientable case, but it carries the same information. In particular, it leads to a uniform combinatorial interpretation of the counting exponent $\frac{5(h-1)}{2}$ for both orientable and non-orientable rooted connected maps of Euler characteristic $2-2h$, and of the algebraicity of their generating functions, 
similar to the one previously obtained in the orientable case via the Marcus-Schaeffer bijection.
It also shows that the renormalization factor $n^{1/4}$ for distances between vertices is universal for maps on all surfaces: the renormalized profile and radius in a uniform random pointed bipartite quadrangulation on any fixed surface converge in distribution when the size $n$ tends to infinity. Finally, we extend the 
Miermont and Ambjørn-Budd bijections
to the general setting of all surfaces. Our construction opens the way to the study of Brownian surfaces for any compact 2-dimensional manifold.
\end{abstract}

\section{Introduction}

\subsection{Maps}
\emph{Maps} (a.k.a. ribbon graphs, or embedded graphs) are combinatorial structures that describe the embedding of a graph in a surface (see Section~\ref{sec:def} for precise definitions). These objects have received much attention from many different viewpoints, because of their deep connections with various branches of discrete mathematics, algebra, or physics (see e.g.~\cite{LandoZvonkin2004, Eynard:book} and references therein). 
In particular, maps have remarkable enumerative properties, and the enumeration of maps (either by generating functions, matrix integral techniques, algebraic combinatorics, or bijective methods) is now a well established domain on its own. The reader may consult~\cite{AP, BF, Chapuy:PhD, Bernardi:HZ} for entry points into this fast-growing literature.
This paper is devoted to the extension of the bijective method of map enumeration to the case of \emph{all} surfaces (orientable and non-orientable), and to its first consequences in terms of combinatorial enumeration and probabilistic results.

\subsection{Orientable surfaces} Let us first recall briefly 
the situation in the orientable case. 
 A fundamental result of Bender and Canfield~\cite{BC1}, obtained with generating functions, says that the number $m_g(n)$ of rooted maps with $n$ edges on the orientable surface of genus~$g\geq 0$ (obtained by adding $g$ handles to a sphere, see Section~\ref{sec:def}) is asymptotically equivalent to
\begin{align}\label{eq:tgintro}
m_g(n) \sim t_g n^{\frac{5(g-1)}{2}} 12^n, \ \ \ n\rightarrow \infty,
\end{align}
for some $t_g>0$. In the planar case ($g=0$) this follows from the exact formula $m_0(n) = \frac{2\cdot 3^n (2n)!}{(n+2)!n!}$ due to Tutte~\cite{Tutte:census}, whose combinatorial interpretation was given by Cori-Vauquelin~\cite{CV} and much improved by Schaeffer~\cite{Schaeffer:phd}. The bijective enumerative theory of planar maps has since grown into a domain of research of its own, out of the scope of this introduction; consult~\cite{AP,BF} and references therein.  For general $g$, the combinatorial interpretation of Formula~\eqref{eq:tgintro}, and, in particular, of the counting exponent $\frac{5(g-1)}{2}$ was given in~\cite{ChapuyMarcusSchaeffer2009},
 using an extension of the Cori-Vauquelin-Schaeffer bijection to the case of higher genus orientable surfaces previously given by Marcus and Schaeffer in~\cite{MS}.
The bijection of~\cite{MS,ChapuyMarcusSchaeffer2009} associates maps on a surface with labeled one-face maps on the same surface. The latter have a much simpler combinatorial structure, and can be enumerated by elementary ways. 
Moreover, the bijective toolbox proved to be relatively flexible, and enabled to prove formulas similar to~\eqref{eq:tgintro} for many different families of orientable maps~\cite{Chapuy:constellations}, extending results previously obtained by generating functions~\cite{Gao}.

Beyond the combinatorial interpretation of counting formulas, an important motivation for developing the bijective methods is that they are the cornerstone of the study of random maps. In the planar case, Schaeffer's bijection was the starting point of the study of distance properties of random planar maps~\cite{ChassaingSchaeffer}. The field has now much developed, culminating with the proof by Le Gall~\cite{LG:GH} and Miermont~\cite{Miermont:GH} that uniform random quadrangulations, rescaled by $n^{1/4}$, converge in distribution to the so-called Brownian map (see the introduction of these papers for exact statements and for more references). For higher genus orientable surfaces, Bettinelli used the Marcus-Schaeffer bijection~\cite{MS, ChapuyMarcusSchaeffer2009} to prove the existence of scaling limits for uniform quadrangulations of fixed genus rescaled by the same exponent $n^{1/4}$, and to study some of their properties~\cite{Bettinelli1, Bettinelli2}.

\subsection{Non-orientable surfaces and main results} From the viewpoint of generating functions, orientable and non-orientable surfaces 
share many features.
In particular, in the paper previously mentioned~\cite{BC1}, Bender and Canfield 
showed, by the same method, that the number $\tilde{m}_h(n)$ of rooted maps with $n$ edges on the non-orientable surface of type $h\geq \frac{1}{2}$ (obtained by adding $2h$ cross-caps to a sphere, see Section~\ref{sec:def}; here $h$ is either an integer or a half-integer) is asymptotically equivalent to
\begin{align}\label{eq:pgintro}
\tilde{m}_h(n) \sim p_h n^{\frac{5(h-1)}{2}} 12^n, \ \ \ n\rightarrow \infty,
\end{align}
for some $p_h>0$. 
However, from the viewpoint of 
existing
bijective methods, the orientable and non-orientable situations are \emph{very} different: indeed all the existing bijections for maps (including the ones of~\cite{MS,ChapuyMarcusSchaeffer2009}) \emph{crucially} use the existence of a global orientation of the surface as a starting point of their construction. In particular, the existing literature says nothing about distance properties of random maps on non-orientable surfaces since it lacks the bijective tools to study them.

The main achievement of this paper is to bridge this gap. We are able to drop the assumption of orientability in the Marcus-Schaeffer construction~\cite{MS,ChapuyMarcusSchaeffer2009} and hence extend it to the case of  \emph{all} surfaces (orientable or non-orientable). For any surface $\mathbb{S}$ there is a  bijection between rooted maps on $\mathbb{S}$ (more precisely, rooted bipartite quadrangulations) and labeled one-face maps on $\mathbb{S}$. This is done at the cost of using only a \emph{local} orientation in the construction instead of a \emph{global} one: the local rules of the bijection are the same as in Marcus-Schaeffer bijection, but the local orientation is recursively constructed in a careful way that in the orientable case is consistent with the global orientation. In particular, the correspondence between \emph{distances} in the quadrangulation and \emph{labels} in the one-face map is preserved.
As a consequence, not only does the bijection lead to a \emph{uniform combinatorial interpretation} of both the orientable~\eqref{eq:tgintro} and non-orientable~\eqref{eq:pgintro} counting formulas, but it also enables to study distances in uniform random maps on a fixed surface. In particular we show that the exponent $n^{1/4}$ is \emph{universal} for distances in maps on \emph{all} surfaces in the following sense: the profile and radius of a uniform random bipartite quadrangulation of size $n$ on a fixed surface converge in distribution, after renormalization by $n^{1/4}$, as $n$ tends to infinity.

Finally, we also extend Miermont's generalization~\cite{Miermont} of the Marcus-Schaeffer bijection to the non-orientable case. It follows that the bijection of Ambjørn and Budd~\cite{AmbjornBudd} can be extended as well.

\subsection{Forthcoming work.} Our construction opens the way to the study of Gromov-Hausdorff scaling limits of random maps of size $n$, renormalized by $n^{1/4}$, on any fixed surface. 
More precisely, for any surface $\mathbb{S}$, let $\qq_n$ be a random map on $\mathbb{S}$ uniformly distributed over the set of rooted bipartite quadrangulations with~$n$ faces. In the forthcoming paper \cite{BettinelliChapuyDolega2015} joint with Jérémie Bettinelli, we will  
 prove that the sequence $\frac{1}{n^{1/4}}\qq_n$ converges to a limiting space $\mathbb{S}_{Brownian}$ in the sense of the Gromov--Hausdorff topology, and that this space has almost surely Haussdorff dimension $4$. These results, which require a fine probabilistic study of the objects inherited from our bijection, crucially rely on the present paper.

\subsection{Related problems}
As we already said, to our knowledge, the bijective enumeration of general maps and the study of distance properties of random maps on non-orientable surfaces have not been addressed before.
However, combinatorial maps on non-orientable surfaces have already been studied from different perspectives. First, there already exist bijective constructions for maps on non-orientable surfaces~\cite{BernardiChapuy:NonOrientable, Bernardi:HZ}, but they deal with the case of \emph{one-face maps}. This is a different problem than the one studied here, involving a very different structure (and although both these constructions and the present work involve one-face maps, we do not know of any interesting way to combine them together).
In a different direction, combinatorics of non-orientable maps was proved to play an important role in understanding the structure of the \emph{double coset algebra of the hyperoctahedral group} \cite{Goulden1996a, MoralesVassilieva2011} and, thus, some \emph{random real matrix models} \cite{HanlonStanleyStembridge1992, Goulden1997}, the structure of \emph{Jack symmetric functions} \cite{Goulden1996, LaCroix2009, DolegaFeraySniady2013} and the  structure of \emph{$\beta$-ensemble models} \cite{LaCroix2009}. It would be interesting to make a connection between these topics and the viewpoint of the present paper.

\subsection{Open question}
We conclude this introduction by mentioning an open question that naturally follows our work.
One may ask whether we can generalize our construction to maps more general than bipartite quadrangulations, for example bipartite maps with given face degrees, as was done in the orientable case~\cite{Chapuy:constellations}. Although it is natural to conjecture that this is possible (given the universal form of the generating functions appearing in~\cite{Gao}), certain arguments in the present paper crucially depend on the fact that we deal with quadrangulations only. We thus leave this question open.

\subsection{Structure of the paper} We now describe the structure of the paper and explain where to find the main results. Section~\ref{sec:def} sets up basic notation and definitions. Section~\ref{sec:bij} deals with the main construction and is divided as follows: the main result is stated as \Cref{theo:MS1} in \Cref{subsect:BijectionOverview}, and a quick overview of the bijection announced by this theorem is also given there. The bijection is fully described in Section~\ref{sec:DEG}, and the converse bijection is given in \Cref{sec:reversebij}, where we also establish some important properties of the construction, in particular a ``distance bounding lemma'' (\cref{lem:DistancesOfGeneric}) that is not used in the present paper but will be crucial for further work on scaling limits. The remaining proofs are given in \cref{sec:bijection}.
Section~\ref{sec:consequences} investigates the consequences of the bijection from the asymptotic and probabilistic viewpoint. This section is mostly devoted to sketching how to apply existing techniques to the present case, since once the bijection is established, most of the consequences we mention follow from the exact same arguments as in the orientable case. In particular we show how to compute combinatorially the generating functions of rooted bipartite quadrangulations on a fixed surface with our approach (\Cref{theo:rationality} and \Cref{prop:explicitGF}), how to derive combinatorially the counting exponents (\Cref{theo:asymptotics}), we study explicitly -- and bijectively -- the case of the projective plane (\Cref{cor:PP}), and we state various probabilistic consequences on the convergence of the normalized distances in \cref{theo:Distributions}.
Section~\ref{sec:Miermont} extends the construction to the case of multipointed maps, \textit{i.e.} it deals with the extension of the Miermont and Ambjørn-Budd bijections to the case of all surfaces, see~\cref{theo:Miermont2} and~\cref{theo:AB}.

\section{Surfaces, maps, and quadrangulations}
\label{sec:def}

\subsection{Surfaces, graphs, and maps}

A \emph{surface} is a compact, connected, $2$-dimensional manifold. We consider surfaces up to homeomorphism.
Denote $\N = \{0, 1, 2, 3, \dots \}$ and $\Nhalf = \{0, \frac{1}{2}, 1, \frac{3}
{2}, \dots \}$. 
 For any $h \in \N$, we denote by $\SSS_h$ the torus of
genus $h$, that is, the \emph{orientable} surface obtained by adding $h$ \emph{handles} to the sphere. For any
$h \in \Nhalf\setminus\{0\}$, we denote by $\NN_h$ the \emph{non-orientable} surface obtained by adding $2h$ \emph{cross-caps} to the
sphere. Hence, $\SSS_0$ is the sphere, $\SSS_1$ is the torus, $\NN_{1/2}$ is the projective plane and $\NN_1$ is the
Klein bottle. The \emph{type} of the surface $\SSS_h$ or $\NN_h$ is the number $h$. By the theorem of classification, surfaces are either orientable or non-orientable, each orientable surface is homeomorphic to one of the $\SSS_h$, and each non-orientable
surface is homeomorphic to one of the $\NN_h$ (see e.g. \cite{MoharThomassen2001}).

\medskip

Our \emph{graphs} are finite and undirected; loops and multiple
edges are allowed. A \emph{map} is an embedding (without edge-crossings) of a connected graph
into a surface, in such a way that the \emph{faces} (connected components of the complement of the
graph) are simply connected. Maps are always considered up to homeomorphism. A map is
\emph{unicellular} if it has a single face. Unicellular maps are also called \emph{one-face maps}. We will call a map \emph{orientable} if the underlying surface is orientable; otherwise we will call it \emph{non-orientable}.

Each edge in a map is made of two \emph{half-edges}, obtained by removing its middle-point.
The \emph{degree of a vertex} is the number of incident half-edges. A \emph{leaf} is a vertex of degree $1$. A
\emph{corner} in a map is an angular sector determined by a vertex, and two half-edges which are
consecutive around it. The \emph{degree of a face} is the number of edges incident
to it, with the convention that an edge incident to the same face on both sides
counts for two. Equivalently, the degree of a face is the number of corners lying in that face. Note that the total number of corners in a map equals the number of half-edges, and also equals twice the number of edges. A map is \emph{rooted} if it is equipped with a distinguished half-edge
called the \emph{root}, together with a distinguished side of this half-edge. The vertex incident to the
root is the \emph{root vertex}. The unique corner incident to the root half-edge and its distinguished
side is the \emph{root corner}.

Note that an equivalent way to root a map is to choose for it a root corner and an orientation of this corner: one can then define the root half-edge as the one lying to the right of the root corner (viewed from the root vertex and according to the root corner orientation). In this paper we will use both conventions depending on the situation, and we will switch from one to the other without explicit mention. 

On pictures, we will represent rooted maps by shading the root corner and by indicating the side of the root half-edge that is incident to it.
 \emph{From now on, all maps are rooted}.

The \emph{type} $h(\M)$ of a map $\M$ is the type of the underlying surface, that is to say, the Euler
characteristic of the surface is $2 - 2h(\M)$. If $\M$ is a map, we let $V(\M)$, $E(\M)$ and $F(\M)$ be its
sets of vertices, edges and faces. Their cardinalities $v(\M), e(\M)$ and $f(\M)$ satisfy the \emph{Euler formula}:
\begin{equation}
\label{eq:euler}
e(\M) = v(\M) + f(\M) - 2 + 2h(\M).
\end{equation}

\subsection{Representation of a unicellular map}
\label{subsec:RepresentationOfMap}

Since unicellular maps play an important role in this paper, we first spend some time discussing their basic properties, and how to represent them in several ways.

Since by definition the unique face of a unicellular map is simply connected, cutting the surface along the edges of a unicellular map with $n$ edges gives rise to a polygon with $2n$ \emph{edge-sides} (a $2n$-gon). This polygon inherits a root corner from the map, and can be oriented thanks to the distinguished side of the root half-edge. Conversely, any unicellular map can be obtained from a rooted oriented $2n$-gon, by gluing edge-sides of the polygon by pairs. It is important to note that given two edge-sides of a polygon, there are two ways of gluing them together: 
\begin{itemize}
\item
we call \emph{orientable
gluing} the identification of the edge-sides giving a topological cylinder; the corresponding edge in the map is called \emph{straight};
\item
we call \emph{non-orientable gluing}
the identification giving a M\"obius band; the corresponding edge in the map is called \emph{twisted}.
\end{itemize}

Observe that, if the edge-sides of the rooted $2n$-gon are oriented into a directed
cycle, then the orientable gluings (non-orientable, respectively) are the one for which  the two edge-sides glued together have opposite (the same) orientations along the gluing.
This corresponds to the fact that there are two kinds of edges in the corresponding unicellular map: straight edges are characterized by the fact that walking along the border of the map, we visit their two sides in opposite direction. On the contrary, the two sides of a twisted edge are visited twice in the same direction when following the border of the map. Note that there is no standard terminology for these two types of edges: for example, straight and twisted edges are sometimes called \emph{two-way} or \emph{one-way} edges~\cite{BernardiChapuy:NonOrientable}.
The set of the straight (twisted, respectively) edges of a unicellular map $\M$ will by denoted by $E_s(\M)$ ($E_t(\M)$, respectively). It is easy to see that $\M$ is orientable iff $E_t(\M) = \emptyset$.

Let $\M$ be a unicellular map and consider its underlying $2n$-gon. If we number the sides of the $2n$-gon from $1$ to $2n$ starting from the root corner and following its orientation, then $\M$ naturally induces a \emph{matching} of the set $[2n] := \{1,2, \dots,2n\}$ (two edge-sides of the $2n$-gon are matched if and only if they are glued together). Clearly, the knowledge of that matching and of the set of edges that are twisted or straight, is enough to reconstruct the polygon gluing, hence the map.
Consequently, there is a one-to-one correspondence between unicellular maps with $n$ edges and pairs of \emph{matchings} $\PP(E_s(\M)), \PP(E_t(\M))$ such that:
\begin{itemize}
\label{item:con1}
\item
they match different points: $\bigcup \PP(E_s(\M)) \cap \bigcup \PP(E_t(\M)) = \emptyset$;
\item
\label{item:con2}
their union $ \PP(E_s(\M)) \cup \PP(E_t(\M))$ is a perfect matching of the set $[2n]$.
\end{itemize}
See \cref{fig:RepresentationOfMap} for an illustration.

\begin{figure}
\centering
\subfloat[]{
	\label{subfig:RepresentationOfMap1}
	\includegraphics[height=65mm]{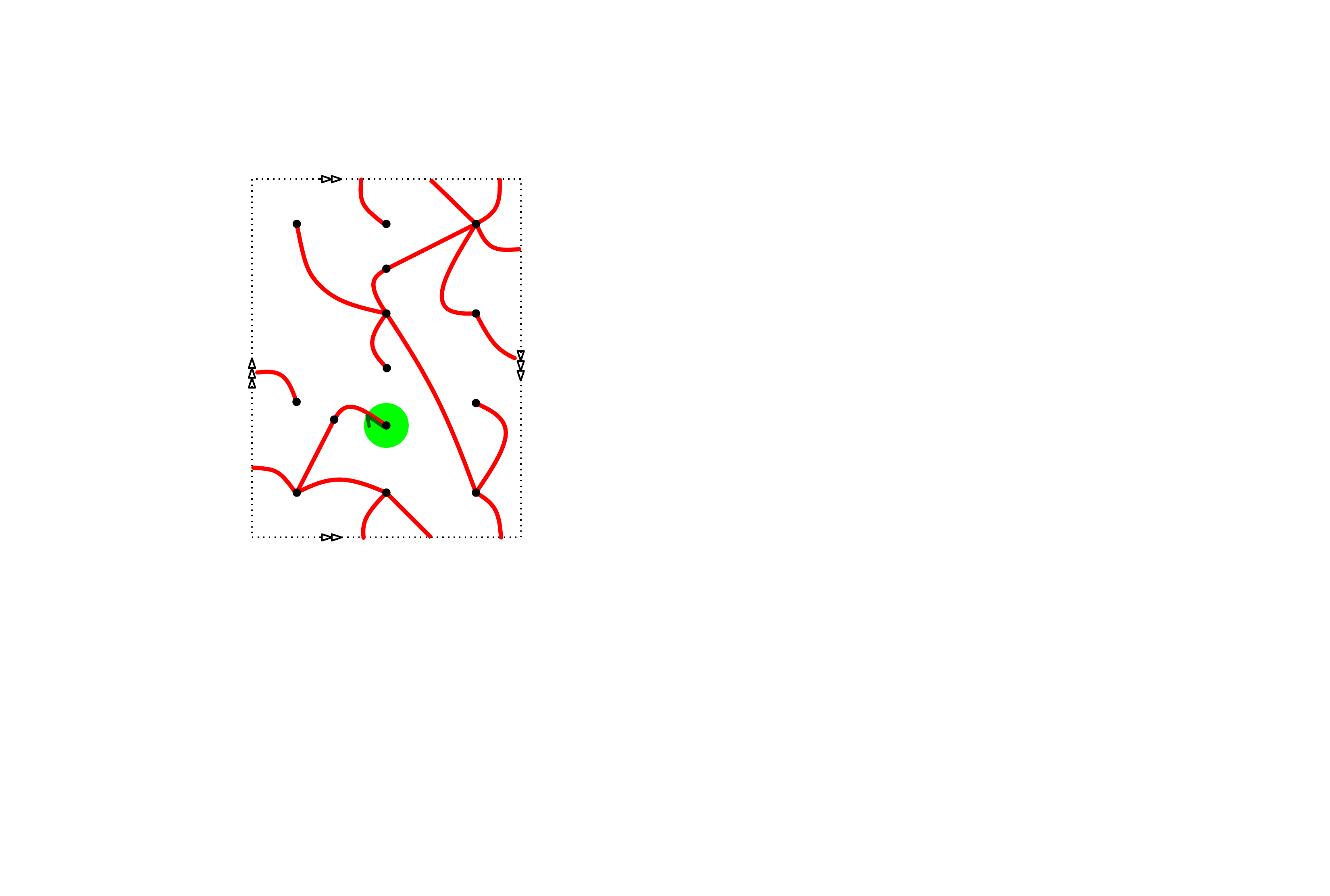}}
\quad
\subfloat[]{
	\label{subfig:RepresentationOfMap2}
	\includegraphics[height=65mm]{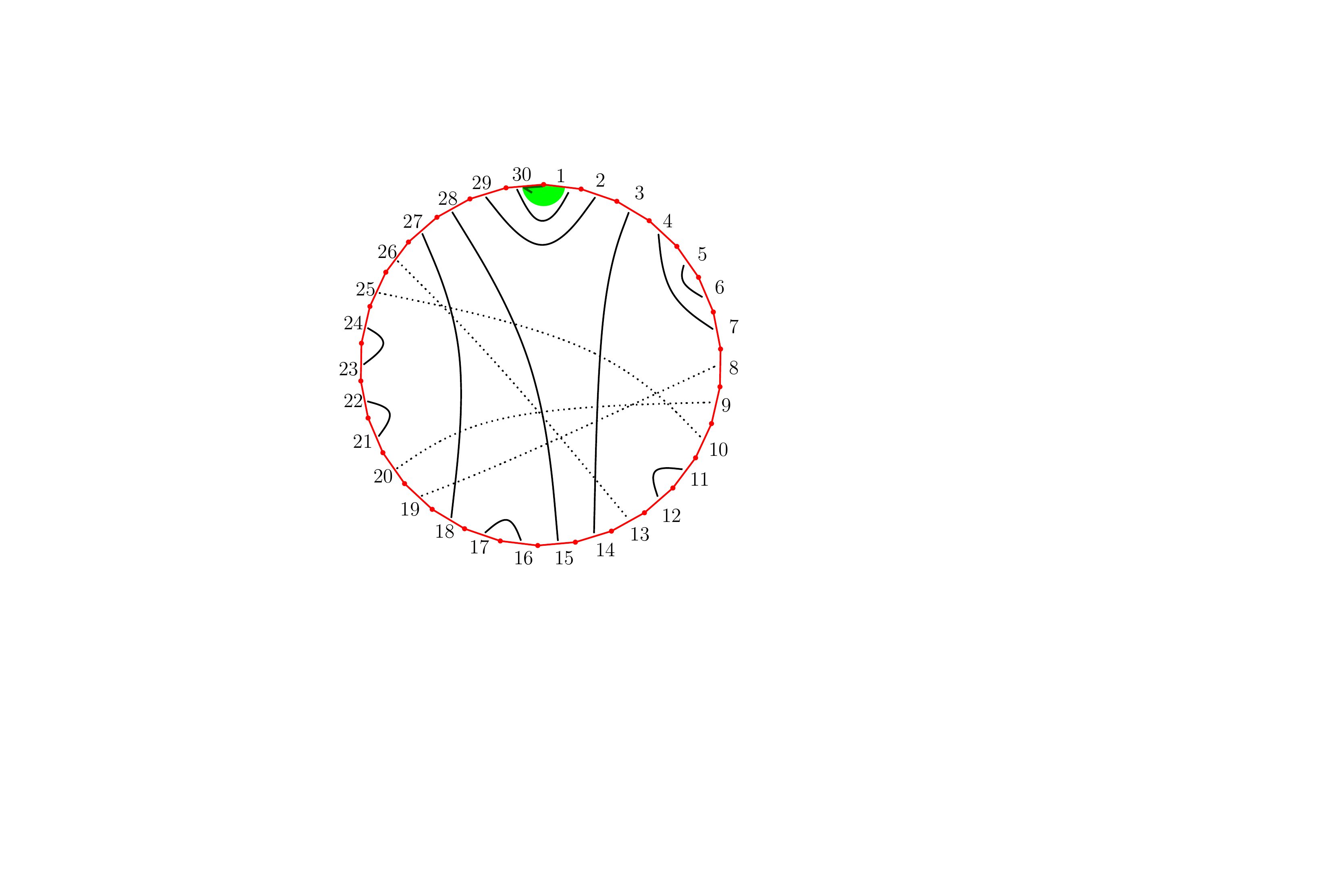}}
\caption{\protect\subref{subfig:RepresentationOfMap1}: A unicellular non-orientable map drawn on a Klein bottle $\NN_1$  (the left side of the dotted square should be glued to the right side,
as well as bottom to top, as indicated by arrows); \protect\subref{subfig:RepresentationOfMap2}: The same map represented as a polygon with glued sides (orientable gluings are indicated by gray, plain lines and non-orientable gluings are indicated by gray, dotted lines). 
The corresponding sets of pairs are $E_s(\M) = \{\{1,30\}, \{2,29\}, \{3,14\}, \{18,27\}, \{23,24\}, \{22,21\}, \{16,17\}$, $\{11,12\}\}$, and $E_t(\M) = \{\{13,26\}, \{10,25\}, \{9,20\}, \{8,19\}\}$.}
\label{fig:RepresentationOfMap}
\end{figure}

\subsection{Maps and bipartite quadrangulations}

A \emph{quadrangulation} is a map having all faces of degree $4$. A
map is \emph{bipartite} if its vertices can be colored in two colors in such a way that adjacent
vertices have different colors (say black and white). By convention the color of the root vertex of a rooted bipartite map is always taken to be black. We recall two
standard results of graph and map theory:
\begin{itemize}
\item
All planar quadrangulations are bipartite.
\item
For all $n, v, f \geq 1$ and any surface $\mathbb{S}$ (orientable or not), there is a bijection between maps on $\mathbb{S}$ with $n$ edges, $v$ vertices and $f$ faces, and bipartite
quadrangulations on $\mathbb{S}$ with $n$ faces, $v$ black and $f$ white vertices (and $2n$
edges). Idem between rooted maps and rooted bipartite quadrangulations.
\end{itemize}
The first result will not be used but is recalled to stress the fact that, a contrario,
a similar statement does not hold for quadrangulations of type $h > 0$: there exist non-bipartite quadrangulations on any surface of positive type.

The second result is based on a classical construction that goes back to Tutte, and which we now briefly recall.
Given a map $\M$ with black vertices, add a new (white) vertex inside each face of $\M$, and link it by a new edge to all the corners incident to that face. By construction, the map $\qq$ obtained by keeping all (black and white) vertices, and all the newly created edges, is a bipartite quadrangulation. Conversely, given a bipartite quadrangulation $\qq$, add a diagonal between the two black corners inside each face: then the set of black vertices, together with the added diagonal edges, forms a map whose associated quadrangulation is $\qq$.

To get a correspondence in the rooted case, a
rerooting convention must be chosen: if $\M$ has root half-edge $e$ and root corner $c$, then let the root of
$\qq$ be the unique half-edge of $\qq$ that lies in the corner $c$ and let the root corner of $\qq$ be the unique corner of $\qq$ where the half-edge $e$ lies. An example of this correspondence is illustrated on \cref{fig:maps-quadrangulations}. This correspondence is the reason why, although we concentrate
in the rest of this text on bipartite quadrangulations, our results have implications
for all maps.

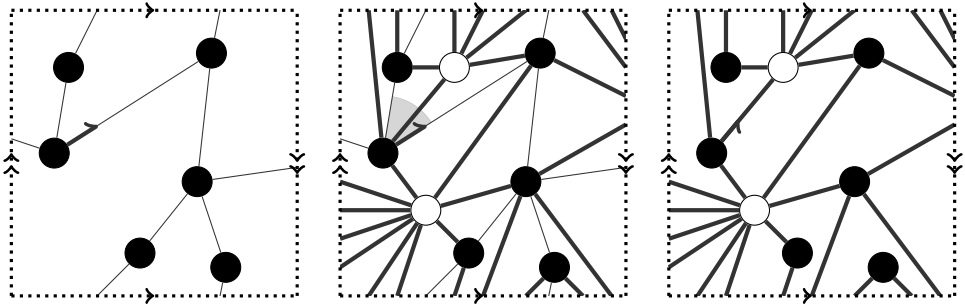
\begin{figure}
	\centering
	\begin{tikzpicture}[scale=0.38,
	white/.style={circle,draw=black,inner sep=4pt},
	black/.style={circle,draw=black,fill=black,inner sep=4pt},
	blue/.style={circle,draw=blue,fill=blue!80,inner sep=1pt},
	connection/.style={draw=black!80,black!80,auto}]
	\footnotesize
	
	\begin{scope}
		\begin{scope}
		\clip (0,0) rectangle (10,10);
		
		\draw (1.5,5) node (b1) [black] {};
		\draw (2,8) node (b2) [black] {};
		\draw (4.5,1.5) node (b3) [black] {};
		\draw (6.5,4) node (b4) [black] {};
		\draw (7,8.5) node (b5) [black] {};
		\draw (7.5,1) node (b6) [black] {};
		
		\draw (b1) +(0,10) node (b1n) [black] {};
		\draw (b1) +(0,-10) node (b1s) [black] {};
		\draw (b2) +(0,10) node (b2n) [black] {};
		\draw (b2) +(0,-10) node (b2s) [black] {};
		\draw (b3) +(0,10) node (b3n) [black] {};
		\draw (b3) +(0,-10) node (b3s) [black] {};
		\draw (b4) +(0,10) node (b4n) [black] {};
		\draw (b4) +(0,-10) node (b4s) [black] {};
		\draw (b5) +(0,10) node (b5n) [black] {};
		\draw (b5) +(0,-10) node (b5s) [black] {};
		\draw (b6) +(0,10) node (b6n) [black] {};
		\draw (b6) +(0,-10) node (b6s) [black] {};
		
		\draw[connection] (b2) -- (b1); 
		\draw[connection, -left to, ultra thick, shorten >=45pt] (b1) -- (b5);
		\draw[connection] (b1) -- (b5);
		\draw[connection] (b2) -- (3,10);
		\draw[connection] (3,0) -- (b3);
		\draw[connection] (b3) -- (b4);
		\draw[connection] (b6) -- (b4);
		\draw[connection] (b4) -- (b5);
		\draw[connection] (b5) -- (b6n);
		\draw[connection] (b5s) -- (b6);
		\draw[connection] (b4) -- (10,4.5);
		\draw[connection] (0,5.5) -- (b1);
		
		\end{scope}

	\draw[very thick,dotted, decoration={
	    markings,
	    mark=at position 0.5 with {\arrow{>}}},
	    postaction={decorate}]  
	(0,0) -- (10,0);
	
	\draw[very thick,dotted, decoration={
	    markings,
	    mark=at position 0.5 with {\arrow{>}}},
	    postaction={decorate}]  
	(0,10) -- (10,10);
	
	\draw[very thick,dotted, decoration={
	    markings,
	    mark=at position 0.5 with {\arrow{>>}}},
	    postaction={decorate}]  
	(0,0) -- (0,10);
	
	\draw[very thick,dotted, decoration={
	    markings,
	    mark=at position 0.5 with {\arrow{<<}}},
	    postaction={decorate}]  
	(10,0) -- (10,10);
	
	\end{scope}
	
	\begin{scope}[xshift=11.5cm]
		\begin{scope}
		\clip (0,0) rectangle (10,10);
		\begin{scope}
			\path[clip] (2,8) -- (1.5,5) -- (7,8.5);
			\fill[black!80, opacity=0.2, draw=black] (1.5,5) circle (2cm);
		\end{scope}
		
		\draw (1.5,5) node (b1) [black] {};
		\draw (2,8) node (b2) [black] {};
		\draw (4.5,1.5) node (b3) [black] {};
		\draw (6.5,4) node (b4) [black] {};
		\draw (7,8.5) node (b5) [black] {};
		\draw (7.5,1) node (b6) [black] {};
		
		\draw (b1) +(0,10) node (b1n) [black] {};
		\draw (b1) +(0,-10) node (b1s) [black] {};
		\draw (b2) +(0,10) node (b2n) [black] {};
		\draw (b2) +(0,-10) node (b2s) [black] {};
		\draw (b3) +(0,10) node (b3n) [black] {};
		\draw (b3) +(0,-10) node (b3s) [black] {};
		\draw (b4) +(0,10) node (b4n) [black] {};
		\draw (b4) +(0,-10) node (b4s) [black] {};
		\draw (b5) +(0,10) node (b5n) [black] {};
		\draw (b5) +(0,-10) node (b5s) [black] {};
		\draw (b6) +(0,10) node (b6n) [black] {};
		\draw (b6) +(0,-10) node (b6s) [black] {};

		\draw (3,3) node (w1) [white] {};
		\draw (4,8) node (w2) [white] {};
		\draw (w1) +(0,10) node (w1n) [white] {};

		\draw[connection, ultra thick] (w1) -- (b1);
		\draw[connection, ultra thick] (w1) -- (b3);
		\draw[connection, ultra thick] (w1) -- (b4);
		\draw[connection, ultra thick] (w1) -- (b5);
		\draw[connection, ultra thick] (w1) -- (2,0);
		\draw[connection, ultra thick] (2,10) -- (b2);
		\draw[connection, ultra thick] (w1) -- (1,0);
		\draw[connection, ultra thick] (1,10) -- (b1);
		\draw[connection, ultra thick] (w1) -- (0,1);
		\draw[connection, ultra thick] (10,9) -- (9.5,10);
		\draw[connection, ultra thick] (9.5,0) -- (b4);
		\draw[connection, ultra thick] (w1) -- (0,2);
		\draw[connection, ultra thick] (10,8) -- (8.5,10);
		\draw[connection, ultra thick] (8.5,0) -- (b6);
		\draw[connection, ultra thick] (w1) -- (0,3);
		\draw[connection, ultra thick] (10,7) -- (b5);
		\draw[connection, ultra thick] (w1) -- (0,4);
		\draw[connection, ultra thick] (10,6) -- (b4);
		
		\draw[connection, ultra thick] (w2) -- (b2);
		\draw[connection, ultra thick] (w2) -- (b1);
		\draw[connection, ultra thick] (w2) -- (b5);
		\draw[connection, ultra thick] (w2) -- (6.5,10);
		\draw[connection, ultra thick] (6.5,0) -- (b6);
		\draw[connection, ultra thick] (w2) -- (5,10);
		\draw[connection, ultra thick] (5,0) -- (b4);
		\draw[connection, ultra thick] (w2) -- (4,10);
		\draw[connection, ultra thick] (4,0) -- (b3);
		
		\draw[connection] (b2) -- (b1); 
		\draw[connection, -left to, ultra thick, shorten >=45pt] (b1) -- (b5);
		\draw[connection] (b1) -- (b5);
		\draw[connection] (b2) -- (3,10);
		\draw[connection] (3,0) -- (b3);
		\draw[connection] (b3) -- (b4);
		\draw[connection] (b6) -- (b4);
		\draw[connection] (b4) -- (b5);
		\draw[connection] (b5) -- (b6n);
		\draw[connection] (b5s) -- (b6);
		\draw[connection] (b4) -- (10,4.5);
		\draw[connection] (0,5.5) -- (b1);
		
		\end{scope}
	
	\draw[very thick, dotted, decoration={
	    markings,
	    mark=at position 0.5 with {\arrow{>}}},
	    postaction={decorate}]  
	(0,0) -- (10,0);
	
	\draw[very thick,dotted, decoration={
	    markings,
	    mark=at position 0.5 with {\arrow{>}}},
	    postaction={decorate}]  
	(0,10) -- (10,10);
	
	\draw[very thick,dotted, decoration={
	    markings,
	    mark=at position 0.5 with {\arrow{>>}}},
	    postaction={decorate}]  
	(0,0) -- (0,10);
	
	\draw[very thick,dotted, decoration={
	    markings,
	    mark=at position 0.5 with {\arrow{<<}}},
	    postaction={decorate}]  
	(10,0) -- (10,10);
	\end{scope}
	
	\begin{scope}[xshift=23cm]
		\begin{scope}
		\clip (0,0) rectangle (10,10);
		
		\draw (1.5,5) node (b1) [black] {};
		\draw (2,8) node (b2) [black] {};
		\draw (4.5,1.5) node (b3) [black] {};
		\draw (6.5,4) node (b4) [black] {};
		\draw (7,8.5) node (b5) [black] {};
		\draw (7.5,1) node (b6) [black] {};
		
		\draw (b1) +(0,10) node (b1n) [black] {};
		\draw (b1) +(0,-10) node (b1s) [black] {};
		\draw (b2) +(0,10) node (b2n) [black] {};
		\draw (b2) +(0,-10) node (b2s) [black] {};
		\draw (b3) +(0,10) node (b3n) [black] {};
		\draw (b3) +(0,-10) node (b3s) [black] {};
		\draw (b4) +(0,10) node (b4n) [black] {};
		\draw (b4) +(0,-10) node (b4s) [black] {};
		\draw (b5) +(0,10) node (b5n) [black] {};
		\draw (b5) +(0,-10) node (b5s) [black] {};
		\draw (b6) +(0,10) node (b6n) [black] {};
		\draw (b6) +(0,-10) node (b6s) [black] {};

		\draw (3,3) node (w1) [white] {};
		\draw (4,8) node (w2) [white] {};
		\draw (w1) +(0,10) node (w1n) [white] {};

		\draw[connection, ultra thick] (w1) -- (b1);
		\draw[connection, ultra thick] (w1) -- (b3);
		\draw[connection, ultra thick] (w1) -- (b4);
		\draw[connection, ultra thick] (w1) -- (b5);
		\draw[connection, ultra thick] (w1) -- (2,0);
		\draw[connection, ultra thick] (2,10) -- (b2);
		\draw[connection, ultra thick] (w1) -- (1,0);
		\draw[connection, ultra thick] (1,10) -- (b1);
		\draw[connection, ultra thick] (w1) -- (0,1);
		\draw[connection, ultra thick] (10,9) -- (9.5,10);
		\draw[connection, ultra thick] (9.5,0) -- (b4);
		\draw[connection, ultra thick] (w1) -- (0,2);
		\draw[connection, ultra thick] (10,8) -- (8.5,10);
		\draw[connection, ultra thick] (8.5,0) -- (b6);
		\draw[connection, ultra thick] (w1) -- (0,3);
		\draw[connection, ultra thick] (10,7) -- (b5);
		\draw[connection, ultra thick] (w1) -- (0,4);
		\draw[connection, ultra thick] (10,6) -- (b4);
		
		\draw[connection, -right to, ultra thick, shorten >=20pt] (b1) -- (w2);
		\draw[connection, ultra thick] (b1) -- (w2);
		\draw[connection, ultra thick] (w2) -- (b2);
		\draw[connection, ultra thick] (w2) -- (b5);
		\draw[connection, ultra thick] (w2) -- (6.5,10);
		\draw[connection, ultra thick] (6.5,0) -- (b6);
		\draw[connection, ultra thick] (w2) -- (5,10);
		\draw[connection, ultra thick] (5,0) -- (b4);
		\draw[connection, ultra thick] (w2) -- (4,10);
		\draw[connection, ultra thick] (4,0) -- (b3);
		
		\end{scope}
	
	\draw[very thick, dotted, decoration={
	    markings,
	    mark=at position 0.5 with {\arrow{>}}},
	    postaction={decorate}]  
	(0,0) -- (10,0);
	
	\draw[very thick,dotted, decoration={
	    markings,
	    mark=at position 0.5 with {\arrow{>}}},
	    postaction={decorate}]  
	(0,10) -- (10,10);
	
	\draw[very thick,dotted, decoration={
	    markings,
	    mark=at position 0.5 with {\arrow{>>}}},
	    postaction={decorate}]  
	(0,0) -- (0,10);
	
	\draw[very thick,dotted, decoration={
	    markings,
	    mark=at position 0.5 with {\arrow{<<}}},
	    postaction={decorate}]  
	(10,0) -- (10,10);
	\end{scope}
	\end{tikzpicture}
	\caption{Tutte's bijection between maps and bipartite quadrangulations. Left: a map of the Klein bottle $\NN_1$: the left side of the dotted square should be glued to the right side,
	as well as bottom to top, as indicated by arrows. Center and Right: constructing the associated bipartite quadrangulation. The root corner of the quadrangulation is shaded in the second step to help the reader visualize the rooting convention.}
	\label{fig:maps-quadrangulations}
\end{figure}

\section{The bijection}
\label{sec:bij}

\subsection{Statement of the main result}
\label{subsect:BijectionOverview}

A unicellular map $\M$ is called \emph{labeled} if its vertices are labeled by integers such that:
\begin{itemize}
\item[(1)] the root vertex has label $1$;
\item[(2)] if two vertices are linked by an edge, their label differ by at most $1$.
\end{itemize}
If in addition we have:
\begin{itemize}
\item[(3)] all the vertex labels are positive (that is, the root vertex has the minimum label),
\end{itemize}
then the unicellular map is called \emph{well-labeled}.

\medskip

The main result of this paper is a bijection establishing the following result:
\begin{theorem}\label{theo:MS1}
For each surface $\mathbb{S}$ and integer $n\geq 1$, there exists a bijection between the set of rooted bipartite quadrangulations on $\mathbb{S}$ with $n$ faces, and the set of well-labeled unicellular maps on $\mathbb{S}$ with $n$ edges.

Moreover, if for a given bipartite quadrangulation we denote by $N_i$ the set of its vertices at distance $i$ from the root vertex, and by $E(N_i,N_{i-1})$ the set of edges between $N_i$ and $N_{i-1}$, then the associated well-labeled unicellular map has $|N_i|$ vertices of label $i$ and $|E(N_i,N_{i-1})|$ corners of label $i$.
\end{theorem}

It is easy to see, using Euler's formula, that a quadrangulation of type $h$ with $n$ faces (and $2n$ edges) has $n+2-2h$ vertices, hence $n+1-2h$ vertices different from the root vertex. So the above statement is compatible with the fact that (by Euler's formula) a unicellular map of type $h$ with $n$ edges (and $2n$ corners) has $n+1-2h$ vertices in total.

As we will see, the above theorem easily implies the following variant, which turns out to be easier to use for enumerative purposes:
\begin{theorem}
\label{theo:MS2}
For each surface $\mathbb{S}$ and integer $n\geq 1$, there exists a $2$-to-$1$ correspondence between the set of rooted bipartite quadrangulations on $\mathbb{S}$ with $n$ faces carrying a pointed vertex $v_0$, and labeled unicellular maps on $\mathbb{S}$ with $n$ edges.

Moreover, if for a given bipartite quadrangulation we denote by $N_i$ the set of its vertices at distance $i$ from the vertex $v_0$, and by $E(N_i,N_{i-1})$ the set of edges between $N_i$ and $N_{i-1}$, then the associated labeled unicellular map has $|N_i|$ vertices of label $i+\ell_{min}-1$ and $|E(N_i,N_{i-1})|$ corners of label $i+\ell_{min}-1$, where $\ell_{min}$ is the minimum vertex label in the unicellular map.
\end{theorem}

It is natural to ask (at least for readers familiar with the orientable case) if the rooting in the last theorem is necessary, or if it is an artifact of our method. Indeed, in the case of orien\emph{ted} surfaces, the Marcus-Schaeffer bijection in its purest form is a bijection between \emph{unrooted} bipartite quadrangulations with a pointed vertex and \emph{unrooted} labeled unicellular maps
(by orien\emph{ted} we mean a surface which is not only orien\emph{table}, but in which an orientation has been fixed; maps are considered up to homeomorphisms preserving this orientation). The bijection preserves the size of the automorphism group, which is why it extends to a correspondence between rooted objects. 
We first observe that this property does not hold in the case of general surfaces, and in fact it is already not true for orien\emph{table} surfaces: the numbers do not agree for $n=4$, neither on the Klein bottle nor on the non-oriented torus.
However, one could ask the following question: is there, for any surface, a bijection between unrooted bipartite quadrangulations with a pointed vertex whose neighbourhood is oriented, and unrooted labeled one-face maps in which the unique face is oriented? We have tried to enumerate the first cases on the Klein bottle by hand, but even for $n=4$ the numbers are big and it is easy to make a mistake, so we prefer not to make any conjecture in general (this property is true for orientable surfaces).

\subsection{Digression: the dual exploration graph in the orientable case}
\label{subsect:OrientableCase}

The full construction leading to Theorem~\ref{theo:MS1} goes in two steps. First, one has to construct what we call the \emph{dual exploration graph (DEG)} and then use it to construct the labeled unicellular map associated with a quadrangulation. This is different from the classical presentation in the orientable case, in which one can construct the labeled unicellular map directly, without mentioning the DEG, but where the DEG then appears as a crucial tool of proof (see~\cite{MS,ChapuyMarcusSchaeffer2009} where a version of the DEG appears, even if it is not given that name). For the convenience of the reader, we will thus start by recalling the construction in the orientable case. This should help motivating the general construction, given in Section~\ref{sec:DEG}, where the DEG plays the most important role.

\begin{figure}[h!]
\centering
\includegraphics[scale=0.6]{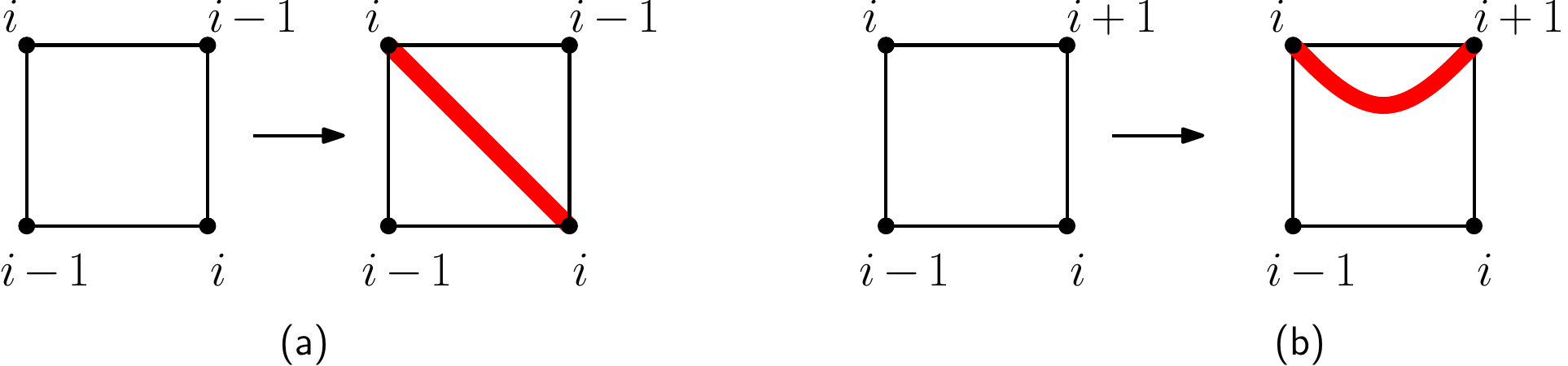}
\caption{The rules of the Marcus-Schaeffer bijection in the \emph{orientable} case. The rule in case (b) relies on the underlying orientation of the surface.}
\label{fig:basicFaceRules}
\end{figure}

Fix a rooted bipartite quadrangulation $\qq$ with $n$ faces on an \emph{orientable} surface $\mathbb{S}$, and let $v_0$ be the root vertex of $\qq$. Note that $\mathbb{S}$ is globally oriented by the root orientation of $\qq$.
 We start by labeling the vertices of $\qq$ by their distance to $v_0$. Since $\qq$ is bipartite, the labels at the extremities of each edge differ by exactly one. Therefore, the labels of corners on the border of a face form either a sequence of the form $(i - 1, i, i - 1, i)$ or of the form $(i - 1, i, i + 1, i)$ for some $i\geq 1$.
To construct the unicellular map $\Phi(\qq)$ associated with $\qq$, we add a new edge inside each face of $\qq$, which connects the two corners of that face that have a vertex label larger than their predecessor in clockwise direction around that face:
\begin{enumerate}[label=(\alph*)]
\item
\label{a}
in each face bordered by labels $(i - 1, i, i - 1, i)$ for some $i\geq 1$, the resulting edge has label $(i,i)$ (see \cref{fig:basicFaceRules}(a)); 
\label{item:basicFaceRules1}
\item
\label{b}
in each face bordered by labels $(i - 1, i, i + 1, i)$ for some $i\geq 1$, the resulting edge has label $(i,i+1)$ (see \cref{fig:basicFaceRules}(b)).
\end{enumerate}
Note that in case (a), we do not need the global orientation of $\mathbb{S}$ to specify which edge we add, but in case (b), the construction relies on the notion of ``clockwise direction'' along the face. This fact will create the main difficulty in extending the construction to the non-orientable case.

We now sketch the main idea of Marcus and Schaeffer to prove that the graph $\Phi(\qq)$ defined by the newly added edges on the vertex set $V(\qq)\setminus\{v_0\}$ is a one-face map on $\mathbb{S}$. It relies on an auxiliary graph that we call the \emph{dual exploration graph (DEG)}.
This graph is defined in the following way: we draw a new vertex inside each face of the map $\qq \cup \Phi(\qq)$ (the map consisting of the union of the quadrangulation $\qq$ and all the newly created edges) and we add one new edge going across each edge of $\qq$ as in \cref{fig:basicFaceRulesDEG}. Since every edge of $\qq$ has extremities labeled by consecutive numbers, we can orient each edge of the DEG such that it passes through the corresponding edge of $\qq$ in a way that a vertex of greater label is lying on its right side.
\begin{figure}[h!]
\centering
\includegraphics[scale=0.6]{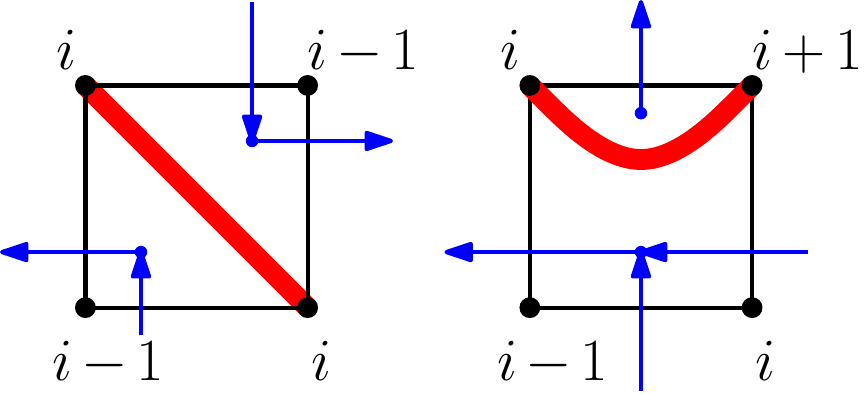}
\caption{The DEG in the orientable case (blue edges)}
\label{fig:basicFaceRulesDEG}
\end{figure}
By construction, the DEG contains a cycle that goes around the vertex~$v_0$. This is, in fact, its only cycle: if one contracts the cycle going around $v_0$ into a single vertex, the DEG becomes a tree. To see this, 
one first notices that each vertex of the DEG has a unique outgoing edge: therefore any cycle of the DEG is an oriented cycle.
Then, one has to examine how labels vary locally along edges of the DEG:
we first observe from Figure~\ref{fig:basicFaceRulesDEG} that the label present on the left of an oriented edge can only decrease when walking on the DEG. This implies that all the labels lying on the left side of each oriented cycle of the DEG are the same. By looking at Figure~\ref{fig:basicFaceRulesDEG} again, one deduces that any cycle in fact circles around a unique vertex, and moreover this vertex has only neighbours of larger label. This implies that it is the vertex $v_0$.
Finally, the fact that the contracted DEG is a tree implies that $\Phi(\qq)$ is a unicellular map. Indeed, the complement $\mathbb{S}\setminus\Phi(\qq)$ is obtained by gluing triangles, quadrangles, and digons along the tree-like structure of the contracted DEG, so it is clearly a contractible cell complex. It is thus homeomorphic to a disk, which shows that $\Phi(\qq)$ is a unicellular map.

\begin{remark}
At the suggestion of a referee, we add the following remark.
In our discussion of the orientable case, we have focused on the forward Marcus-Schaeffer bijection (from quadrangulations to one-face maps), on which it is clear that the rules (\cref{fig:basicFaceRules}(b)) depend crucially on the orientation of the surface. The reader familiar with the subject may argue that there is no such obvious objection for the \emph{backwards} Marcus-Schaeffer construction (from one-face maps to quadrangulations) to be applied on a non-orientable surface $\mathbb{S}$. Recall that, in this construction, one makes the tour around a well-labelled unicellular map on $\mathbb{S}$ and links each corner of label $i$ to the last visited corner of label $i-1$ (or to a central vertex $v_0$ if $i=1$), thus reconstructing the quadrangulation. This procedure can, indeed, be applied if the surface $\mathbb{S}$ is not orientable. However, it is easy to see that it does not always construct a quadrangulation: at the end of the construction, a \emph{twisted} edge of the original one-face map whose endpoints have labels of the form $(i,i+1)$ lies either inside a digon (if the endpoint of label $i$ is visited before the other one) or an hexagon (surrounded by corners of labels $(i,i-1,i,i+1,i,i-1)$).
\end{remark}

\bigskip
To conclude this digression on the orientable case, we emphasize once again that the Marcus-Schaeffer rule in case (b) relies on the existence of a global orientation. This orientation also induces the orientation of the edges of the DEG. In the general case, we will need to replace this global orientation by a \emph{local one} that will be constructed recursively.  To do that, we will in fact construct \emph{first} the DEG, and only in the end construct the associated unicellular map.

The DEG will be constructed by first drawing the cycle around the vertex $v_0$, and then adding recursively oriented edges crossing edges of $\qq$ with label $(i,i+1)$, by increasing value of $i$.
The rules will be chosen carefully in a way that at the end, the configuration of oriented edges in each face of $\qq$ looks locally as in \cref{fig:basicFaceRulesDEG} (and that in the orientable case we recover DEG used by Marcus and Schaeffer). Moreover, by construction, our DEG will be a forest of trees attached to a unique cycle around $v_0$, which will guarantee that the corresponding graph
 obtained by inserting new edges as in \cref{fig:basicFaceRulesDEG} is a unicellular map. We are now going to describe this construction in detail.

\begin{quote}
{\it We end here the digression on the orientable case: from now on, surfaces can be orientable or not.}
\end{quote}

\subsection{From quadrangulations to well-labeled unicellular maps}
\label{sec:DEG}

\newcommand{\DEG}{\nabla}

\subsubsection{Constructing the dual exploration graph $\DEG(\qq)$}
\label{subsubsect:SpanningTreeInSteps}

Let $\qq$ be a rooted quadrangulation on a surface, and let $v_0$ be the root vertex of $\qq$. In this subsection we describe how to draw a directed graph $\DEG(\qq)$  on the same surface.
To distinguish vertices, edges, etc.~of the quadrangulation and of the new graph we will say that the vertices, edges, etc.~of the graph $\DEG(\qq)$ are \emph{blue}, while those of $\qq$ are \emph{black}.

\smallskip

\begin{enumerate}[label=$\bullet$, ref=Step 0--\alph*, leftmargin=0cm]
\item \label{Step0a} \textbf{Step 0--a (Labeling).}
We label the vertices of $\qq$ according to their distance from the root vertex $v_0$. Note that this also induces a labeling of the corners (the label of a corner is the label of the unique vertex it is incident to). Recall that since $\qq$ is bipartite the extremities of each edge have labels that differ by one, so that faces are either of type $(i - 1, i, i - 1, i)$ or of type $(i - 1, i, i + 1, i)$ for $i \geq 1$, where the \emph{type} of a face is the sequence of its corner labels. An edge of $\qq$ whose extremities are labeled by $i$ and $i+1$ is said to be of \emph{label $i$}.

\medskip
Our goal is to draw a blue graph in such a way that at the end of the construction, each edge of the quadrangulation $\qq$ is crossed by exactly one blue edge. We introduce some terminology: edges of $\qq$ that are not crossed by a blue edge are called \emph{free}, and the \emph{label} of a blue edge is the label of the unique edge of $\qq$ it crosses.
\medskip 

We are going to construct the blue edges (and thus the graph $\DEG(\qq)$) by increasing label. We start by drawing edges of label $0$:
\item \label{Step0b} \textbf{Step 0--b (Initialization).}
We add a new blue vertex in each corner labeled by $0$ and we connect them by a cycle of blue edges around the 
root vertex $v_0$ as on \cref{fig:step0}. 
We orient this cycle in such a way that it is oriented from the root half-edge to the root corner. There is a unique vertex of the blue graph lying in the root corner of $\qq$, and this vertex has a unique corner that is separated from $v_0$ by the blue cycle. We declare that corner to be the \emph{last visited corner (LVC)} of the construction and we equip it with the orientation inherited from the one of the cycle (the LVC will be dynamically updated in the sequel). We set $i:=1$ and we continue.
\begin{figure}[h!!!!]
\begin{center}
\includegraphics[height=29mm]{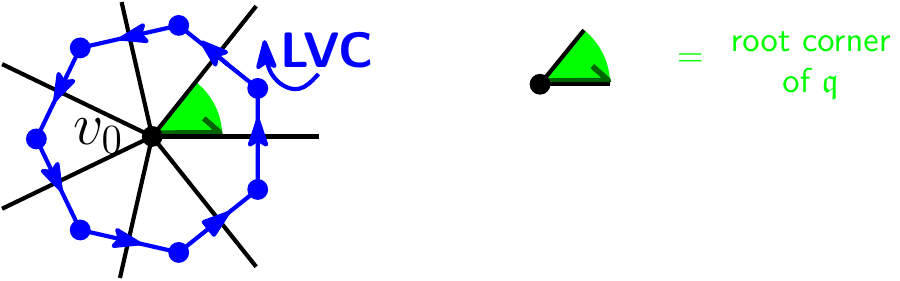}
\end{center}
\caption{\ref{Step0b}}\label{fig:step0}
\end{figure}
\end{enumerate}

\noindent We now proceed with the inductive part of the construction. 

\begin{enumerate}[label=$\bullet$, ref=Step \arabic*, leftmargin=0cm]
\item \label{Step1} \textbf{Step 1 (Choosing where to start).}
If there are no more free edges in $\qq$, we stop. Otherwise, we perform the tour of the blue graph, starting from the LVC. 
We stop as soon as we visit a face $F$ of $\qq$ having the following properties: $F$ is of type $(i-1,i,i+1,i)$, and $F$ has exactly one blue vertex already placed inside it. \cref{prop:welldef} below ensures that such a face always exists.

If the face $F$ is incident to only one free edge, we let $e$ be that edge.
If not, let $u$ be the blue vertex already contained in the face $F$. \cref{lemma:invariant} below ensures that $u$ is incident to two blue edges of label $(i-1)$, one incoming and one outgoing. We use these two oriented edges to define an orientation of $F$ by saying that they turn counterclockwise around the corner of label $(i-1)$. We then let $e$ be the first edge of label $i$ encountered clockwise around $F$ after that corner.
One can sum up the choice of the edge $e$ with \cref{fig:chooseFe} below that, by \cref{lemma:invariant} below, covers all the possible cases:
\begin{figure}[h!!!!]
\begin{center}
\includegraphics[height=30mm]{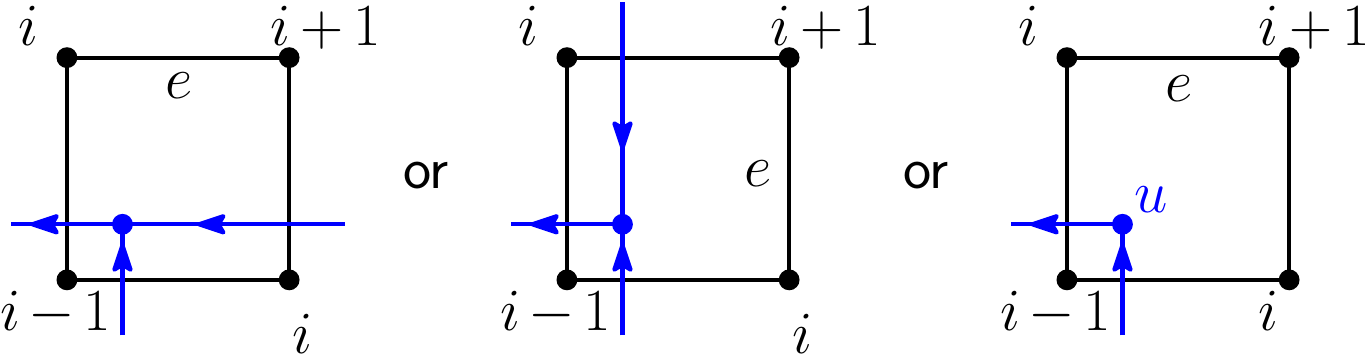}
\end{center}
\caption{The face $F$ and the edge $e$ selected by \ref{Step1}.}
\label{fig:chooseFe}
\end{figure}

\item \label{Step2} \textbf{Step 2 (Attaching a new branch of blue edges labeled by $i$ starting across $e$).}
We 
draw a new blue vertex $v$ in the unique corner of $F$ delimited by $e$ and its neighboring edge of label $i-1$, and we let $a$ be the vertex of $\qq$ incident to this corner.
 We now start drawing a path of directed blue edges starting from $v$ as follows: we cross $e$ with a blue edge leaving the face $F$, thus entering a face $F'$. If $F'$ contains a corner of label $(i-1)$, we attach the new blue edge to the blue vertex lying in that corner (\cref{lemma:invariant} ensures that this blue vertex exists\footnote{indeed, since face $F'$ is of type $(i-1,i,i+1,i)$ and edges of $\qq$ labeled by $i-1$ were already crossed, this means that just before crossing the edge $e$ the face $F'$ was of kind \ref{Bb}, \ref{Bc}, or \ref{Bd} from \cref{lemma:invariant}(B) -- in fact only \ref{Bb} or \ref{Bd} are possible although we do not need this observation here.}).
 If not, then we continue recursively drawing a path of new blue vertices and new blue edges turning around $a$, as on \cref{fig:step2} below, until we reach a face containing a corner of label $i-1$, and we finish by attaching the path to the blue vertex lying in that corner. We define the LVC as the corner lying to the right of the last directed blue edge we have drawn, in the local orientation defined by the fact that the path just drawn turns counterclockwise around $a$, see \cref{fig:step2}.
\begin{figure}[h!!!!]
\begin{center}
\includegraphics[height=55mm]{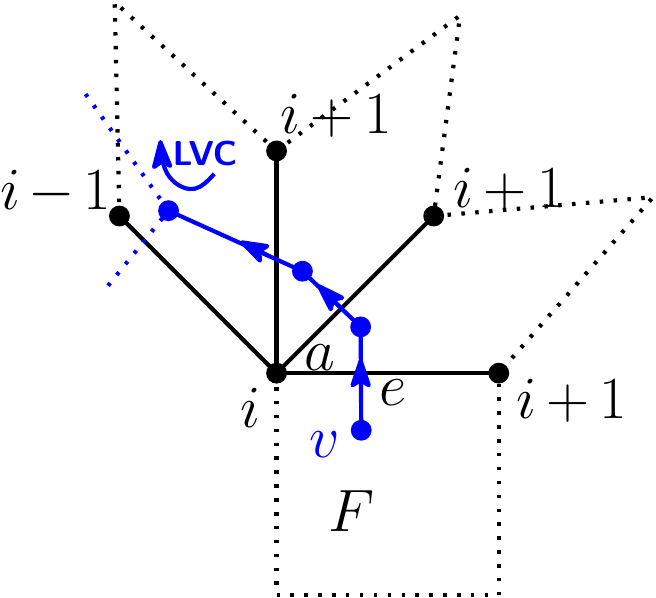}
\end{center}
\caption{\ref{Step2}}\label{fig:step2}
\end{figure}

\item \label{Step3} \textbf{Step 3 (Induction).}
If there are  no more free edges of label $i$ in $\qq$, we set $i:=i+1$, otherwise we let $i$ unchanged. We then go back to \ref{Step1} and continue.
\end{enumerate}

\begin{enumerate}[label=$\bullet$, ref=Termination, leftmargin=0cm]
\item \label{Termination} \textbf{Termination.}
We let $\DEG(\qq)$ be the blue embedded graph on $\mathbb{S}$ obtained at the end of the construction.
\end{enumerate} 
See \cref{fig:exampleDEG} page \pageref{fig:exampleDEG} for a full example of the construction on the Klein bottle.

\begin{remark}
Note that if the quadrangulation $\qq$ is orientable then in all steps of the construction we draw each oriented edge of the DEG in a way that when it crosses an edge of $\qq$ with label $(i,i+1)$, it has the vertex with smaller label on its left. Thus we recover the DEG defined by Marcus and Schaeffer in their original construction for the orientable quadrangulations (see \cref{subsect:OrientableCase}).
\end{remark}

\begin{proposition}\label{prop:welldef} The construction of $\DEG(\qq)$ is well-defined.
\end{proposition}

In order to prove~\cref{prop:welldef}, we first establish some invariants of the construction:
\begin{lemma}\label{lemma:invariant} The following properties remain true during the construction:
\begin{enumerate}[label=(\Alph*)]
\item \label{A} The graph $\DEG(\qq)$ is formed by a directed cycle on which trees are attached. Those trees are oriented towards the cycle.
\item \label{B} For $i\geq 1$, each face $F$ of $\qq$ of type $(i-1,i,i+1,i)$ can be of five kinds:
\begin{enumerate}[label=(\alph*)]
\item \label{Ba} there is no blue vertex inside $F$;
\item \label{Bb} there is one blue vertex of degree $2$ inside $F$; 
\item \label{Bc} there is one blue vertex of degree $3$ inside $F$;
\item \label{Bd} there are two blue vertices in $F$, of respective degrees $1$ and $2$; 
\item \label{Be} there are two blue vertices in $F$, of respective degrees $1$ and $3$; 
\end{enumerate}
Moreover, in each case, the blue vertices and edges inside $F$ are placed and oriented as on one of the following figures (up to reflection):

\vspace{1mm}
\includegraphics[width=\linewidth]{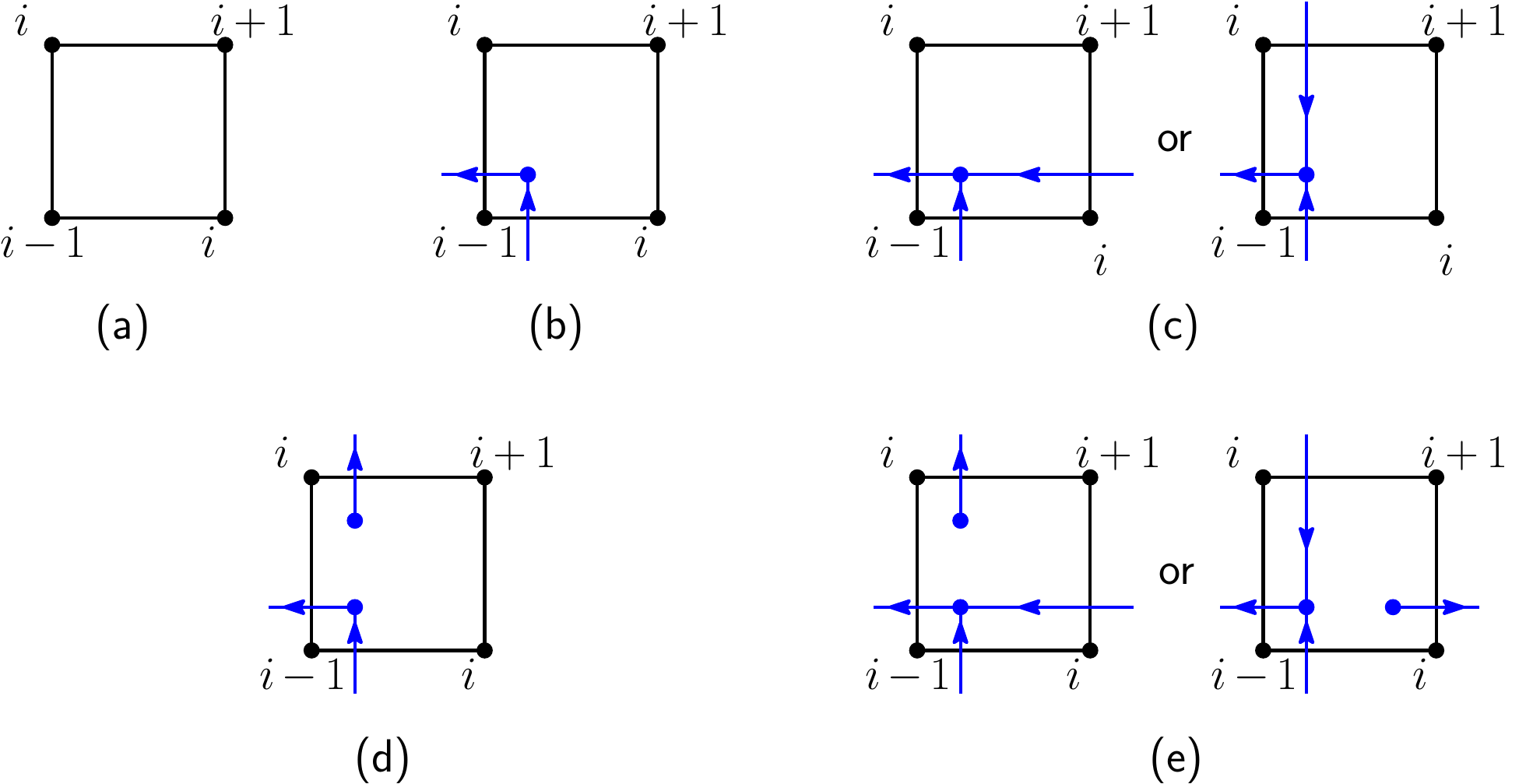}
\end{enumerate}
\end{lemma}

\begin{proof}
The blue graph is constructed by attaching recursively oriented blue directed paths of edges to the existing graph. Since after \ref{Step0b} the blue graph is a directed cycle, Property \ref{A} is clear by induction.

Property \ref{B} is also ensured by induction. First, it is true  after \ref{Step0b} since at this point all faces of type $(0,1,2,1)$ are of kind \ref{Bb}. We now check that the  property is preserved each time we add a new blue edge. There are two ways of adding a new blue edge to a face $F$ of type $(i-1,i,i+1,i)$. The first one is to start drawing a new path of edges in \ref{Step2}, starting from a new vertex of degree $1$ in $F$. In this case, by requirement of the algorithm, $F$ has to be of kind \ref{Bb} or \ref{Bc}, and it becomes of kind \ref{Bd} or \ref{Be} after drawing the blue edge.
The other way is to visit $F$ while drawing a path of blue edges turning around a vertex in  \ref{Step2}. In this case, if $F$ is of kind \ref{Ba}, we continue turning around the vertex of label $i-1$, adding a blue vertex of degree $2$, and $F$ becomes of kind \ref{Bb}. If $F$ is of kind \ref{Bb} or \ref{Bd}, by construction we attach the current path of edges to the blue vertex of degree $2$ in $F$, and $F$ becomes of kind \ref{Bc} or \ref{Be}, respectively.

It remains to check that when we perform \ref{Step2} of the construction, we never enter a face of kind~\ref{Bc}. This is the case since, by construction of the algorithm, each time we create a face $F$ of kind~\ref{Bc}, we immediately go back to  \ref{Step1}, 
and, since $F$ contains the LVC,
we add a new vertex of degree $1$ inside $F$, thus transforming it into a face of kind~\ref{Be}

\end{proof}

\begin{proof}[Proof of \cref{prop:welldef}]
~\\
\noindent $\bullet$ We first prove that during \ref{Step1} of the construction, we always succeed in finding a face with the wanted properties.

First, we claim that if there are still free edges of label $i$ in $\qq$, 
then at least one of them is incident to a face of minimum label $i-1$. Indeed, assume the contrary. Let $e$ be a free edge of $\qq$ of label $i$ and let $a$ be the endpoint of label $i$ of $e$. Let $e=e_0, e_1, \dots, e_r$ (with $r\geq 1$) be the sequence of edges encountered clockwise around $a$ (in some conventional orientation) starting from $e$, where $e_r$ is the first edge clockwise after $e$  having label $i-1$, as on the leftmost picture below:

\begin{center}
\includegraphics[width=\linewidth]{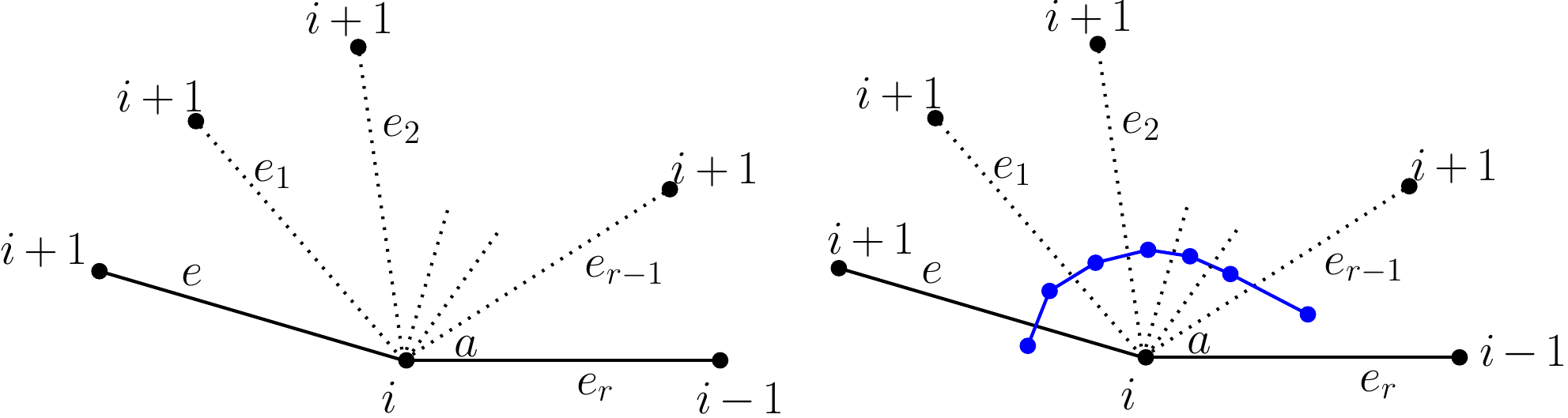}
\end{center}

\noindent Note that $e_r$ exists since $a$ must be incident to a vertex of label $i-1$ (since there is a geodesic going from $a$ to $v_0$). The edge $e_{r-1}$ has label $i$ and is incident to a face of minimum label $i-1$, so by the assumption we have made, $e_{r-1}$ is crossed by a blue edge. But according to the construction rules, this blue edge is part of a path of edges labeled by $i$ turning around the vertex $a$, and originating in a face containing a corner of label $(i-1)$. This path must cross the edge $e$ (although we do not know in which direction), as on the rightmost picture above. This is a contradiction and proves the claim.

\smallskip
We now claim that, if there are no more free edges of label less than $i$ in $\qq$ but if there are still free edges of label $i$, then there is at least one free edge of label $i$ that is incident to a face of type $(i-1,i,i+1,i)$ containing exactly one blue vertex, i.e., a face of kind \ref{Bb} or \ref{Bc} with the notation of \cref{lemma:invariant}. 
The integer $i\geq 1$ being fixed, consider the set $\mathcal{F}$ of all faces of $\qq$ having type $(i-1,i,i+1,i)$. Each time we go through Step 2 of the construction, we add one blue vertex of degree~$1$ in one of the faces in $\mathcal{F}$, and we increase the degree of a vertex of degree $2$ by attaching the current path to it. Therefore, among all blue vertices lying inside the faces of $\mathcal{F}$, there are as many vertices of degree $1$ as of degree $3$. This implies that, after performing Step 2, there are as many faces of kind \ref{Bc} as of kind \ref{Bd} in~$\mathcal{F}$. Therefore it is not possible that all faces in $\mathcal{F}$ are of kind \ref{Bd}, and since there are free edges of label $i$ remaining, we know by the previous paragraph that they are not all of kind \ref{Be}. Moreover, since by assumption there are no more free edges of label $(i-1)$, there are no faces of kind \ref{Ba} in $\mathcal{F}$, so there is at least one face in $\mathcal{F}$ that is of kind \ref{Bb} or \ref{Bc}. 

This is enough to conclude that \ref{Step1} of the algorithm always succeeds, since a face of kind \ref{Bb} or \ref{Bc} is necessarily visited when one performs the tour of the blue graph (since the blue graph intersects such faces).

\medskip

\noindent $\bullet$ To complete the proof, 
it is enough to check that all the assumptions made on the position of blue vertices and edges encountered in the faces visited during the choice of the edge $e$ in \ref{Step1} and during \ref{Step2} are valid. This is ensured by \cref{lemma:invariant}.
\end{proof}

\begin{figure}
\includegraphics[width=\linewidth]{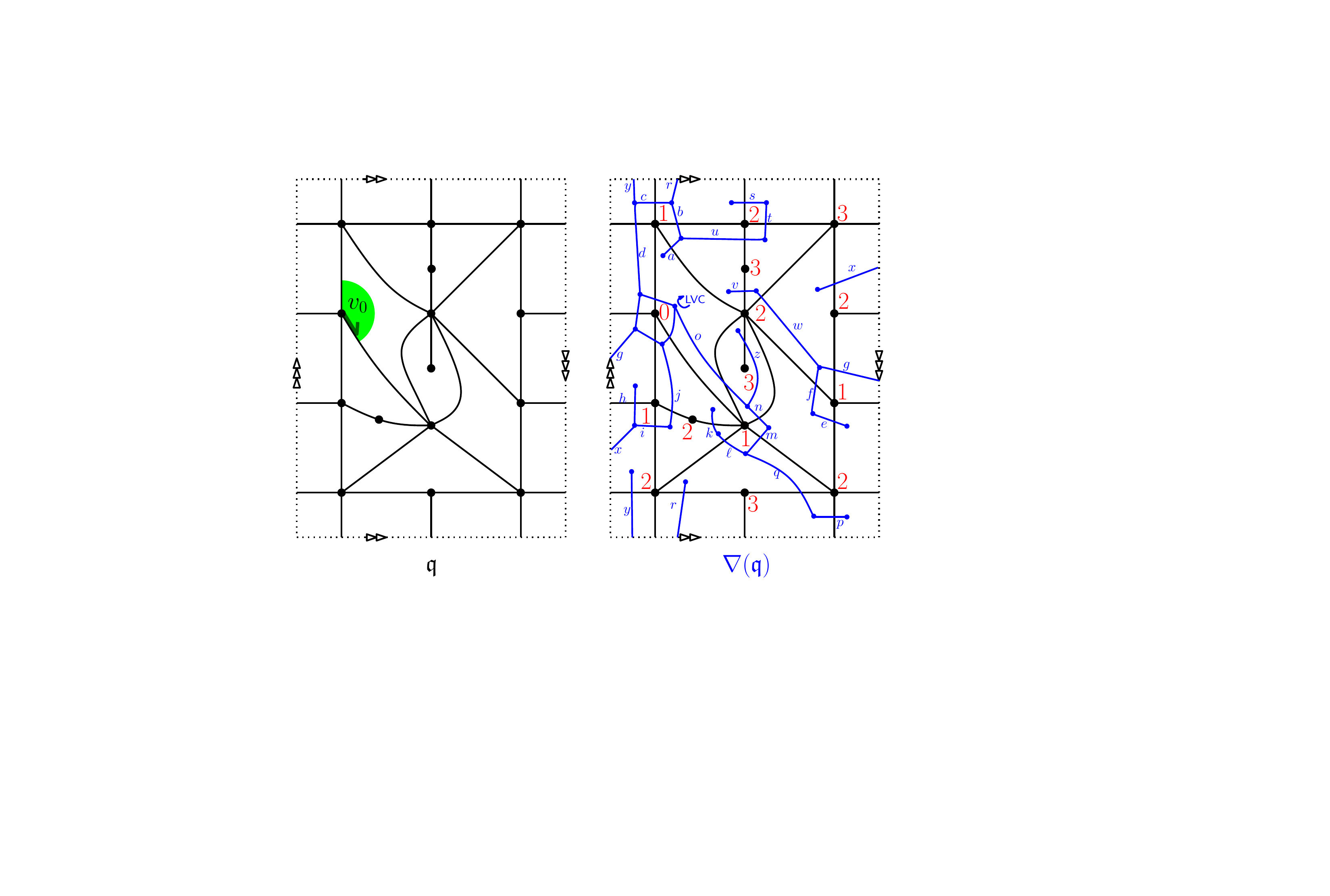}
\caption{Left: A rooted quadrangulation $\qq$ on the Klein bottle $\NN_1$ (on this picture the left side of the dotted rectangle should be glued to the right side, as well as bottom to top, as indicated by arrows). Right: The associated dual exploration graph $\DEG(\qq)$ is depicted in blue. To help the reader visualize the construction, we have numbered the blue edges appearing after \ref{Step0b} by $a,b,\dots, z$ in their order of appearance in the construction. We have not indicated the orientation of these edges, to make the picture lighter.}
\label{fig:exampleDEG}
\end{figure}

\begin{remark}
There are several places in the construction of $\DEG(\qq)$ where the rules we chose may seem arbitrary, and indeed other choices could have been made, leading to a different but equally valid construction.
For example, in \ref{Step1}, although it is crucial to choose for $F$ a face of type $(i-1,i,i+1,i)$ containing exactly one blue vertex, there is no particular reason to pick the first encountered such face around the blue graph starting from the LVC. Any other canonical choice that uses only the one-neighborhood of the already constructed blue graph would have been equally efficient, provided it satisfies the property needed to make the last argument in the proof of \cref{lemma:invariant} work (namely that after creating a face of kind \ref{Bc} we immediately reselect that face for the next round of the construction). We made this choice because it is easily described.
Similarly, in \ref{Step1}, the way we chose to break the tie in the case where $F$ contains two free edges could be replaced by some other convention. 
In the same vein, the initial choice of the LVC in \ref{Step0b} may influence the construction, and the way we chose it (which is arguably the most natural one) could be replaced by some other convention.
 However, now that these conventional choices have been made, we will work with them in the sequel, and in particular we will make consistent choices in the construction of the inverse bijection.
\end{remark}

\newcommand{\bij}{\Phi}
\subsubsection{Constructing the labeled unicellular map $\bij(\qq)$.}
Let as before $\qq$ be a rooted quadrangulation on a surface $\mathbb{S}$, and $v_0$ be its root vertex. We construct the dual exploration graph $\DEG(\qq)$ as in the previous subsection, and as before we call its edges \emph{blue}. By construction, each edge of $\qq$ is crossed by exactly one blue edge of $\DEG(\qq)$, and therefore \cref{lemma:invariant} implies that at the end of the construction, all the faces of $\qq$ are of one of the following types for $i\geq 1$, up to reflection (we forget orientation of blue edges in this picture as there are irrelevant for our purpose):

\begin{center}
\includegraphics[height=25mm]{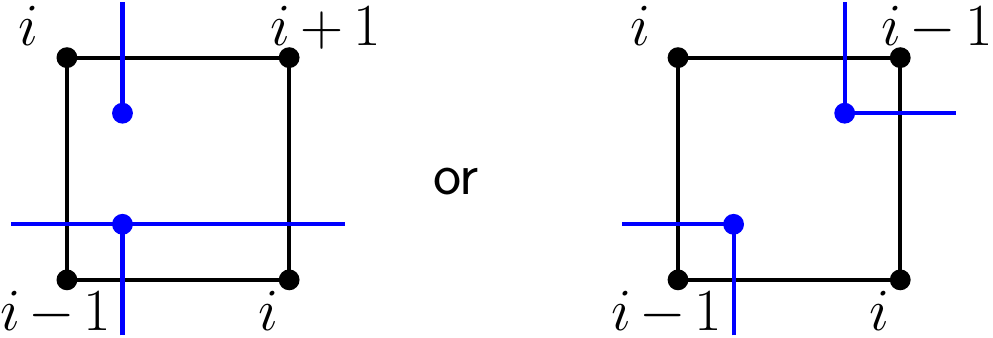}
\end{center}
\noindent We now add one red edge in each face of $\qq$ according to the rule of~\cref{fig:redBlueRule}.
\begin{figure}[h!!!!]
\includegraphics[height=25mm]{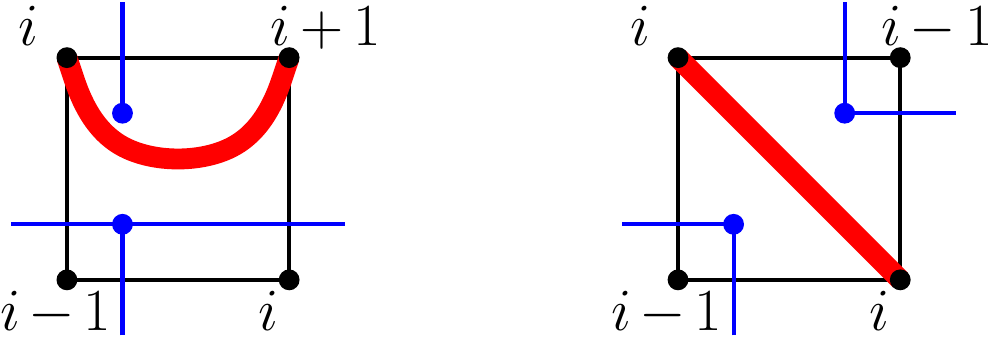}
\caption{Constructing the unicellular map $\bij(\qq)$ from the DEG $\DEG(\qq)$.}
\label{fig:redBlueRule}
\end{figure}

\noindent We let $\bij(\qq)$ be the map on $\mathbb{S}$ consisting of all the red edges, and of all the vertices of $\qq$ except $v_0$.
We declare the root corner of $\bij(\qq)$ to be the unique corner of label~$1$ of $\bij(\qq)$ incident to the root edge of $\qq$, and we equip it with the local orientation inherited from the one of the root corner of $\qq$ along its root edge.
 \cref{fig:exampleDEG2} gives an example of the construction.
\begin{figure}
\includegraphics[width=\linewidth]{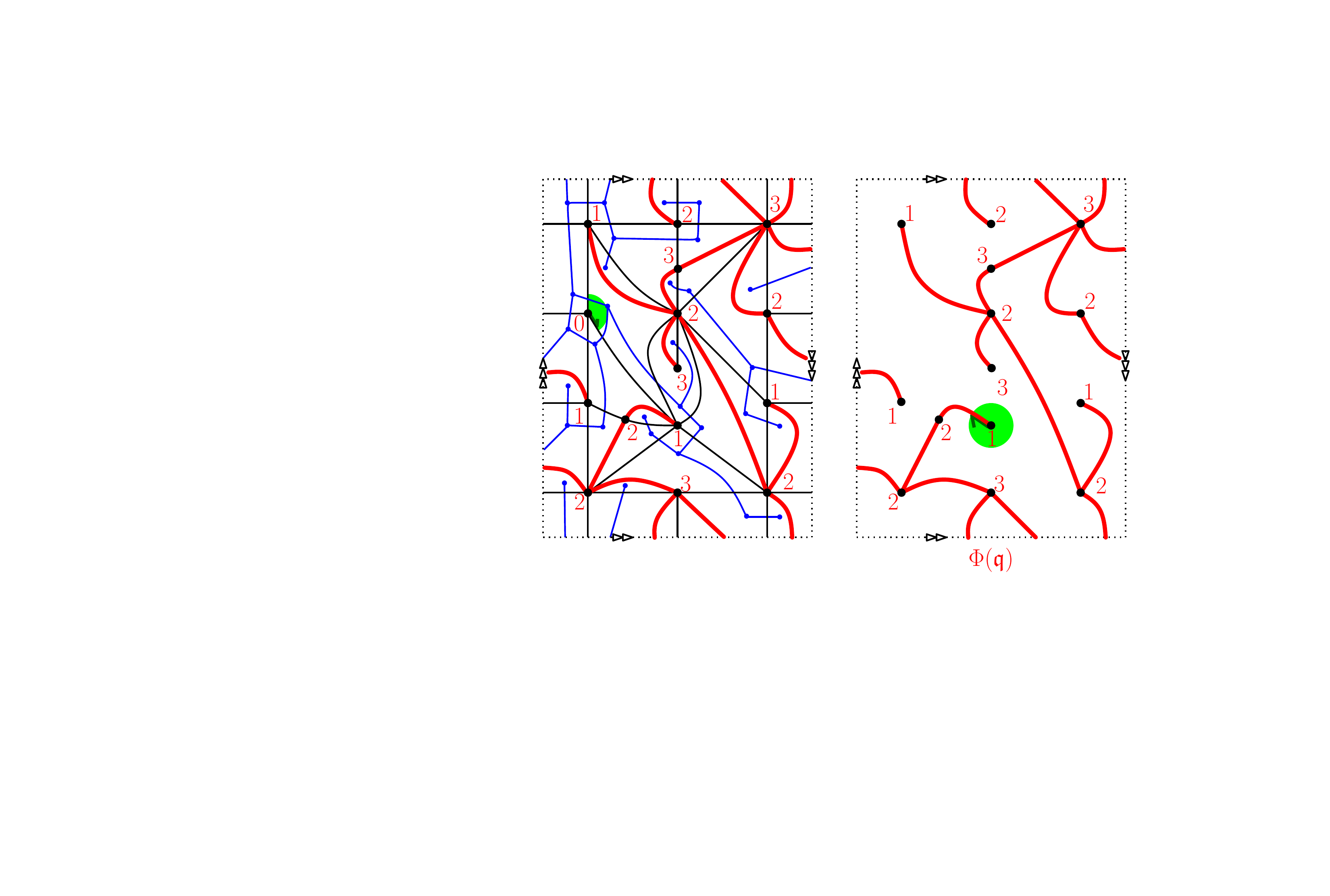}
\caption{Construction (red, fat, edges) of the labeled unicellular map $\bij(\qq)$ associated with the bipartite quadrangulation of \cref{fig:exampleDEG}.}
\label{fig:exampleDEG2}
\end{figure}

\begin{lemma}
$\bij(\qq)$ is a well-defined well-labeled unicellular map.
\end{lemma}
\begin{proof}
We first notice that $H=\mathbb{S}\setminus \bij(\qq)$ is homeomorphic to a disk. Indeed, the dual exploration graph $\DEG(\qq)$ goes through any face of $H$. From this it is easy to see that any loop in $H$ can be retracted to a point along the tree-like structure of $\DEG(\qq)$. This proves both that $\bij(\qq)$ is connected, that it is a valid map on $\mathbb{S}$, and that it has only one face.

It remains to check that the labeling of vertices makes $\bij(\qq)$ a valid well-labeled unicellular map. First, it is clear by construction that the label of the root vertex is~$1$, and that all labels in $\bij(\qq)$ are at least~$1$. Moreover, by construction, any two vertices of $\bij(\qq)$ that are linked by an edge have labels differing by $\pm 1$ or $0$, so there is nothing more to prove.
\end{proof}

\subsection{From well-labeled unicellular maps to quadrangulations}
\label{sec:reversebij}

\renewcommand{\M}{\mathfrak{u}}

We now describe the reverse construction. We let $\M$ be a rooted well-labeled unicellular map with $n$ edges on a surface $\mathbb{S}$. From $\M$ we are going to reconstruct simultaneously two graphs: the quadrangulation $\qq$ associated with $\M$, and the dual exploration graph $\DEG(\qq)$. In order to distinguish between those three graphs we will refer to edges of the unicellular map, of the quadrangulation, and of the dual exploration graph as \emph{red}, \emph{black}, and \emph{blue}, respectively. Vertices of the dual exploration graph will also be referred to as \emph{blue}. Black edges will also be referred to as \emph{internal edges}.

The construction goes in several steps that reproduce the steps of the construction of $\DEG(\qq)$. To facilitate this analogy we use a consistent numbering, for example \ref{StepR0} here is the analog of \ref{Step0a} and \ref{Step0b} in \cref{sec:DEG} (the letter ``R'' is for ``reverse'').

\smallskip

\newcommand{\rplan}{\Psi}
\newcommand{\rDEG}{\Delta}
\newcommand{\rbij}{\Lambda}

\begin{enumerate}[label=$\bullet$, ref=Step R0, leftmargin=0cm]
\item \label{StepR0} \textbf{Step R0 (Initialization)}.
Let us consider the representation of $\M$ as a labeled $2n$-gon with identified pairs of edges as in \cref{subsec:RepresentationOfMap} (we denote by $\pp$ the associated $2n$-gon and by $E_s(\M)$ and $E_t(\M)$ the associated sets of 
pairs). We draw a vertex $w_0$ labeled by $0$ inside our polygon $\pp$ and we connect each corner of $\pp$ labeled by $1$ to $w_0$ by a new (black) edge. In this way we dissect the polygon into $k_1$ areas, where $k_1$ is the number of corners of $\M$ labeled by $1$. In each such area we now draw one blue vertex and we connect them by a directed (counterclockwise) blue loop encircling $w_0$. There is a unique internal edge incident to the root corner of $\pp$, and a unique blue edge crossing it. We declare the blue corner lying at the extremity of that edge (and exterior to the cycle) to be the \emph{last visited corner (LVC)} of the construction. We set $i:=1$ and we continue.
\begin{figure}[h!!!!]
\begin{center}
\includegraphics[height=55mm]{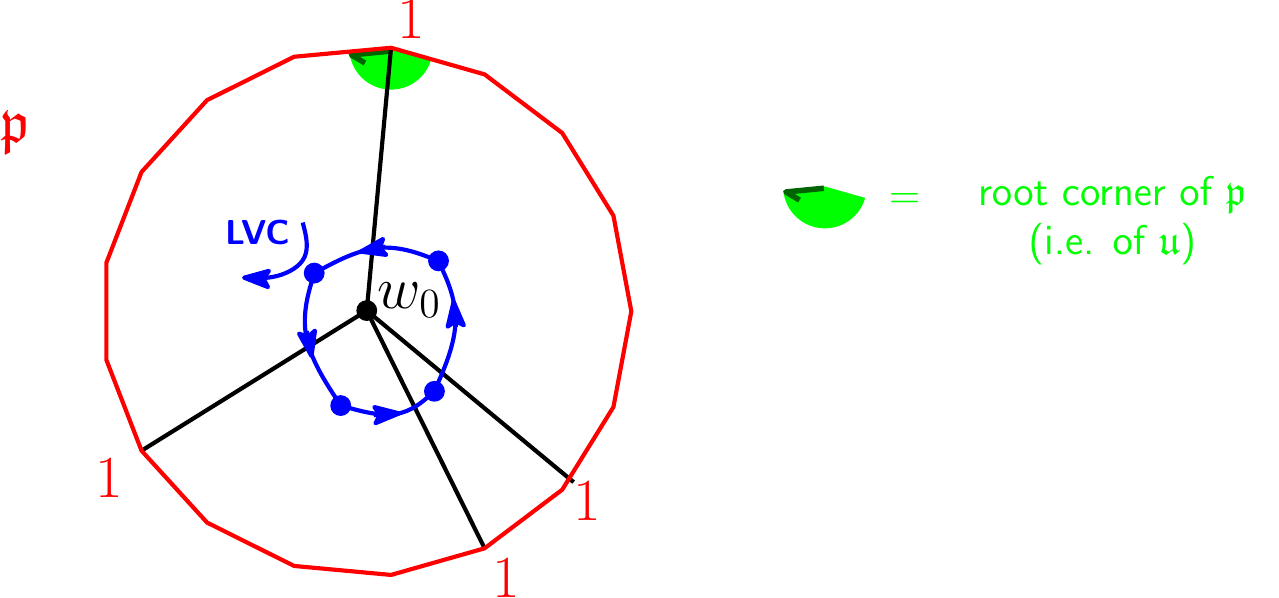}
\end{center}
\caption{Step R0}\label{fig:stepR0}
\end{figure}

\smallskip

We now proceed with the inductive step of the construction.
We are going to construct simultaneously and recursively a planar graph $\rplan(\M)$ drawn in $\pp$ and a blue directed graph $\rDEG(\M)$ drawn in $\pp$. 
 The vertices of that planar graph will be $V(\pp) \cup w_0$ and its edges (called internal edges because they are lying in the interior of $\pp$) will have the property that each vertex $w \in V(\pp) \cup w_0$ labeled by $i$ will be connected to a unique vertex $w' \in V(\pp) \cup w_0$ labeled by $i-1$.
Let us stress here the fact that $V(\pp)$ denotes the vertices of the polygon $\pp$ itself and that this notation does not take into account the subsequent identifications among those vertices given by the unicellular map structure of $\M$.

\smallskip

A property of our construction is that each time we draw an edge of $\rplan(\M)$, we draw at the same time a blue directed edge of $\rDEG(\M)$ passing through it.  As in \cref{sec:DEG} we define the \emph{label} of a blue edge as the minimum of labels of endpoints of the black edge it crosses.
The graph $\rplan(\M)$ dissects the polygon $\pp$ into several regions, and in the sequel we use the word \emph{area} (rather than face, or domain) to refer to these regions. Induction (\cref{prop:Rwelldef} below) ensures that in each area there is exactly one blue vertex $v$, which has exactly one outgoing blue edge, and the number  of the blue edges attached to $v$ is equal to the number of the corresponding internal edges lying on the boundary of that area, with multiplicity. Moreover, those areas have one of three possible types, see \cref{fig:TypesOfAreas}:
\begin{enumerate}[label=(T\arabic*)]
\item \label{T1}
the border of the area $A$ consists of exactly two edges: one edge $e$ with endpoints labeled by $i$ and $i-1$ belonging to the boundary of $\pp$ and one internal edge connecting the endpoints of $e$.  In that area there is exactly one outgoing blue edge labeled by $i-1$ and no incoming blue edges (see \cref{subfig:T1});
\item \label{T2}
the border of the area $A$ consists of a piece $x$ of the boundary of $\pp$ with endpoints labeled by $i$ (all the vertices between those endpoints, if any,
being labeled by integers strictly greater than $i$) and exactly two internal edges which connect the two endpoints of $x$ to some vertex labeled by $i-1$. In that area there are exactly one outgoing and one incoming blue edges labeled by $i-1$ that define an orientation of the area $A$ (see \cref{subfig:T2}); 
\item \label{T3}
the border of the area $A$ consists of exactly four edges: one edge $e$ with endpoints labeled by $i+1$ and $i$ belonging to the boundary of $\pp$ and three internal edges such that the border of $A$ is labeled by a cycle $(i+1, i, i-1, i)$.  In that area there are exactly two incoming blue edges labeled by $i$ and $i-1$ and one outgoing blue edge labeled by $i-1$ (see \cref{subfig:T3}).
\end{enumerate}

\begin{figure}[h!!!!!]
\centering
\subfloat[Type \ref{T1}]{
	\label{subfig:T1}
	\includegraphics[scale=0.46]{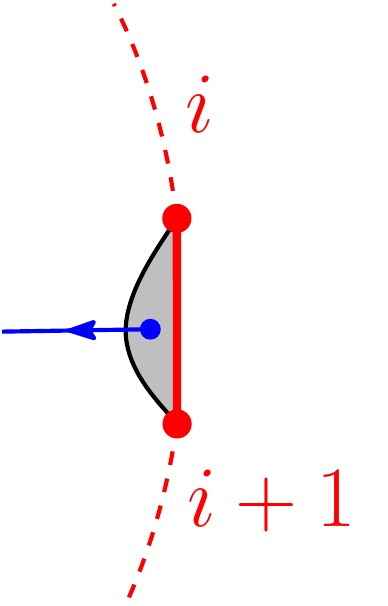}}
\quad
\subfloat[Type \ref{T2}]{
	\label{subfig:T2}
	\includegraphics[scale=0.46]{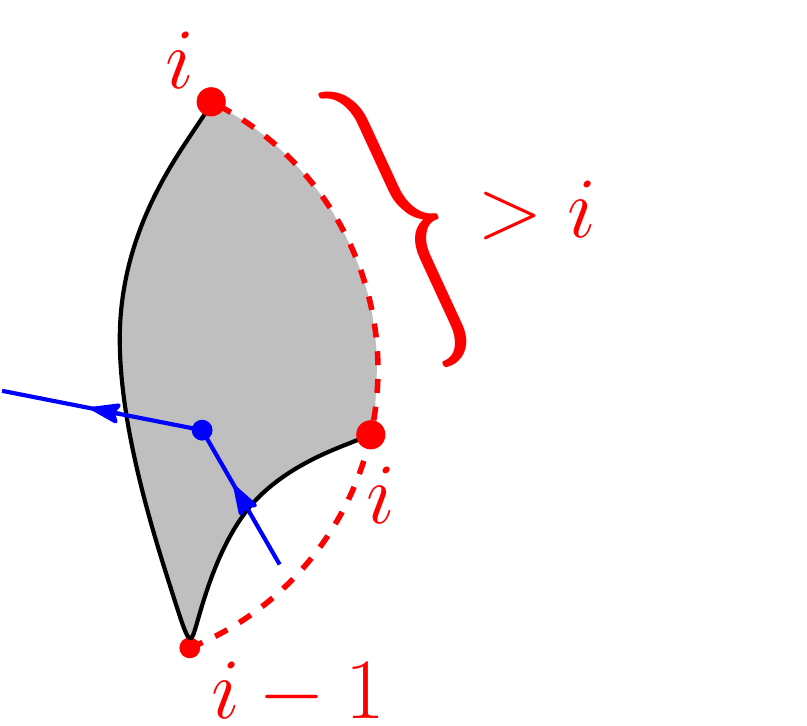}}
\quad
\subfloat[Type \ref{T3}]{
	\label{subfig:T3}
	\includegraphics[scale=0.46]{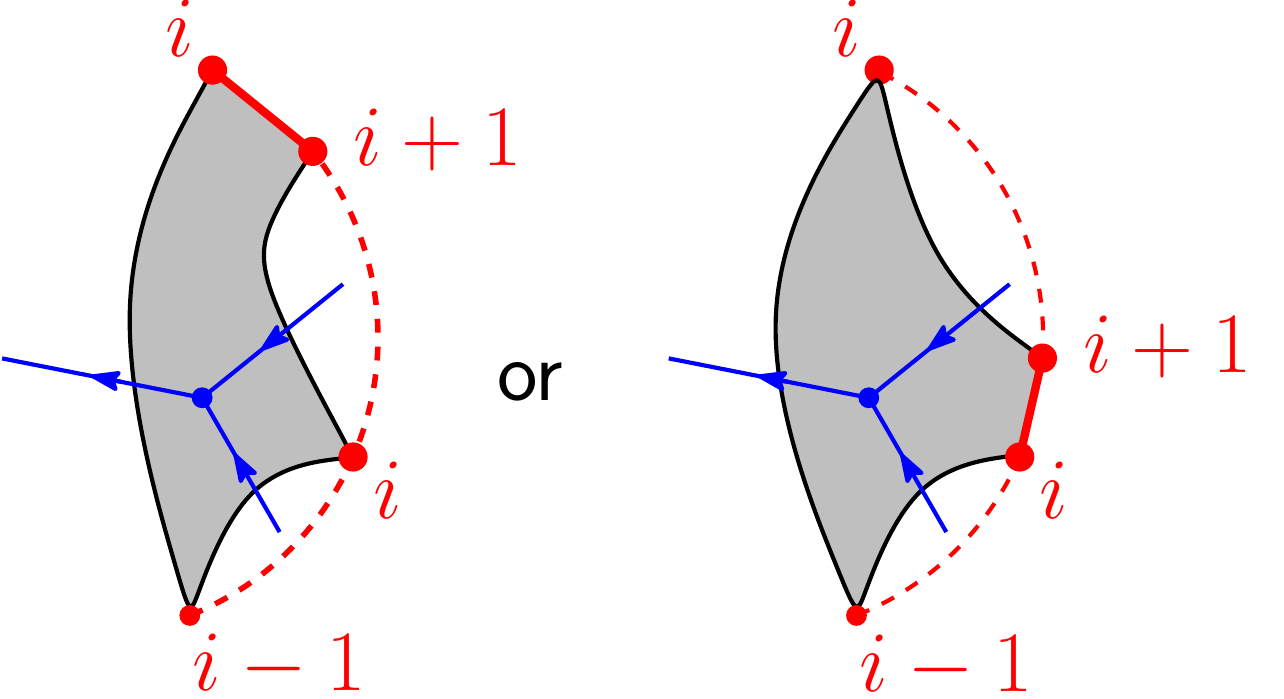}}
\caption{The three possible types of areas appearing during the construction of $\rplan(\M)$ and $\rDEG(\M)$ (up to reflection).}
\label{fig:TypesOfAreas}
\end{figure}
\end{enumerate}

\begin{enumerate}[label=$\bullet$, ref=Step R\arabic*, leftmargin=0cm]
\item \label{StepR1} \textbf{Step R1 (Choosing in which face to start, and through which edge)}.
If there are no vertices of label $i+1$ in $\pp$, go to the termination step.
Otherwise, walk along the blue graph, starting from the LVC, thus visiting some areas. We let $F$ be the first visited area having the following properties:
\begin{itemize}
\item[$\bullet$] The minimum corner label in $F$ is $i-1$.
\item[$\bullet$] $F$ is of type \ref{T2} or \ref{T3}.
\item[$\bullet$] If $F$ is of type \ref{T3}, let $e$ be the unique (red) edge of $\pp$ bordering $F$. If $F$ is of type \ref{T2}, let $e$ be the last (red) edge of $\pp$ bordering $F$ having extremities labeled $i,i+1$, in the counterclockwise orientation induced by the blue graph on $F$ (see \cref{fig:stepR1} below).
Let $\tilde{e}$ be the unique edge of $\pp$ that is matched with $e$ in the unicellular map structure inherited from $\M$, and let $\tilde{F}$ be the area containing $\tilde{e}$. Then $\tilde{F}$  is of type \ref{T2}.
\end{itemize}
\cref{prop:Rwelldef} below ensures that such an area $F$ always exists.

\begin{figure}[h!!!!!]
\begin{center}
\noindent\includegraphics[width=\linewidth]{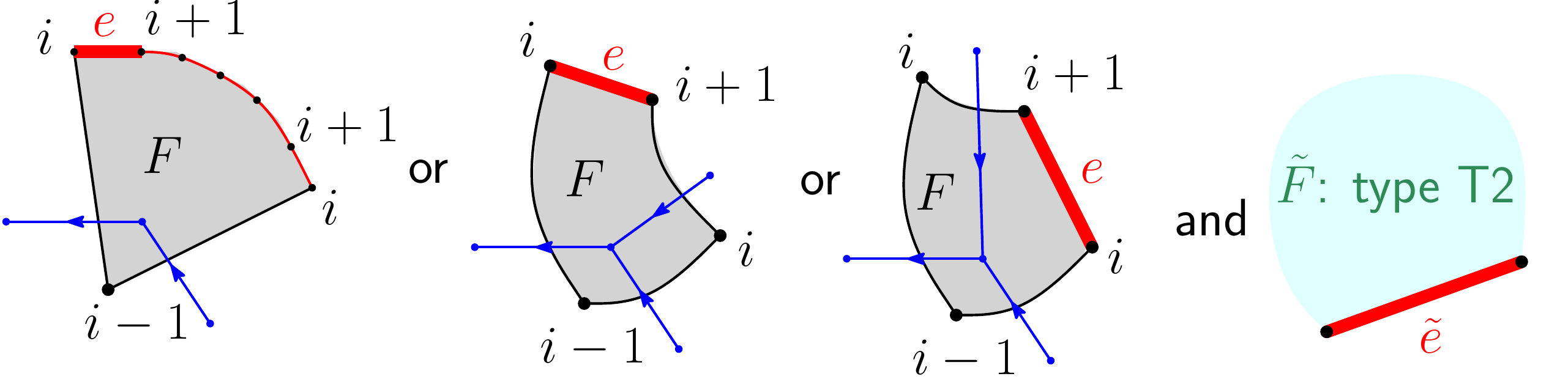}
\end{center}
\caption{\ref{StepR1}}\label{fig:stepR1}
\end{figure}

\item \label{StepR2} \textbf{Step R2 (Constructing a path of edges labeled by $i$)}. We let $F, e, \tilde{e}$ and $\tilde{F}$ be defined as above.
We let $v$ be the vertex of $\pp$ of label $i$ incident to $\tilde{e}$, and we link $v$ by new internal edges to all the corners of $\tilde{F}$ having label $i+1$, without crossing any existing blue edge. We thus subdivide $\tilde{F}$ into several new areas.
We let $f_1,f_2,\dots,f_k$ be these areas, starting from the one incident to $\tilde{e}$, turning around~$v$. 
Then $f_k$ contains a unique blue vertex of degree $2$, call it $v_k$. 
 We now add a new blue vertex $v_i$ in each area $f_i$ for $1\leq i \leq k-1$, and we connect the vertices $v_1,v_2,\dots,v_k$ by a blue directed path as on the following figure:
\begin{figure}[h!!!!!]
\begin{center}
\noindent\includegraphics[height=4cm]{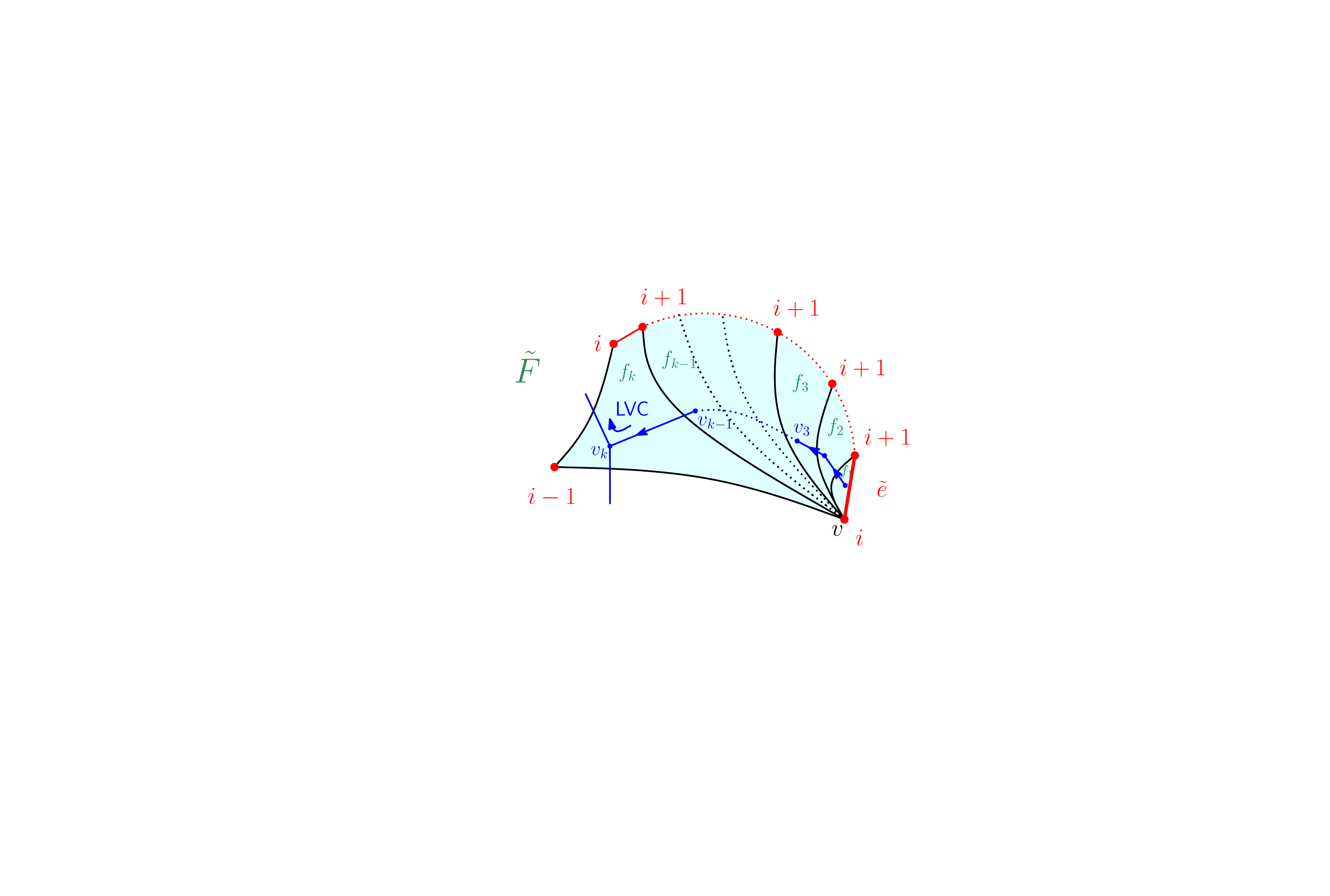}
\end{center}
\caption{\ref{StepR2}} \label{fig:stepR2}
\end{figure}

\noindent We declare the corner of $v_k$ incident to the last drawn blue edge and exterior to $v$ to be the last visited corner (LVC) as on~\cref{fig:stepR2}.

\item \label{StepR3} \textbf{Step R3 (Induction).} If each vertex of label $i+1$ of $\pp$ is linked to an internal edge, we set $i:=i+1$, otherwise we let $i$ unchanged. We then go back to \ref{StepR1} and continue.
\end{enumerate}

\begin{enumerate}[label=$\bullet$, ref=Termination, leftmargin=0cm]
\item \label{TerminationR} \textbf{Termination}. We perform the identifications of edges of $\pp$ according to the unicellular map structure of $\M$, thus reconstructing the surface $\mathbb{S}$. We call $\rbij(\M)$ the map on $\mathbb{S}$ consisting of all the internal edges, with vertex set $V(\M)\cup\{w_0\}$.
\end{enumerate}

\smallskip
\cref{fig:exampleReverse} gives an example of the construction.

\begin{figure}
\begin{center}\includegraphics[width=\linewidth]{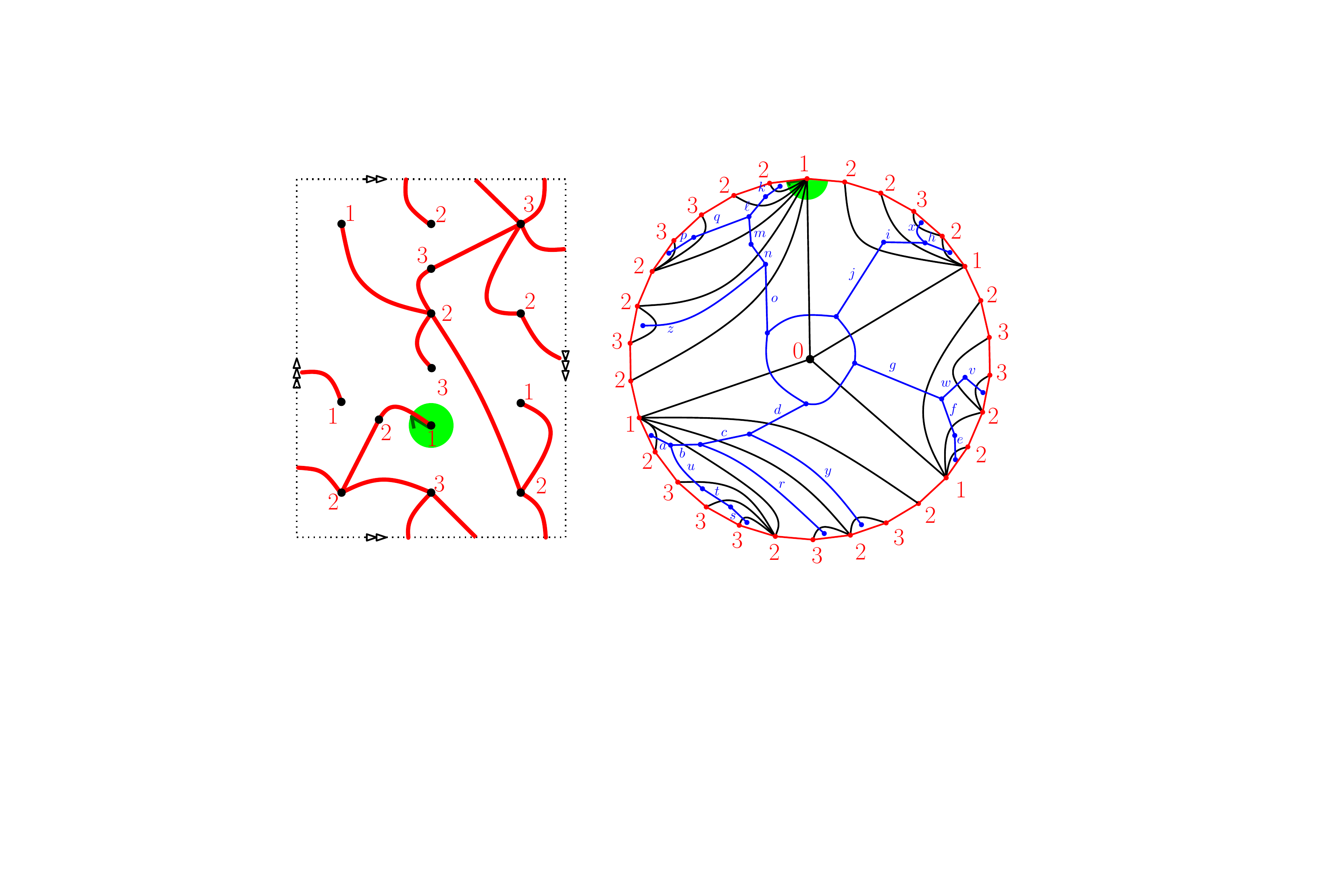}
\end{center}
\caption{Left: A labeled unicellular map $\M$ on the Klein bottle (the root corner is indicated in green). Right: The same unicellular map displayed as a polygon, and the construction of the associated quadrangulation (black edges). To help the reader visualize the construction, blue edges added after \ref{StepR0} have been numbered $a,b,\dots,z$ in their order of appearance in the construction. To make the picture lighter we have not drawn the orientation of the blue edges (they are all oriented towards the inner blue cycle, which is oriented counterclockwise).}
\label{fig:exampleReverse}
\end{figure}
\begin{proposition}\label{prop:Rwelldef}
The map $\rbij(\M)$ is a well-defined bipartite quadrangulation.
\end{proposition}

In order to prove~\cref{prop:Rwelldef}, we first establish some properties of the construction. First, we introduce some helpful terminology. If $F$ is an area of type \ref{T2} or \ref{T3}, we note $e=e(F)$ for the edge defined as in~\cref{fig:stepR1}. We let $\tilde{e}(F)$ be the edge matched with $e(F)$ in $\M$, and  we let $\tilde{F}$ be the area containing $\tilde{e}(F)$.
We say that $\tilde{F}$ is the \emph{coarea} of $F$ and that its type (\ref{T1},\ref{T2} or \ref{T3}) is the \emph{cotype} of $F$.

\begin{lemma}
\label{lemma:PropertiesOfR}
The following properties hold true:
\begin{enumerate}[ref=(\arabic*)]
\item
After any step of the reverse construction each area of type \ref{T3} has cotype \ref{T1} except, possibly, the last area created, and then the only possible cotypes are \ref{T1} and \ref{T2};
\label{lemma:PropertiesOfR:1}
\item
After any step of the reverse construction the number $N$ of areas of type \ref{T2} and cotype \ref{T1} is $0$ or $1$.
\label{lemma:PropertiesOfR:2}
\end{enumerate}
\end{lemma}

\begin{proof}
Both properties will be proved by induction.

\noindent \emph{(1)} In order to prove \cref{lemma:PropertiesOfR}--\ref{lemma:PropertiesOfR:1} we assume that we just constructed an area $F$ of type \ref{T3} in \ref{StepR2}, by subdividing a larger area $F'$ of type \ref{T2} into smaller areas (this is the only possible way for creating an area of type \ref{T3}), and we examine the next step of the construction. If the cotype of $F$ is \ref{T3} then its coarea $\tilde{F}$  was of cotype \ref{T2} in the previous step of the construction, and by induction hypothesis it was the last created area. This leads to a contradiction: indeed, the construction rules imply that in \ref{StepR2} following the creation of the area $\tilde{F}$, we subdivide  $F'$ in a way that the edge $\tilde{e}(e(\tilde{F})) = e(F)$ (we recall that $\tilde{e}(e(\tilde{F}))$ is by definition the edge matched with $e(\tilde{F})$ in $\M$, that $e(\tilde{F})$ is the edge defined as in~\cref{fig:stepR1} for the area $\tilde{F}$ of type \ref{T2} or \ref{T3}, and that $\tilde{F}$ is defined as the coarea of $F$; the fact that $\tilde{e}(e(\tilde{F})) = e(F)$ is a simple observation) belongs to an area of type \ref{T1}. This contradicts our assumption that $F$, which contains $e(F)$, is of type \ref{T3}. Therefore there remains two possible situations:
\begin{itemize}
\item The cotype of area $F$ is \ref{T1} (in which case \cref{lemma:PropertiesOfR}--\ref{lemma:PropertiesOfR:1} holds true);
\item The cotype of area $F$ is \ref{T2}. Then in the next step of the construction (\ref{StepR1}) we select the area $F$ and we perform \ref{StepR2} to subdivide $\tilde{F}$ into smaller areas in a way that $\tilde{e}(F)$ belongs to an area of type \ref{T1} (after the subdivision). After this step the cotype of $F$ is \ref{T1}, which finishes the proof of the inductive step.
\end{itemize}

\noindent \emph{(2)} In order to prove \cref{lemma:PropertiesOfR}--\ref{lemma:PropertiesOfR:2} we observe first that after \ref{StepR0} we have $N=0$ since there are no areas of type \ref{T1}. We now look how $N$ is modified in each round of the construction. Let us assume that we are going to create an internal edge labeled by $i$ by applying \ref{StepR2}. There are two possible situations:
\begin{itemize}
\item If the area $F$ selected in \ref{StepR1} is of type \ref{T3}, we continue in \ref{StepR2} by subdividing its coarea $\tilde{F}$ into one area of type \ref{T1} (that becomes the coarea of $F$), one area of type \ref{T3}, and possibly some new areas of type \ref{T2}. Since the smallest corner label inside these new areas of type \ref{T2} is $i$, their cotypes are different from \ref{T1}, because no internal edge labeled by $i+1$ exists yet. Therefore $N$ stays constant during this process.
\item If the area $F$ selected in \ref{StepR1} is of type \ref{T2}, there are two possibilities. Either we have not drawn yet any internal edge of label $i$, in which case we have $N=0$ by construction, or we have already drawn such an edge, in which case the last created area in the execution of the algorithm (\ref{StepR2}) was an area $G$ of type \ref{T3} and minimum label $i-1$. Since $G \neq F$ this means that $G$ was not selected in the \ref{StepR1} following its creation, and by the rules of \ref{StepR1} this means that its cotype is not \ref{T2}. By \cref{lemma:PropertiesOfR}--\ref{lemma:PropertiesOfR:1} this implies that $G$ has cotype \ref{T1}. This means that in \ref{StepR2} during which $G$ was created by subdividing a larger area $G'$ of type \ref{T2}, the area $G'$ had cotype \ref{T1}. Hence, by induction, at that time we had $N=1$ and $G'$ was the \emph{unique} area contributing to $N$. Therefore right after subdividing $G'$ and creating $G$ we had $N=0$, so after subdividing $\tilde{F}$ the number $N$ is at most $1$, which finishes the proof of the inductive step.
\end{itemize}

\end{proof}

\begin{proof}[Proof of \cref{prop:Rwelldef}]

\noindent $\bullet$ 
First, the fact that during the construction, all areas remain of type \ref{T1}, \ref{T2} and \ref{T3} is clear by induction, since in \ref{StepR2} all newly created areas are of this type (the area $f_1$ is of type \ref{T1}, $f_2,f_2,\dots,f_{k-1}$ are of type \ref{T2}, and $f_k$ is of type \ref{T3}), and this is the only step in which new areas are created.

\smallskip

\noindent$\bullet$ We now check that \ref{StepR1} is always successful in finding a face $F$ with the desired properties. 
Assume that, at some point of the execution of the algorithm, each vertex of $\pp$ labeled by $i$ is connected to at least one internal edge and that there is still at least one vertex of $\pp$ of label $i+1$ that is not connected to an internal edge. Let $v$ be such a vertex and let $F$ be the area that $v$ belongs to. Then $F$ is of type \ref{T2}.

We first note that if $F$ has cotype \ref{T2}, then we are done since $F$ is a valid choice. Moreover, if  $F$ has cotype \ref{T3}, its coarea $\tilde{F}$ is a valid choice and we are done too. So we now assume that $F$ has cotype \ref{T1}. By \cref{lemma:PropertiesOfR}--\ref{lemma:PropertiesOfR:2} $F$ is the only area of type \ref{T2} and cotype \ref{T1}.
 Let $e = e(F) \in \pp$ and let $e'$ be the other edge of $\pp$ bordering $F$ of minimum label $i$. Let $\tilde{F}'$ be the area containing the edge $\tilde{e'}$ matched with $e'$ in~$\M$. We distinguish cases according to the type of $\tilde{F}'$:
\begin{itemize}
\item We first claim that $\tilde{F}'$ cannot be of type \ref{T1}.
 Indeed observe that $e'$ is the \emph{first} edge of $\pp$ bordering $F$ having extremities labeled $i,i+1$, in the counterclockwise orientation induced by the blue graph on $F$. But in the construction rules,  an area of type \ref{T1} can be constructed only as the coarea of some area $G$ matched to the \emph{last} edge of $\pp$ bordering $G$ having extremities labeled $i,i+1$, in the counterclockwise orientation induced by the blue graph on~$G$.
\item If $\tilde{F}'$ is of type \ref{T3}, we are done since $\tilde{F}'$ is a valid choice.
\item If $\tilde{F}'$ is of type \ref{T2}, by \cref{lemma:PropertiesOfR}--\ref{lemma:PropertiesOfR:2} we know that $\tilde{F}'$ has cotype \ref{T2} or \ref{T3}. In the first case  $\tilde{F}'$ is a valid choice, and in the second case its coarea is one.
\end{itemize}

This concludes the proof that all the operations in \ref{StepR1} are well defined and that we always succeed in selecting a valid area. Since all the operations made in \ref{StepR2} rely only on the assumption that $\tilde{F}$ has type \ref{T2}, there is nothing more to prove and we conclude that the construction of $\rDEG(\M)$ and $\rplan(\M)$ is well-defined.

\medskip

\noindent$\bullet$ It remains to prove that $\rbij(\M)$ is a bipartite quadrangulation on $\mathbb{S}$. First, note that in the end of the construction all the areas of type \ref{T2} have degree three (since all the vertices of $\pp$ are connected to an internal edge). Moreover, by \cref{lemma:PropertiesOfR}--\ref{lemma:PropertiesOfR:1}, each area of type \ref{T3} has cotype \ref{T1}, and because in each \ref{StepR2} exactly one area of type \ref{T1}, and exactly one area of type \ref{T3} are created, the numbers of areas of type \ref{T1} and \ref{T3} are the same. Therefore the embedded graph $\rbij(\M)$ defined by the drawing of internal edges on the surface $\mathbb{S}$ decomposes $\mathbb{S}$ into two kind of components: faces obtained by gluing two areas of type \ref{T2} and degree 3 along their red edge; and faces obtained by gluing an area of type \ref{T3} and an area of type \ref{T1} along their red edge, as on~\cref{fig:mergingFaces} below.
In both cases these components are quadrangles, which proves both that $\rbij(\M)$ is a well-defined map on $\mathbb{S}$ (i.e., it provides a cellular decomposition), and that it is a quadrangulation. Finally, the fact that $\rbij(\M)$ is bipartite is clear since internal edges always link vertices whose labels have a different parity. 
\begin{figure}[h!!!!!]
\includegraphics[width=0.8\linewidth]{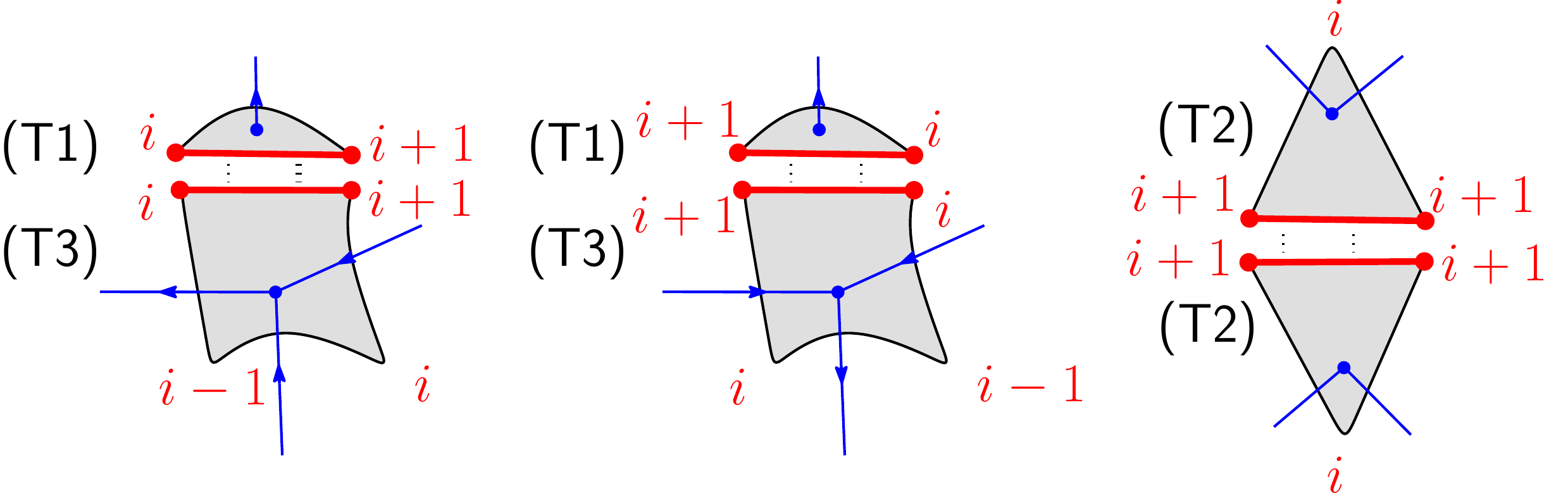}
\caption{Possible faces in $\rbij(\M)$ (up to reflection). Each of them is obtained by gluing together two areas (delimited by the planar graph $\rplan(\M)$ in $\pp$) 
along an edge-identification given by the matching structure of $\M$.
}\label{fig:mergingFaces}
\end{figure}
\end{proof}

We conclude this section with two important observations. By construction, the vertex-set of the quadrangulation $\qq=\rbij(\M)$ is $V(\M)\cup\{v_0\}$. Thus we can define the \emph{label} of a vertex of $\qq$ different from $v_0$ as its label in $\M$, and define the label of $v_0$ as $0$. We have:
\begin{lemma}\label{lemma:labelsDistances}
The label of a vertex in $V(\M)\cup\{v_0\}$ is equal to its distance to $v_0$ in the quadrangulation $\rbij(\M)$.
\end{lemma}
\begin{proof}
Just note that each vertex of label $i\geq1$ is linked by an internal edge to at least one vertex of label $i-1$, and to no vertex of label less that $i-1$. So the statement follows by induction on $i$.
\end{proof}

Moreover, let $d_\qq(v,w)$ denote the distance in $\qq$ between two vertices $v,w \in V(\qq)$. For any corner $c$ in $\M$ labeled by $i > 1$ there exists a unique corner $\suc(c)$ in $\M$ labeled by $i-1$ and a unique internal edge $c \to \suc(c)$ that connects these two corners (constructed in \ref{StepR2}). In particular, for any corner $c$ in $\M$ labeled by $i \geq 1$ there exists a geodesic in $\qq$ between the vertex $v(c)$ incident to the corner $c$ and the vertex $v_0$ given by $c \to \suc(c) \to \suc^2(c) \to \cdots \to \suc^{i-1}(c) \to v_0$, where $\suc^{i-1}(c) \to v_0$ is the unique internal edge connecting the corner $\suc^{i-1}(c)$ labeled by $1$ to the vertex $v_0$. We call it \emph{simple geodesic}. Let $c_1,c_2$ be two corners of $\M$ and let $[c_1,c_2]$ be the set of all visited corners during the walk along the boundary of the unique face of $\M$ starting from $c_1$ and finishing in $c_2$. The following observation has been communicated to us by Bettinelli \cite{BettinelliLemma2014} (with a slightly different proof).
It will not be used in the present paper, but 
it will be crucial in studying the convergence of the random quadrangulations in our forthcoming paper \cite{BettinelliChapuyDolega2015}:
\begin{lemma}
\label{lem:DistancesOfGeneric}
For any two corners $c_1, c_2$ of $\M$ the following inequality holds:
\[ d_{\qq}(v(c_1),v(c_2)) \leq \ell(c_1) + \ell(c_2) - 2\left(\max\left(\min_{x \in [c_1,c_2]}(\ell(x)),\min_{x \in [c_2,c_1]}(\ell(x))\right)-1\right),\]
where $v(c_i)$ denotes the unique vertex of $\qq$ incident to $c_i$ and $\ell(c_i)$ denotes the label of $c_i$.
\end{lemma}

\begin{proof}
Let $l := \max\left(\min_{x \in [c_1,c_2]}(\ell(x)),\min_{x \in [c_2,c_1]}(\ell(x))\right)$. We claim that 
\begin{equation}
\label{eq:CrucialEquality}
\suc^{\ell(c_1)-l+1}(c_1) = \suc^{\ell(c_2)-l+1}(c_2).
\end{equation}
This will finish the proof since
\begin{multline*} 
c_1 \to \suc(c_1) \to \cdots \to \suc^{\ell(c_1)-l+1}(c_1) = \\
= \suc^{\ell(c_2)-l+1}(c_2) \to \suc^{\ell(c_2)-l}(c_2) \to \cdots \to \suc(c_2) \to c_2
\end{multline*}
is a path between $v(c_1)$ and $v(c_2)$ of length
\[ \left(\ell(c_1)-l+1\right) + \left(\ell(c_2)-l+1\right) = \ell(c_1) + \ell(c_2) - 2(l-1).\]

Now, notice that any simple geodesic has the following property. If $c_1$ and $c_2$ are two corners of $\M$ with the same label $i \geq 1$, and all the corners in the open interval $(c_1,c_2)$ have labels strictly greater than $i$, then there exists a corner $c$ such that either $c = c_1$ or $c = c_2$ and such that any simple geodesic passing through $(c_1,c_2)$ has to pass through $c$. It is a straightforward consequence of the \ref{StepR2}, see \cref{fig:stepR2}.

Without loss of generality, we can assume that the 
label $l$ is realized by some corner $x$ lying in the segment $[c_1,c_2]$ and that $l > 1$ (if $l=1$, then the lemma holds true trivially). 
Let $x_1$ be the last corner in $[c_2,c_1]$ of label $l-1$ and let $x_2$ be the first corner in $[c_2,c_1]$ of label $l-1$
(note that $x_1$ and $x_2$ are well defined, because the minimum label in $[c_2,c_1]$ is $1$ since $\M$ is well-labelled, and $1<l$ by assumption).
 Then all the corners of the open interval $(x_1,x_2)$ have labels at least $l$, and the interval $[c_1,c_2]$ is contained in the interval $(x_1,x_2)$. Thus, from the above observation, there exists a corner $x$ such that either $x= x_1$ or $x = x_2$, and such that both simple geodesics from $c_1$, and from $c_2$ are passing through $x$. But this is exactly what we wanted to prove, namely \eqref{eq:CrucialEquality} holds true.
\end{proof}

\begin{remark}
An anonymous referee asked about the algorithmic complexity of our bijection. Clearly, the construction of $\bij(\qq)$ from $\qq$ can be done in time $O(n^2)$ since there are at most $n$ iterations of the main loop of the algorithm, and since \ref{Step1} can {\it a priori} take time $O(n)$ (indeed, one could have to follow the whole tour of the DEG before finding the desired face). In fact, using appropriate data structures, the complexity can be made linear (or quasilinear if one takes into account pointing and addressing operations). A way to achieve this is to maintain three cyclic lists of corners: the list $L$ of all corners of the DEG in the order they are visited when walking along the DEG, together with the position of the LVC in that cyclic list; and for $\epsilon\in\{0,1\}$, the list $L_\epsilon$ which is the sublist of the previous one consisting of corners of the DEG that belong to a face of type $(i'-1,i',i'+1,i')$ such that there is a unique blue vertex in that face, where $i'=i+\epsilon$ and $i$ is the running variable of the algorithm. In other words, $L_0$ lists corners that belong to a face that has the desired property to be selected at \ref{Step1} of the algorithm, and $L_1$ lists corners that may have this property after we increment $i$ by one. We also need to remember the position of the LVC in these lists.
 Assuming we have this data, the cost of finding the desired face in \ref{Step1} takes constant-time: one just has to look at the next position in the list $L_0$. There is a cost, however, in maintaining these three lists, since everytime we add a blue branch of edges in \ref{Step2}, we add some new corners to the tour -- thus changing the list $L$, and we may change the status of some of the faces -- thus changing the lists $L_0$ and $L_1$. However, over the whole execution of the algorithm, since each corner of the DEG is added at most once to each list, and since the status of each face is updated at most $3$ times, the total cost of these maintaining operations is linear (assuming that all access and pointing operations are done in constant time). 
Finally, when we finish the main loop of the algorithm and we update the value $i$ to $i+1$, we also have to update $L_0$ and $L_1$, which is easy to do in constant time by setting $L_0:=L_1$ and $L_1:=\emptyset$.
A similar construction shows that $\rbij(\M)$ can also be computed in linear time.
\end{remark}

\subsection{Bijection}\label{sec:bijection}
We are now ready to prove \cref{theo:MS1}, in the following, more precise, form.
\begin{theorem}
For each $n\geq 1$ and each surface $\mathbb{S}$, the mapping $\bij$ and $\rbij$ are inverse bijections between the set of rooted bipartite quadrangulations on $\mathbb{S}$ with $n$ faces and the set of rooted well-labeled unicellular maps on $\mathbb{S}$ with $n$ edges.
\end{theorem}
\begin{proof}
First, we have already proved that $\bij$ and $\rbij$ are well-defined mappings between the two sets in question. 

Before we prove that for any well-labeled unicellular map $\M$ one has $\bij(\qq)=\M$ where $\qq=\rbij(\M)$, and for any rooted bipartite quadrangulation $\qq$ one has $\rbij(\M)=\qq$ where $\M=\bij(\qq)$, we make the following observation. Note that by \cref{lemma:labelsDistances} all the notions of labelings used in the different constructions are the same. Moreover, by comparing~\cref{fig:redBlueRule} and~\cref{fig:mergingFaces}, we note that the relative position of the blue and black edges in both the forward and reverse construction are the same, and moreover, they force the position of red edges. 
So in order to prove that $\bij(\qq)=\M$ where $\qq=\rbij(\M)$
 (and, similarly, to prove that $\rbij(\M)=\qq$ where $\M=\bij(\qq)$) 
 it suffices to prove that the two blue graphs coincide, more precisely that $\DEG(\qq)=\rDEG(\M)$.

We will give a detailed proof in the first direction (starting from a well-labeled unicellular map $\M$ and $\qq=\rbij(\M)$). The proof of the other direction (when $\qq$ is a bipartite quadrangulation and $\M = \bij(\qq)$) is very similar and will be given with slightly less details.

\medskip
{\bf First direction.} 
Let us fix an arbitrary well-labelled unicellular map $\M$ and let $\qq=\rbij(\M)$. As explained above in order to prove that $\bij(\qq)=\M$, it suffices to show that the two blue graphs coincide, \textit{i.e.} that $\DEG(\qq)=\rDEG(\M)$.
Our proof is by induction and the idea is that in the forward and reverse construction, the rules are the same: the two graphs can be constructed simultaneously, step by step, and one can check that all the construction rules coincide. More precisely number the edges of $\DEG(\qq)$ (of $\rDEG(\M)$, respectively) in the order they are drawn during the construction by $e_1, e_2\dots, e_{2n}$ (by $e'_1,e'_2,\dots,e'_{2n}$, respectively). Then we claim that $e_i=e'_i$ for all $i$. For $i=1,2,\dots,d$ where $d$ is the degree of the root vertex $v_0$ in $\qq$ (equivalently $d$ is the number of corners of $\M$ labeled by $1$), this is clear by comparing \ref{Step0a}, \ref{Step0b} and \ref{StepR0}. Moreover, the LVC's coincide after these steps. Now suppose that $e_k=e'_k$ for $k \leq m$, and that after drawing the edges $e_m$ and $e'_m$ the forward and reverse algorithms have the same LVC,  and assume that we are going to construct a blue edge labeled by $i$.

\medskip

\noindent$\bullet$ In both \ref{Step1} and \ref{StepR1} we perform a tour around the already constructed blue graph starting from the LVC and we stop in some particular blue vertex. In \ref{Step1} that is the first vertex that lies in a face of $\qq$ bordered by $(i-1,i,i+1,i)$ and such that there are no other blue vertices in this face. We are going to prove that in \ref{StepR1} we stop in the same vertex. First, notice that after each step in the reverse construction there is exactly one blue vertex in each already constructed area. Moreover, at the end of the construction each
face of $\qq$ will be divided into exactly two areas.
We claim that the rules of \ref{StepR1} ensure that the blue vertex $v$ chosen in this step belongs to a face of $\qq$ bordered by $(i-1,i,i+1,i)$ and such that there are no blue vertices in this face. Indeed, there are two possible situations:
\begin{enumerate}
\item
\label{th:item1}
The vertex $v$ belongs to an area $F$ of type \ref{T3} with minimum label $i-1$ and cotype \ref{T2}. This means that in the next \ref{StepR2} we will subdivide the coarea $\tilde{F}$ of $F$ into smaller subareas, and a new subarea $\tilde{F}'$ of $\tilde{F}$ of type \ref{T1} will be the coarea of $F$ after the subdivision. Thus in the end of the construction $v$ will belong to the face $f(v)$ of $\qq$ consisting of areas $F$ of type \ref{T3} and $\tilde{F}$ of type \ref{T1}, 
bordered by $(i-1,i,i+1,i)$. Moreover, the unique blue vertex of degree $1$ belonging to this face is constructed when we create $\tilde{F}'$, hence it does not exists yet when we stop at the vertex $v$. This proves the claim in this case. 
Moreover, we note that when we stop at $v$, there is a unique free edge in $\qq$ incident to $f(v)$, namely the internal edge belonging to the area $\tilde{F}'$ (to be constructed in the following \ref{StepR2}), see~\cref{subfig:bigProofCase1-2}.

\begin{figure}[h!!!!!]
\centering
\subfloat[]{
	\label{subfig:bigProofCase1-1}
	\includegraphics[scale=0.47]{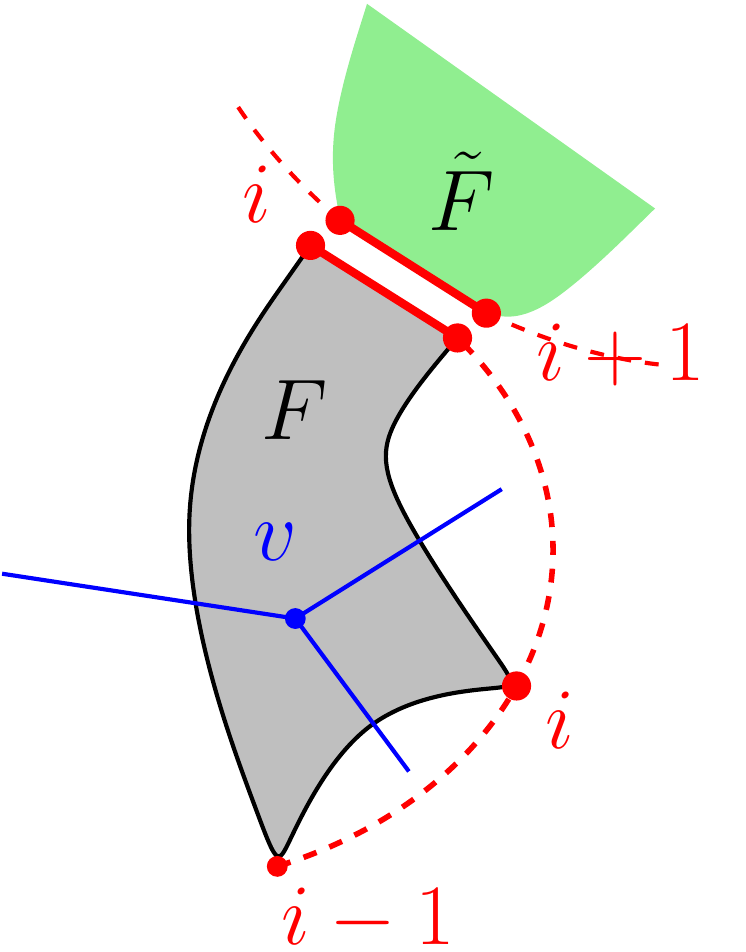}}
\quad
\subfloat[]{
	\label{subfig:bigProofCase1-2}
	\includegraphics[scale=0.47]{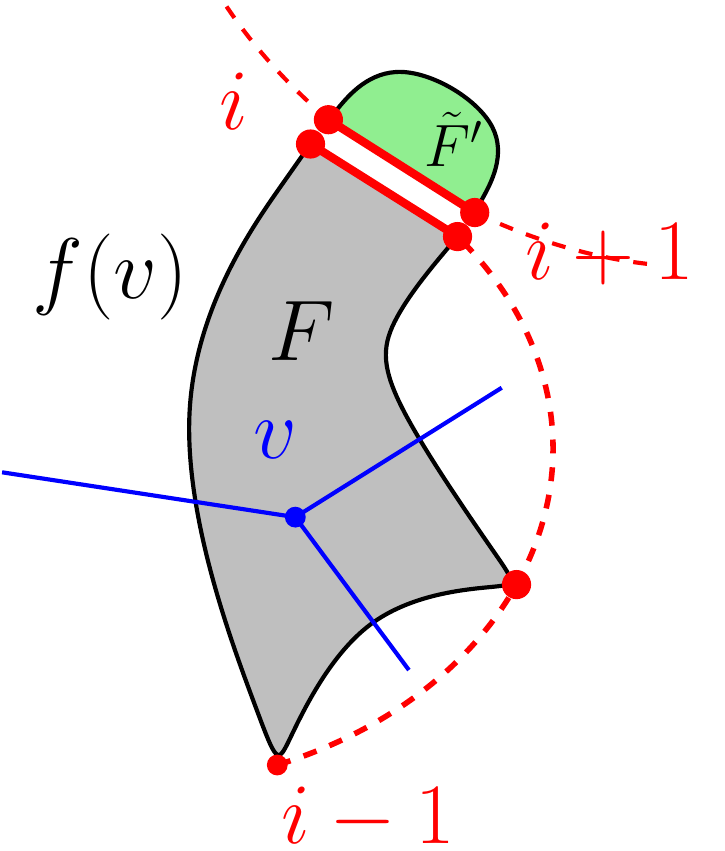}}
\quad
\subfloat[]{
	\label{subfig:bigProofCase1-3}
	\includegraphics[scale=0.47]{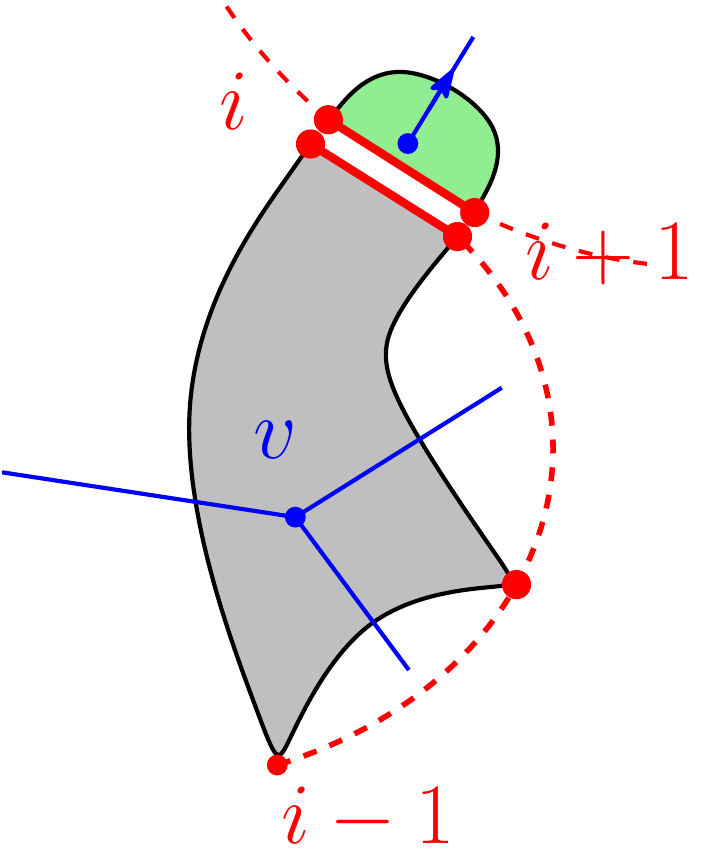}}
\caption{\protect\subref{subfig:bigProofCase1-1}: the vertex $v$ selected at \ref{StepR1} of the algorithm, in case \eqref{th:item1}. Only internal edges constructed at that step of the algorithm are displayed. \protect\subref{subfig:bigProofCase1-2}: The same vertex $v$, with the blue graph displayed at the same time of the construction as in \protect\subref{subfig:bigProofCase1-1}, but where we display the quadrangulation $\qq=\rbij(\M)$ as it will be in the end of the construction. \protect\subref{subfig:bigProofCase1-3}: Both the quadrangulation $\qq=\rbij(\M)$ and the blue graph $\DEG(\qq)=\rDEG(\M)$ at the end of the construction.}
\label{fig:bigProofCase1}
\end{figure}

\item
\label{th:item2}
The vertex $v$ belongs to an area $F$ of type \ref{T2} containing the label $i+1$, having minimum label $i-1$ 
and cotype \ref{T2}. This means that in the next \ref{StepR2} we will subdivide the coarea $\tilde{F}$ of $F$ into smaller subareas, and a new subarea $\tilde{F}'$ of $F$ of  type \ref{T1} will be the coarea of $F$ after subdivision. Since at the end of the construction each area of type \ref{T1} has cotype \ref{T3} (see \cref{fig:mergingFaces}), this means that, in $\qq$, $v$ belongs to a face $f(v)$ consisting of an area $F'$ of type \ref{T3} (that is a subarea of $F$) and the area $\tilde{F}'$ (see \cref{subfig:bigProofCase2-3}).
This face is bordered by $(i-1,i,i+1,i)$ and the unique blue vertex of degree $1$ belonging to it will be constructed when we create $\tilde{F}'$, hence it does not exists at the current stage of the construction (see \cref{subfig:bigProofCase2-2}).
This proves the claim in this case.
Moreover, observe that the edge of label $i$ appearing first in $f(v)$, after the unique corner labeled by $i-1$ encountered clockwise with respect to the orientation inherited from the blue graph around $f(v)$, coincides with the unique external edge of $\pp$ incident to $\tilde{F}'$. See \cref{fig:bigProofCase2}. 
\begin{figure}
\centering
\subfloat[]{
	\label{subfig:bigProofCase2-1}
	\includegraphics[scale=0.42]{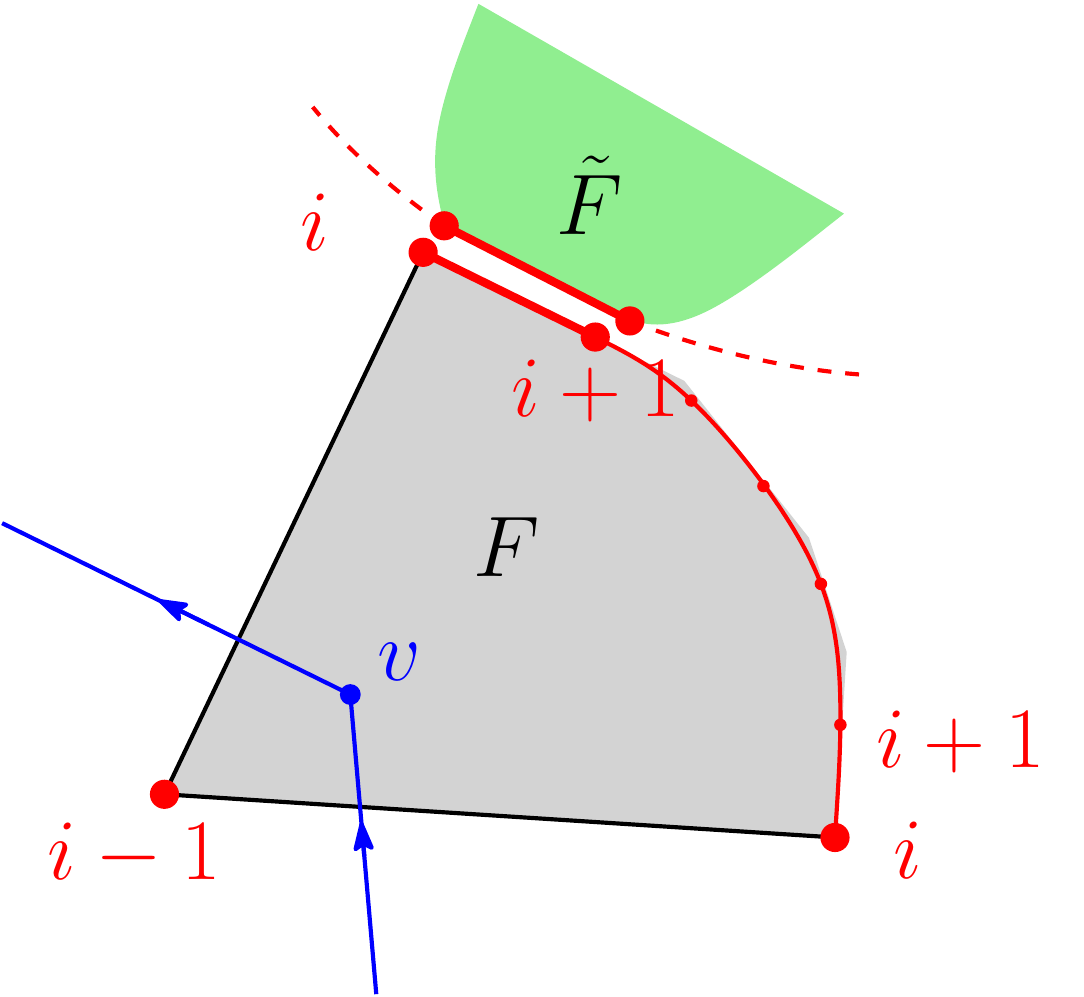}}
\subfloat[]{
	\label{subfig:bigProofCase2-2}
	\includegraphics[scale=0.42]{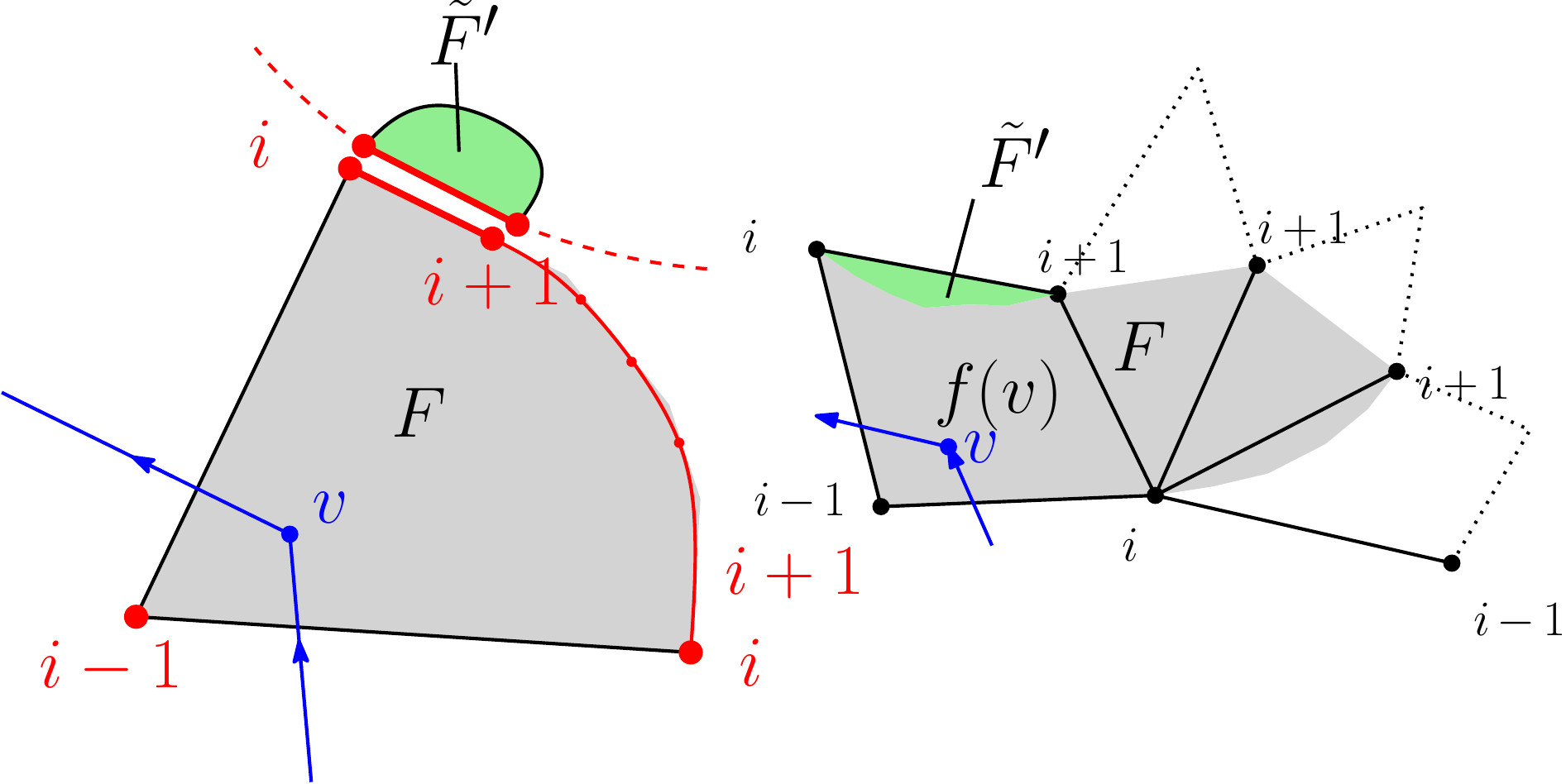}}

\subfloat[]{
	\label{subfig:bigProofCase2-3}
	\includegraphics[scale=0.42]{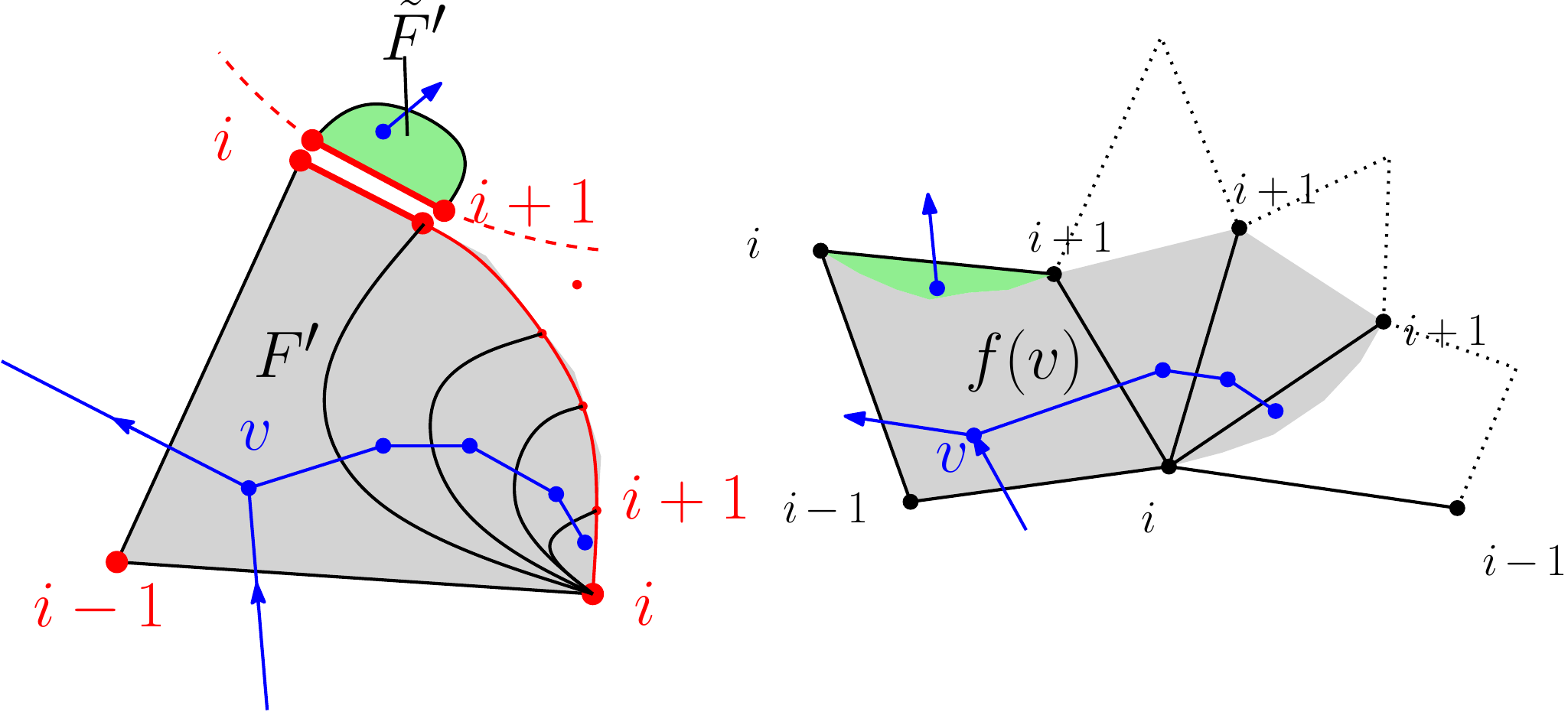}}
\caption{\protect\subref{subfig:bigProofCase2-1}: the vertex $v$ selected at \ref{StepR1} of the algorithm, in case \eqref{th:item2}. Only internal edges constructed at that step of the algorithm are displayed. \protect\subref{subfig:bigProofCase2-2}: The same vertex $v$, with the blue graph displayed at the same time of the construction as in \protect\subref{subfig:bigProofCase2-1}, but where we display the quadrangulation $\qq=\rbij(\M)$ on the right hand side as it will be in the end of the construction. We also indicate positions of areas $F$ and $\tilde{F}'$ relative to the position of edges of the quadrangulation $\qq$. \protect\subref{subfig:bigProofCase2-3}: Both the quadrangulation $\qq=\rbij(\M)$ and the blue graph $\DEG(\qq)=\rDEG(\M)$ at the end of the construction.}
\label{fig:bigProofCase2}
\end{figure}
\end{enumerate}
In order to finish the proof that both \ref{Step1} and \ref{StepR1} gives the same rules for 
selecting the vertex $v$,
we have to show that $v$ is the first blue vertex in our tour that belongs to the face of $\qq$ bordered by $(i-1,i,i+1,i)$ and such that there are no other blue vertices in this face. Assume, by contradiction, that there is a blue vertex $v'$ appearing before $v$ in our tour that has desired properties and assume that this is the first vertex with these properties. First, $v'$ cannot belong to an area $G$ of type \ref{T1}, since its coarea $\tilde{G}$ is either of type \ref{T3} or of type \ref{T2} and in both cases there already exists a second blue vertex in the face $f(v')$. From the same reason $v'$ cannot belong to an area of cotype \ref{T1}. We examine the remaining possibilities:
\begin{enumerate}
\item
The vertex $v'$ belongs to an area $G$ of type \ref{T3} and cotype \ref{T2} (by \cref{lemma:PropertiesOfR} and considerations above this is the only possible cotype). But the condition that $f(v')$ is bordered by $(i-1,i,i+1,i)$ implies that the minimum label of $G$ is $i-1$. This is a contradiction with the assumption that $v$ (which is different from $v'$) is the first vertex with these properties.
\item
The vertex $v'$ belongs to an area $G$ of type \ref{T2} and cotype \ref{T3}. By \cref{lemma:PropertiesOfR} coarea $\tilde{G}$ was the last created area. But it means that the blue vertex contained in $\tilde{G}$ is a valid choice in \ref{StepR1} and it appears before $v'$
 (since it is the first visited vertex in the tour) 
which is a contradiction.
\item
The vertex $v'$ belongs to an area $G$ of type \ref{T2} and cotype \ref{T2}. Since $f(v')$ is bordered by $(i-1,i,i+1,i)$ then the minimum label of $G$ is $i-1$ and area $G$ contains a vertex labeled by $i+1$. But this is a contradiction with the assumption that $v$ (which is different from $v'$) is the first vertex with these properties.
\end{enumerate}
This concludes the proof that \ref{Step1} and \ref{StepR1} give the same rules for selecting the next blue vertex to continue the construction.

\medskip

\noindent$\bullet$ We now prove that the rules to draw the new blue edges in \ref{Step2} and \ref{StepR2} coincide. Let $F, e, \tilde{e}$ and $\tilde{F}$ be defined as in \ref{StepR2}
 and let $v$ be the vertex of $\pp$ of label $i$ incident to $\tilde{e}$. 
In \ref{StepR2}, we subdivide $\tilde{F}$ into several areas $f_1, f_2,\dots, f_k$ as on \cref{fig:stepR2}, by adding new internal edges incident to $v$. These edges are (by definition) edges of $\qq$, and no other edge of $\qq$ incident to $v$ will be inserted in the corners separating these edges later in the construction. Therefore it is clear that the path
of blue edges created at this step (\cref{fig:stepR2}) turns counterclockwise around the vertex~$v$ in~$\qq$ as in \cref{fig:step2}. Moreover, this path ends at an already existing blue vertex (namely, $v_k$) lying in a face of $\qq$ of minimal label $i-1$. Hence, the only thing to prove, in order to prove that the rules of \ref{Step2} applied in $\qq$ lead to the construction of the same blue edges, is that the face of $\qq$ containing $v_k$ is the \emph{first} encountered around $v$ having minimal label $i-1$.
But this is clear since all the areas $f_2, \dots,f_{k-1}$ containing $v_2, \dots, v_{k-1}$ are of type \ref{T2} and have minimal label $i$, hence in $\qq$, each vertex $v_2, \dots, v_{k-1}$ can belong only to a face of type $(i,i+1,i,i+1)$ or $(i,i+1,i+2,i+1)$. This concludes the proof that \ref{Step2} and \ref{StepR2} construct the same blue edges.

\noindent$\bullet$ Finally, it is clear that \ref{Step3} and \ref{StepR3} lead to the same rule.

It means that $e_i = \tilde{e}_i$ for $1 \leq i \leq m+k-1$ and by the inductive argument it follows that $\DEG(\qq)=\rDEG(\M)$, which finishes the proof of the first direction. 

\medskip
{\bf Second direction.}
We now fix an arbitrary bipartite quadrangulation $\qq$, and we let $\M=\bij(\qq)$. As explained at the beginning of the proof, in order to prove that $\rbij(\M)=\qq$, it suffices to show that the two blue graphs coincide, \textit{i.e.} that $\rDEG(\M)=\DEG(\qq)$. As before, we are going to do it by induction, by showing that these two graphs can be constructed simultaneously and that the rules of construction are the same.
 More precisely let $e_i$ ($e'_i$, respectively) be the edge of $\qq$ (of $\rbij(\M)$, respectively) which is crossed by the $i$-th edge of $\DEG(\qq)$ ($\rDEG(\M)$, respectively) to be added during the construction. We are going to prove by induction that the two maps $\M \cup \{e_1,\dots,e_i\}$ and $\M \cup \{e'_1,\dots,e'_i\}$ coincide, which will conclude the proof.

\smallskip
 In order to do that, let us start from the following observation. For every corner of $\M$ labeled by $1$, there is a unique edge of $\rbij(\M)$ that connects it with the root vertex $v_0$, and for every corner of $\M$ labeled by $i \geq 2$ there is a unique edge of $\rbij(\M)$ that connects this corner with some corner of $\M$ labeled by $i-1$. This is straightforward from the rules of the construction of $\rbij(\M)$. We are going to show that the map $\qq$ has the same property, namely for every corner of $\M$ labeled by $1$, there is a unique edge of $\qq$ that connects it with the root vertex $v_0$, and for every corner of $\M$ labeled by $i \geq 2$ there is a unique edge of $\qq$ that connects this corner with some corner of $\M$ labeled by $i-1$.
First, it is obvious from the rules of the construction of $\bij(\qq)$ 
(see \cref{fig:redBlueRule}) 
that for every corner of $\M$ labeled by $i \geq 1$, there exists an edge of $\qq$ with label $i$, or $i-1$ having one extremity in that corner.

 We are going to prove that for each such corner there exists precisely one such edge with label $i-1$. Let us fix some corner $c$ of $\M$ labeled by $i$, and first assume (by contradiction), that all the edges of $\qq$ lying in $c$ have label $i$. Looking at the two possibilities shown on 
\cref{fig:redBlueRule}
one can see that the two edges of $\M$ adjacent to the corner $c$ belong to faces $f$ and $f'$ of $\qq$, respectively, both of type $(i-1,i,i+1,i)$. Moreover, since all the edges of $\qq$ lying in the corner $c$ have label $i$, there exists a path $p$, which is a connected component of $\DEG(\qq)$, going from $f$ to $f'$ and passing through all the faces of $\qq$ lying in $c$, as shown in \cref{subfig:OneDecreasingEdge1}. But \cref{lemma:invariant} ensures that $\DEG(\qq)$ is connected, hence $p = \DEG(\qq)$, which is clearly not true, and proves that there is at least one edge of $\qq$ labeled by $i-1$ and lying in $c$.
\begin{figure}[h!]
\centering
\subfloat[]{
	\label{subfig:OneDecreasingEdge1}
	\includegraphics[scale=0.7]{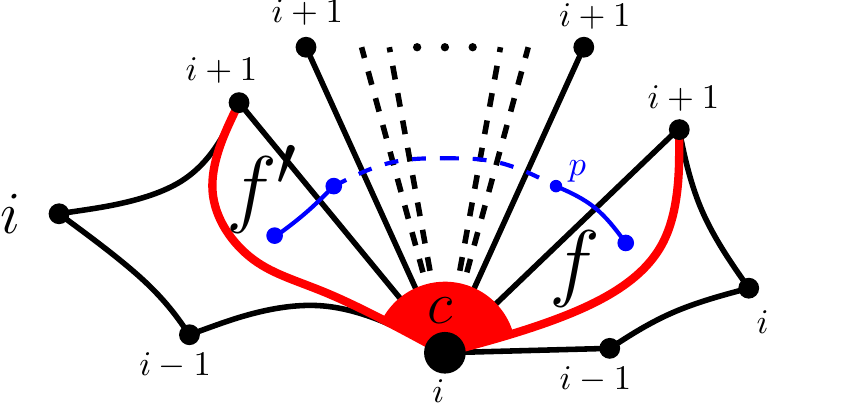}}
\subfloat[]{
	\label{subfig:OneDecreasingEdge2}
	\includegraphics[scale=0.7]{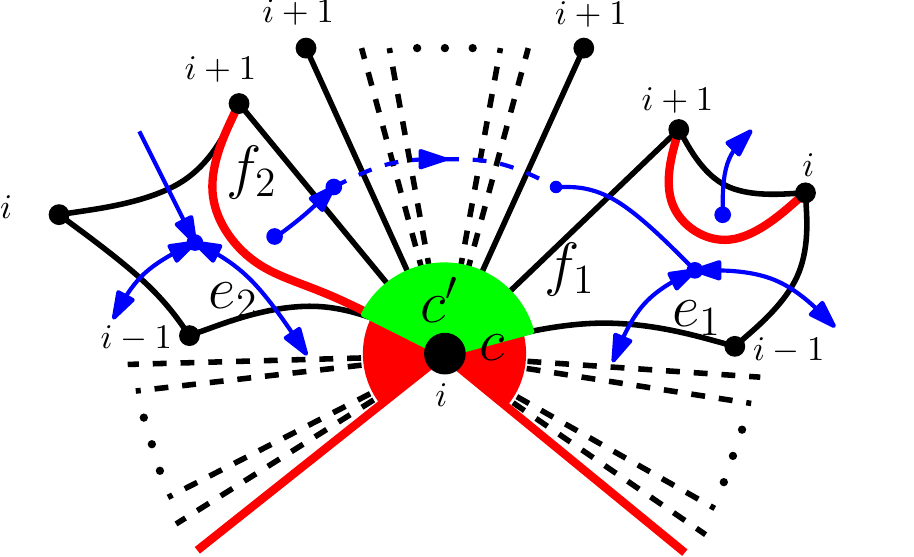}}
\caption{Illustration for the proof of \cref{theo:MS1}.} 
\label{fig:OneDecreasingEdge}
\end{figure}

Now assume (again by contradiction) that there are at least two edges of $\qq$ labeled by $i-1$ and lying in $c$. Two such consecutive edges have to be separated by at least one edge of $\qq$ labeled by $i$ 
(otherwise, there will be a face of type $(i,i-1,i,i-1)$ or $(i,i-1,i-2,i)$ lying in the corner $c$ and thus the edge of $\M$ lying in that face has label $i-1$, and lies in the corner $c$, which contradicts the assumption that $c$ is a corner of $\M$). 
Let $e_1$ and $e_2$ be two consecutive edges of $\qq$ labeled by $i-1$ and lying in the corner $c$ and let $c'$ be the unique corner between $e_1$ and $e_2$, which is a subcorner of $c$. Finally, let 
$f_1$ be the unique face of $\qq$ of type $(i-1,i,i+1,i)$ delimited by $e_1$ and lying in $c'$, and similarly, let $f_2$ be the unique face of $\qq$ of type $(i-1,i,i+1,i)$ delimited by $e_2$ and lying in $c'$, see \cref{subfig:OneDecreasingEdge2}.
 Then, the edge of $\M$ lying in the face $f_1$ connects a vertex labeled by $i+1$ to a corner of $\M$ labeled by $i$ (different from the corner~$c$). But then the rules of construction of $\DEG(\qq)$, see 
\cref{fig:redBlueRule}, 
say that the edge of $\M$ lying in $f_2$ has to connect a vertex labeled by $i+1$ with a vertex incident to $c$, which is a contradiction with the fact that $c$ is a corner of $\M$. This finishes the proof of the fact that for each corner of $\M$ labeled by $i \geq 1$, there exists a unique edge of $\qq$ with label $i-1$ lying in that corner.

\smallskip
We can now give the main arguments of the inductive proof that the maps $\M_j:=\M \cup \{e_1,\dots,e_j\}$, and $\M'_j:=\M \cup \{e'_1,\dots,e'_j\}$ are the same. We assume (induction hypothesis) that this is the case, where $e_j=e'_j$ is the edge crossed by the last edge added at some execution of \ref{Step2} and \ref{StepR2}. By induction hypothesis, both $\M_j$ and $\M'_j$ have only areas of types \ref{T1}, \ref{T2}, and \ref{T3} (since this is the case for $\M'_j$).
Moreover, let $f$ be an area of $\M_j=\M'_j$ of type \ref{T2} bordered by corners of $\M$ of labels $i,i+1,\dots,i+1,i$ for some $i\geq 1$. By the property that we have just proved (that any corner of $\M$ of label $i+1$ contains a unique edge of $\qq$ of label $i$), we deduce that all corners of label $i+1$  of $f$ are linked by an edge of $\qq$ to the \emph{same} corner of label $i$ of $f$ (since otherwise there would exist a face of $\qq$ of degree at least 5, which is impossible).

With this observation in mind, it is easy to conclude the proof by a careful observation of the rules of both constructions, similarly as in the proof of the first direction.
First, one has to prove that \ref{Step1} and \ref{StepR1} select the same area.
For this, since the first $j$ edges of $\DEG(\qq)$ and $\rDEG(\qq)$ are the same, walking around $\DEG(\qq)$ in \ref{Step1} is the same as walking around $\rDEG(\qq)$ in \ref{StepR1}, and it suffices to show that the first face meeting the requirements of \ref{Step1} is the same as the first face meeting the requirements of \ref{StepR1}. This is done by a case-by-case analysis of the different types of areas in $\M_j=\M'_j$, which is very similar to the one presented in the proof of the first direction.

Now, consider \ref{Step2} of the algorithm. During this step, we draw a new branch of blue edges inside an area $f$ of $\M_j$ of minimal label $i-1$: before drawing this path, this area cannot be of type \ref{T1} nor \ref{T3} so it has to be of type \ref{T2}. As observed above, all corners of label $i+1$  of $f$ are linked by edges of $\qq$ to the \emph{same} corner of label $i$ of $f$. It is then clear that the path of blue edges drawn in \ref{Step2} in $\M_j$ and in \ref{StepR2} in $\M'_j$ coincide, and that the black edges added to $\M'_j$ across this path coincide with the edges of $\qq$ crossing the corresponding blue path in \ref{Step2}. It is also clear that the position of the LVC will be updated coherently between both maps. 
This concludes the proof of the second direction.

\end{proof}

As promised, we now deduce \cref{theo:MS2} from \cref{theo:MS1}. 
Let $\qq$ be a rooted bipartite quadrangulation and let $v_0$ be a pointed vertex of $\qq$. Let us choose a corner $\rho(\qq,v_0)$ of $\qq$ incident to $v_0$. In general, there are several ways to choose this corner, and a simultaneous choice of the corner $\rho(\qq, v_0)$ for all $\qq$ and $v_0$ is called an \emph{oracle}. Given an oracle, we can consider the \emph{rerooted quadrangulation} $\qq'$ obtained by declaring that $\rho(\qq, v_0)$ is the root corner of $\qq'$. This quadrangulation is equipped with an additional marked corner (the original root corner of $\qq$). We can then apply the bijection $\bij$ to $\qq'$ and we obtain a well-labeled unicellular map $\M$. Since $\M$ has $2n$ corners and $\qq'$ has $4n$ corners, we can use the marked corner of $\qq'$ to mark (in some fixed canonical way) a corner of $\M$, and to get an additional sign $\epsilon\in\{+,-\}$ (for example choose the marked corner to be the unique corner of $\M$ incident to the root edge $e$ of $\qq$, and choose the sign $\epsilon$ according to whether the root vertex of $\qq$ has greater or smaller label than the other extremity of $e$ in the distance labeling of the rerooted quadrangulation $\qq'$). We declare this corner to be the new root of $\M$, and we now shift all the labels of $\M$ by the same integer, in such a way that this corner receives the label $1$. We thus have obtained a labeled unicellular map $\M'$, that carries a marked oriented corner $c$ of minimum label (the root corner of $\M$), together with a  sign $\epsilon$, and we denote by $\Lambda_\rho(\qq,v_0):=(\M',c,\epsilon)$.

We can now prove \cref{theo:MS2} in the following more precise form:
\begin{theorem}
For each $n\geq 1$ and each surface $\mathbb{S}$, there exists a choice of the oracle $\rho$ that makes $\Lambda_\rho$ a bijection between the set of rooted bipartite quadrangulations on $\mathbb{S}$ with $n$ faces and a marked vertex $v_0$, and the set of rooted labeled unicellular maps on $\mathbb{S}$ with $n$ edges and equipped with a sign $\epsilon\in\{+,-\}$.
\end{theorem}
\begin{proof}
For $d\geq 1$, consider the set $\mathcal{Q}_{\mathbb{S},n,d}$ of all rooted quadrangulations on $\mathbb{S}$ with $n$ faces equipped with a pointed vertex of degree $d$, and the set $\mathcal{U}_{\mathbb{S},n,d}$ of all rooted labeled unicellular maps on $\mathbb{S}$ with $n$ edges in which there are $d$ corners of minimum label and carrying a sign $\epsilon\in\{+,-\}$.

We now form a bipartite graph with vertex set $\mathcal{Q}_{\mathbb{S},n,d}\biguplus\mathcal{U}_{\mathbb{S},n,d}$ as follows. For each $(\qq,v_0)\in\mathcal{Q}_{\mathbb{S},n,d}$ we link it by an edge to all the pairs $(\M',\epsilon) \in \mathcal{Q}_{\mathbb{S},n,d}$ such that there exists a choice of the corner $\rho(\qq,v_0)$ and a corner $c$ of $\M'$ such that $\Lambda_\rho(\qq,v_0):=(\M',c,\epsilon)$.

From~\cref{theo:MS1}, this bipartite graph is regular: each rooted and pointed quadrangulation $(\qq, v_0)$ is incident to exactly $2d$ labeled unicellular maps (that corresponds to $2d$ possible choices of the oriented corner $\rho(\qq,v_0)$), and each labeled unicellular map $\M$ can be obtained in $2d$ ways (corresponding to the $2d$ ways of choosing an oriented corner of minimum label). Therefore from the Hall perfect matching theorem, this graph has a perfect matching. This perfect matching defines an oracle on quadrangulations pointed at a vertex of degree $d$, and doing this for all $d\geq 1$ we construct an oracle that fulfills the statement of the theorem.
\end{proof}

\section{Enumerative and probabilistic consequences}
\label{sec:consequences}

\subsection{Structure of the generating functions}\label{subsec:gf}
The first consequence of the bijection given in this paper is the combinatorial interpretation of the algebraicity of map generating functions, and more precisely of their rationality in terms of some parameters. The bijection being established, the situation is \emph{totally} similar to~\cite{ChapuyMarcusSchaeffer2009}, whose enumerative sections could be copied here almost verbatim.
In the rest of this section, we state the enumerative consequences of our bijection, and we give a brief overview of the ideas of the proofs,  referring to~\cite{ChapuyMarcusSchaeffer2009} for complete proofs in the analogous case of orientable surfaces.
\begin{theorem}[Combinatorial interpretation of the algebraicity and structure of map generating functions]
\label{theo:rationality}
Let $\mathbb{S}$ be a surface of type $h$, and let $\vec q_\mathbb{S}(n)$ be the number of rooted maps on $\mathbb{S}$ with n edges (equivalently, rooted quadrangulations with $n$ faces). 
Let 
$$Q_{\mathbb{S}}(t) := \sum_{n\geq 0} (n+2-2h) \vec q_{\mathbb{S}}(n) t^n
$$ be the generating function of rooted maps pointed at a vertex or a face, by the number of edges (equivalently, rooted quadrangulations pointed at a vertex, by the number of faces).

Then $Q_\mathbb{S}(t)$ is an algebraic function. More precisely, let $U\equiv U(t)$ and $T\equiv T(t)$ be the two formal power series defined by:
$$
T=1+3 t T^2, ~ ~ ~ ~ ~ U=tT^2(1+U+U^2).
$$
Then $Q_\mathbb{S}(t)$ is a rational function in $U$.
\end{theorem}

Similarly as in the orientable case~\cite{ChapuyMarcusSchaeffer2009}, the main singularity of $Q_\mathbb{S}(t)$ is easily localized and its exponent is determined combinatorially, from which we obtain a combinatorial interpretation of a famous result of Bender and Canfield~\cite{BC1}:
\begin{theorem}[Combinatorial interpretation of the map counting exponents]\label{theo:asymptotics}
For each $h\in\{\frac{1}{2},1,\frac{3}{2},2,\dots\}$, there exists a constant $p_h$ such that the number of rooted maps with $n$ edges on the non-orientable surface of type $h$ satisfies:
$$
\vec q_{h}(n) \sim p_h n^{\frac{5}{2}(h-1)} 12^n.
$$
\end{theorem} 

\begin{remark}
In order to be able to perform the enumeration of labeled one-face maps combinatorially, it is crucial to remove the positivity condition of labels (i.e., to use \cref{theo:MS2} rather than \cref{theo:MS1}). Because of this, we still cannot give a direct combinatorial proof that the series of rooted maps (without pointing) is algebraic. However, this is known to be the case from the generating function approach~\cite{BC1}. Note that the same remark applies to the orientable case treated in~\cite{ChapuyMarcusSchaeffer2009}. 
\end{remark}
\begin{remark}
The interesting point in \cref{theo:rationality} is that it is proved bijectively. Note however that it does not capture the full structure of the generating function $Q_\mathbb{S}(t)$. Indeed, $Q_\mathbb{S}(t)$ is known by generating function techniques~\cite{ArquesGiorgetti} to be a rational function in the parameters $T(t)$ and $a$ with $a=\frac{(1-U)(1+U)}{1+4U+U^2}$, which is a stronger statement than the one we prove here. 
\end{remark}
\begin{remark}
The combinatorial approach gives only a very complicated value of the constant $p_h$ of \cref{theo:asymptotics}, as a finite sum that is \emph{in principle} possible to compute for each fixed $h$ but in practice is too big even for $h=2$ (see~\cref{rem:pg} below). Again, this is not the end of the story: the deep algebraic structure of map generating functions (related to integrable hierarchies of differential equations) allows one to determine a simple recurrence formula for  the sequence $(p_h)$~\cite{Carrell2014}, but we are unable to give a combinatorial interpretation of it. Again, the situation in the orientable case is similar (see e.g.~\cite{BenderGaoRichmond2008}).
\end{remark}

As we said, the proof of~\cref{theo:rationality} is similar to the orientable case, so we just give a rapid account of where it comes from and refer to~\cite{ChapuyMarcusSchaeffer2009} for details. 
In the rest of this subsection we assume that $h>1/2$, i.e. we exclude the case of the projective plane that would require additional notation, whereas it is simple enough to be treated by hand. See~\cref{subsec:PP} for an explicit treatment of that case.

\smallskip

If $\M$ is a labeled unicellular map, its \emph{scheme}~$\mathfrak{s}$ is the unicellular map obtained by removing recursively all the vertices of degree~$1$ of~$\M$, and then replacing all the maximal chains of vertices of degree~$2$ by an edge.
 The vertices of~$\mathfrak{s}$ (call them $v_1, v_2, \dots, v_k$) naturally inherit integer labels $\ell(v_1), \ell(v_2), \dots, \ell(v_k)$ from the labeling of $\M$. We now normalize these labels by preserving their relative ordering so that they form an integer interval of the form $[1..K]$ where $K$ is the number of different labels in $\mathfrak{s}$. In other words, we define $\ell^*$ such that $\ell^*(v_i)<\ell^*(v_j)$ iff $\ell(v_i)<\ell(v_j)$ and $\cup_i\{\ell^*(v_i)\} = [1..K]$. The pair $(\mathfrak{s}, \ell^*)$ is called the \emph{normalized scheme} of $\M$.

We now define a \emph{scheme} as a unicellular map with no vertex of degree $1$ nor $2$. Note that this terminology is compatible with the above, since each scheme is \emph{the} scheme of some unicellular map (for example, of itself).
For each surface we let $\mathcal{S}_\mathbb{S}$ be the set of schemes on $\mathbb{S}$. We have:

\begin{lemma}\label{lemma:finiteSchemes}
For each surface $\mathbb{S}$, the set $\mathcal{S}_\mathbb{S}$ is finite. Moreover, the maximal number of vertices of a scheme on $\mathbb{S}$ is $4h-2$, where $h$ is the type of $\mathbb{S}$. This number is realized by and only by schemes whose all vertices have degree $3$. Those schemes also have the largest number of edges among elements of $\mathcal{S}_\mathbb{S}$, which is $6h-3$.
\end{lemma}

\begin{proof} The proof can be copied verbatim from~\cite[Section 4.6]{ChapuyMarcusSchaeffer2009} by replacing ``genus $g$'' with ``type $h$'', but since the statement is not detached from the text in that reference, let us recall the argument. By Euler's formula, a unicellular map with $n$ edges on a surface $\mathbb{S}$ of type $h$ has $v=n+1-2h$ vertices. Call $d_i$ its number of vertices of degree $i$ for $i\geq 1$. We have $\sum_i d_i=v$ and $\sum_i i d_i =2n$. For a scheme we have moreover $d_1=d_2=0$, and we deduce that $\sum_{i\geq 3}(i-2)d_i = 4h-2$. There are only finitely many integer sequences $(d_i)_{i\geq 3}$ satisfying this condition, which in turns implies that there are finitely many schemes on $\mathbb{S}$. Moreover,it is easy to see that a sequence maximizing the quantity $\sum d_i$ (or $\sum_i id_i$) under that condition is such that $d_i=0$ for $i>3$, and the lemma follows.
\end{proof}
\noindent The idea to prove~\cref{theo:rationality} is therefore to express the generating function of all labeled unicellular maps as a (finite) sum over schemes. We let $f_{(\mathfrak{s},\ell^*)}(t)$ be the generating function of rooted unicellular maps of normalized scheme $(\mathfrak{s},\ell^*)$. By~\cref{theo:MS2} we thus have:
\begin{align}\label{eq:finiteSum}
Q_{\mathbb{S}}(t) = 2 
\sum_{\mathfrak{s}\in\mathcal{S}_\mathbb{S}}
\sum_{K=1}^{|V(\mathfrak{s})|}
\sum_{\ell^*: V(\mathfrak{s}) \rightarrow [1..K]\atop
\mbox{\scriptsize surjective}}
f_{(\mathfrak{s},\ell^*)}(t).
\end{align}
Now, each unicellular map of normalized scheme $(\mathfrak{s},\ell^*)$ can be reconstructed as follows: 
\begin{itemize}[itemsep=0pt, topsep=0pt, leftmargin=20pt]
\item choose a labeling $\ell(v_1), \ell(v_2), \dots, \ell(v_k)$ of the vertices of the scheme compatible with $\ell^*$.
\item replace each edge $e$ of the scheme of extremities $e_-$, $e_+$ by a doubly-rooted labeled tree with roots labeled $\ell(e_-)$, $\ell(e_+)$.
\item choose an oriented corner as the root.
\end{itemize}
If follows by classical combinatorial decompositions of Motzkin walks (detailed in~\cite[Section 4.4]{ChapuyMarcusSchaeffer2009}) that the generating function of doubly-rooted labeled trees with roots of labels $i$ and $j$ is given by:
\begin{align}\label{eq:geometric}
M_{i-j} = B U^{|i-j|},
\end{align}
where $U$ is given in~\cref{theo:MS1} and $B=\frac{t(1+2U)}{1-t(1+2U)}$, so that:
\begin{align}\label{eq:polytope}
f_{(\mathfrak{s},\ell^*)}(t) = 
\frac{1}{|E(\mathfrak{s})|} \frac{td}{dt}
\sum_{\ell \mbox{\scriptsize ~compatible}\atop \mbox{\scriptsize with }\ell^*} \prod_{e\in E(\mathfrak{s})} M_{|\ell(e_-)-\ell(e_+)|} - {\bf 1}_{\ell(e_-)=\ell(e_+)}.
\end{align}
where the derivative accounts for the choice of the root corner, and the weight 
$\frac{1}{|E(\mathfrak{s})|}$ accounts for the multiplicity of the construction due to the fact that the scheme $\mathfrak{s}$ is rooted (more precisely the operator $4\cdot \frac{td}{dt}$ selects an oriented root corner in the map, and the prefactor $\tfrac{1}{4|E(\mathfrak{s})|}$ takes into account the rooting of the scheme; in the displayed result the simplification by $4$ has already  been made). Note that the indicator function in~\eqref{eq:polytope} is there to exclude the empty tree. Since the set of labelings $\ell$ compatible with a given $\ell^*$ forms an integer polytope, and given the geometric nature of the sequence $M_i$ (Equation~\eqref{eq:geometric}), it is to be expected that the infinite sum~\eqref{eq:polytope} leads to a rational function in $U$. The details of this resummation are similar to the proof of~\cite[Thm.2]{ChapuyMarcusSchaeffer2009}, and lead to:
\begin{proposition}\label{prop:explicitGF}
The generating function 
$f_{(\mathfrak{s},\ell^*)}(t) $ of labeled unicellular maps of normalized scheme $(\mathfrak{s},\ell^*)$ is given by:
$$
f_{(\mathfrak{s},\ell^*)}(t) = 
\frac{1}{|E(\mathfrak{s})|} \frac{td}{dt}
\frac{U^{e_=}(1+2U)^{e_{=}}(1+U+U^2)^{e_{\neq}}}{(1-U^2)^{|E(\mathfrak{s})|}}
\prod_{j=2}^K \frac{U^{d(j)}}{1-U^{d(j)}},
$$
where $e_=$ ($e_{\neq}$, respectively) is the number of edges of $\mathfrak{s}$ whose extremities have the same (different, respectively) label, where $K=\mathrm{card} \ell^*(V(\mathfrak{s}))$ is the number of different values taken by $\ell^*$, and where for $j\in[2..K]$ we let $d(j)$ be the number of edges $e$ of $\mathfrak{s}$ such that $\ell^*(e_-)<j\leq\ell^*(e_+)$.
\end{proposition}
Note that together with~\eqref{eq:finiteSum} this concludes the proof of~\cref{theo:rationality}. Finally,  it is easy to study the dominant singularity of the generating function $f_{(\mathfrak{s},\ell^*)}(t)$. First, $U(t)$ has radius of convergence $\tfrac{1}{12}$, and  we have when $t\rightarrow\tfrac{1}{12}$:
$$
U(t)=1-\sqrt{6}(1-12t)^{1/4} + O(\sqrt{1-12t}). 
$$
Therefore  $f_{(\mathfrak{s},\ell^*)}(t)$  has radius of convergence $\frac{1}{12}$, and its unique dominant singularity is of the form $c(1-12z)^{-\frac{L}{4}-1}$ where $L=|E(\mathfrak{s})|+K-1$. This implies that the leading contribution to \eqref{eq:finiteSum} comes from schemes in which the value of $L$ is maximal. From Lemma~\ref{lemma:finiteSchemes} these cases are exactly the ones where all vertices of the scheme have degree $3$ (so that $|E(\mathfrak{s})|=6h-3$) and in which all vertices of the scheme get different labels (so that $K=4h-2$), leading to the maximum value of $L=10h-6$, and a dominant singularity for $f_{(\mathfrak{s},\ell^*)}(t)$ of the form $c(1-12z)^{\frac{1}{2}-\frac{5}{2}h}$.
It is then straightforward to obtain \cref{theo:asymptotics}  by applying a standard singularity analysis of generating functions, see again~\cite{ChapuyMarcusSchaeffer2009} for details in the orientable case.

\begin{remark}\label{rem:pg}
Similarly as what was done for $t_g$ in~\cite[Cor.10]{ChapuyMarcusSchaeffer2009}, we obtain an expression of the constant $p_h$ that seems quite difficult to use in practice,  except maybe for very small values of $h$:
$$
p_{h}= \frac{3^h}{(6h-3)2^{11h-7}\Gamma\left(\frac{5h-3}2\right)}
\left(
\sum_{\mathfrak{s},\ell^*}\prod_{j=2}^{4h-2}\frac1{d(j)}
\right)
$$
where the summation is taken on pairs $(\mathfrak{s},\ell^*)$ such that $\mathfrak{s}$ is a scheme on the surface $\mathcal{N}_h$ with all vertices of degree $3$ and $\ell^*:V(\mathfrak{s})\rightarrow[1..4h-2]$ is a bijection. The notation for $d(j)$ is as above. 
\end{remark}

\subsection{The case of the projective plane}
\label{subsec:PP}

In the case of the projective plane, unicellular maps have a simple structure: a single cycle (whose neighborhood on the surface forms a M\"obius strip), with plane trees attached on it. These objects are easily enumerated by hand leading to the following result (unfortunately, although the enumeration is straightforward, the result turns out be expressed in summatory form):
\begin{corollary}[Bijective counting of maps on the projective plane]\label{cor:PP}
The number of rooted maps with $n$ edges on the projective plane is equal to:
$$
\frac{n}{n+1} \sum_{k\geq 1}  \frac{4 b_k}{n+k}{2n-1 \choose n-k} 3^{n-k},
$$
where $b_k = \sum_{a+2b=k} {k \choose a,b,b}$.
\end{corollary}

\begin{proof}
Call \emph{bridge} a lattice walk on $\mathbb{Z}$ starting at $0$, ending at $0$, and taking steps in $\{0, -1, +1\}$. Then clearly the number $b_k$ given in the statement of the theorem is the number of bridges with $k$ steps.

We claim that all labeled unicellular maps with $n$ edges on the projective plane can be constructed as follows. First, choose a bridge of size $k$ for some $k\geq 1$, close it in order to form a cycle of length $k$ with a marked vertex of label $0$, and embed this cycle in the projective plane so that its neighborhood forms a M\"obius strip. Now, on each of the $2k$ corners of this map, attach a labeled plane tree in order to have $n$ edges in total. The number of ways to do that is the number of plane forests with $n-k$ edges and $2k$ components, which is $\frac{2k}{n+k}{2n-1\choose n-k}$, times $3^{n-k}$ to choose the label variations along each edge of the forest.
Finally, there are $2n$ ways to choose a root corner in such a map, and $2$ ways to choose its orientation. In this way we obtain a doubly rooted map (in addition to its root it is also rooted at the origin of the bridge we started with). This doubly rooted map has $4k$ preimages by this construction, so we need to add a multiplicative factor of $\frac{4n}{4k}$ to compensate this multiplicity. We obtain that the number of rooted labeled unicellular maps with $n$ edges on the projective plane is equal to:
$$
 \sum_{k\geq 1}b_k \frac{2k}{n+k}{2n-1 \choose n-k} 3^{n-k} \times  \frac{n}{k} .
$$
By \cref{theo:MS2}, we multiply this number by $2$ to get the number of rooted and pointed quadrangulations with $n$ faces on the projective plane. Since such a quadrangulation has $n+1$ vertices, we finally need to divide the result by $n+1$ to get the number we want.
\end{proof}
\noindent We leave the reader check that \Cref{cor:PP} is equivalent to the generating function expression already computed in~\cite{BC1}.

\subsection{Distances in random quadrangulations}

Let $\M$ be a map and $v \in V(\M)$. We define the \emph{radius} of $\M$ centered at $v$ as the quantity
\[ R(\M,v) := \max_{u \in V(\M)} d_\M(v,u).\]
For any $r > 0$, we also define the \emph{profile of distances} from the distinguished vertex~$v$:
\[ I_{(\M,v)}(r) := \#\{u \in V(\M): d_\M(v,u) = r\}.\]

Then, using the now-standard machinery developed by \cite{ChassaingSchaeffer, LeGall2006, Bettinelli1} one can prove the following theorem:

\begin{theorem}
\label{theo:Distributions}
Let $q_n$ be uniformly distributed over the set of rooted, bipartite quadrangulations with $n$ faces on a surface $\mathbb{S}$, and let $v_0$ be the root vertex of $q_n$. 
Conditionally on $q_n$ let $v_*$
be a vertex chosen uniformly in $V(q_n)$. Then, there exists a continuous stochastic process $\frak{L}^\mathbb{S} = (\frak{L}^\mathbb{S}_t, 0 \leq t \leq 1)$ such that the following convergence results hold:
\begin{enumerate}[label=(C\arabic*)]
\item 
$\left(\frac{9}{8n}\right)^{1/4}R(q_n,v_*) \xrightarrow[n \to \infty]{(d)} \sup \frak{L}^\mathbb{S} - \inf \frak{L}^\mathbb{S}$;
\label{eq:DistributionOfRadius}
\item
$\left(\frac{9}{8n}\right)^{1/4}d_{q_n}(v_0,v_{*}) \xrightarrow[n \to \infty]{(d)} \sup \frak{L}^\mathbb{S}$;
\label{eq:DistributionOfTwoPoints}
\item
$\frac{I_{(q_n,v_*)}\left((8n/9)^{1/4}\cdot\right)}{n+2-2h} \xrightarrow[n \to \infty]{(d)} \mathcal{I}^\mathbb{S}$,
\newline
where $\mathcal{I}^\mathbb{S}$ is the occupation measure of $\frak{L}^\mathbb{S}$ above its infimum, defined as follows: for every non-negative, measurable $g :\R_+ \to \R_+$,
\[ \langle \mathcal{I}^\mathbb{S}, g \rangle = \int_{0}^1 \dt g(\frak{L}^\mathbb{S}_t - \inf \frak{L}^\mathbb{S}). \]
\label{eq:DistributionOfProfile}
\end{enumerate}
These convergences hold in distribution (and in the case of \ref{eq:DistributionOfProfile}, the underlying topology is the topology of weak convergence on real probability measures on $\mathbb{R}_+$).
\end{theorem}

We do not give the full proof of this theorem, since introducing all technical details is out of the scope of this paper, and the main technical points are totally analogous to what was done in~\cite{Bettinelli1} in the orientable case. However, let us briefly describe the origin of the process $\frak{L}^\mathbb{S}$ for interested reader. 

The idea is to generalize the \emph{label process} from the planar case. The label process of a pointed, rooted, planar quadrangulation with $n$ faces is a continuous function $\ell_n : [0,2n] \to \R$ defined in the following way: we set $\ell_n(0)$ to be the label of the root corner of the associated tree by the Schaeffer bijection, and we start walking along the boundary of that tree, visiting corners of the tree one by one. The function $\ell_n(i)$ records the label of the corner visited in the $i$-th step. In that way in the $2n$-th step we visit the root corner again, and we extend this function on the interval $[0,2n]$ by linearity. The \emph{contour process} $c_n : [0,2n] \to \R$ is defined in a similar way, but it records the height of the vertex visited at the $i$-th step rather than its label. Namely, $c_n(i)$ is the distance in the tree between the root vertex and the vertex visited in the $i$-th step of the walk. Then, it can be shown \cite{ChassaingSchaeffer} that if ones starts with a uniformly chosen random rooted and pointed planar quadrangulation with $n$ faces, the normalised contour process $\frac{1}{\sqrt{n}}c_n(2n \cdot)$ converges in distribution to the so-called \emph{normalized Brownian excursion} $c^\infty$ which, informally, can be understood as a standard Brownian motion conditioned to remain non-negative on $[0,1]$ and to take value $0$ at the time $1$. Moreover, the normalised label process $\left(\frac{9}{8n}\right)^{1/4}\ell_n(2n\cdot)$ converges to the so-called \emph{head of the Brownian snake} $\ell^\infty = (\ell^\infty_t, 0\leq t \leq 1)$ which is, conditionally on $c^\infty$, the continuous Gaussian process with covariance
$$\Cov(\ell^\infty_s, \ell^\infty_t) = \inf\{c^\infty_u : \min(s,t,) \leq u \leq \max(s,t)\}.$$
\medskip

For rooted and pointed bipartite quadrangulations with $n$ faces on a general surface $\mathbb{S}$, one  can define analogues of the label process and the contour process in the same way as Bettinelli did for orientable surfaces~\cite[Section 3]{Bettinelli1}. Bettinelli's construction is more technical than the planar construction, nevertheless the definition of the label process is the same: for $0\leq i\leq 2n$, one lets $\ell_n(i)$ be the label of the corner visited at the $i$-th step when going around the unicellular map, starting
from the root, associated with the pointed map $(q_n,v_*)$ via \cref{theo:MS2}. Using the decomposition of labeled one-face maps via schemes sketched in \cref{subsec:gf}, one can view, roughly speaking, the label process of a random labeled one-face map on $\mathbb{S}$ as the ``concatenation'' of the label processes of finitely many random labeled forests (each forest corresponding to the trees attached to one edge-side of the scheme).
Once the notion of scheme has been adapted to cover the case of general surfaces, the details of this decomposition are totally independent of the notion of orientability. In particular, the construction in \cite[Sections 4,5,6]{Bettinelli1} and the convergence results in that reference apply verbatim to our case. It follows that the normalized label process of the uniform random, rooted and pointed, bipartite quadrangulation $(q_n,v_*)$ with $n$ faces on the surface~$\mathbb{S}$  converges after normalization:
\begin{align}\label{eq:convProcess}
\left(\frac{9}{8n}\right)^{1/4} \ell_n^{\mathbb{S}}(2n\cdot)\longrightarrow \frak{L}^{\mathbb{S}}_{(\cdot)},
\end{align}
where $\frak{L}^{\mathbb{S}}_{(\cdot)} = \frak{L}^{\mathbb{S}}$ is a continuous process that describes the concatenation of the limiting label processes of each individual forest, defined analogously as the process $\frak{L}_\infty$ of~\cite[p. 1632]{Bettinelli1}. Here, the convergence is in distribution for the uniform topology on continuous functions on $[0,1]$.

We conclude by sketching how the statements in \cref{theo:Distributions} follow from the convergence~\eqref{eq:convProcess}. First, it follows from \cref{theo:MS2} that in the discrete setting, $R(q_n,v_*)$ is equal to $\max_{v\in V(q_n)} \ell(v) -\min_{v\in V(q_n)} \ell(v)+1$. Therefore \ref{eq:DistributionOfRadius} follows from \eqref{eq:convProcess} and the fact that $\max$ and $\min$ are continuous for the uniform norm. Similarly, $d_{q_n}(v_*,v_0)$ is equal to  $1-\min_{v\in V(q_n)} \ell(v)$, and \ref{eq:DistributionOfTwoPoints} follows from the continuity of $\min$ and the fact that labeled one-face maps are invariant by global reflexion of all labels.
Finally, let us sketch a way of deducing \ref{eq:DistributionOfProfile} from \eqref{eq:convProcess}, inspired by the proof of Theorem 4.1.2 in \cite{Miermont:cours}.  To simplify the argument we choose a probability space on which the convergence~\eqref{eq:convProcess} holds almost surely, and we want to prove that for each real uniformly continuous bounded function $g$ the quantity $\frac{1}{n+2-2h}\sum_{k} I_{(q_n,v_*)}(k) g\left((\frac{9}{8n})^{1/4} k \right)$ converges to $\langle \mathcal{I}^\mathbb{S}, g \rangle$. First, note that by definition the first quantity is equal to 
\begin{align}\label{eq:profileAsExpectation}
\mathbb{E}_{v_{**}}\left[g\left((\frac{9}{8n})^{1/4}d_{q_n}(v_*,{v_{**}})\right)\right]
\end{align}
 where $\mathbb{E}_{v_{**}}$ denotes the expectation over a random uniform vertex ${v_{**}}$ in $V(q_n)$, conditionnally on $q_n$. The main point now will be to argue that the value of $d_{q_n}(v_*,{v_{**}})$ is well approximated in the limit by an evaluation of the shifted label process at a uniform random point of $[0,2n]$. To do this, consider the labeled unicellular map $\M_n$ associated with $q_n$. Orient the edges of $\M_n$ as follows: first choose an arbitrary orientation for each edge $e$ of the scheme, and orient each edge of $\M_n$ in the chain of vertices corresponding to $e$ with this orientation; then orient all the edges that belong to the trees ``attached'' to these chains towards the leaves of the trees. In this way, all the vertices of $\M_n$ except at most $6h(\mathbb{S})-3$ of them (the vertices of the scheme) have indegree exactly one. Now, for each $s\in [0,2n]$, consider the edge-side $\bar e$ of $\M_n$ that is visited by the contour of $\M_n$ at time $s$, starting from the root, where we now think of $s$ as a continuous time. There exists $\langle s\rangle \in \{\lfloor s \rfloor, \lceil s \rceil\}$ such that the vertex visited by the contour at time $\langle s\rangle$ is the out-vertex of $\bar e$. In this way we construct a function 
 $\langle \cdot \rangle: [0,2n] \rightarrow [0..2n]$ such that $|\langle s \rangle -s | \leq 1$, and such that if $U$ is a uniform variable on $[0,1]$, the vertex of $\M_n$ visited by the contour at time $\langle 2n U \rangle$ is almost a uniform vertex on $V(\M_n)$ -- in the sense that the two differ by at most $O(\frac{1}{n})$ in total variation.
It follows that \eqref{eq:profileAsExpectation} is equivalent when $n$ tends to infinity to
$$\mathbb{E}_U\left[g\left((\frac{9}{8n})^{1/4}(\ell_n(\langle 2n U\rangle)+1-\min_v \ell_n(v)\right)\right].$$
Since $|\langle s \rangle -s | \leq 1$, this last quantity clearly converges, when $n$ goes to infinity, to $\mathbb{E}_U\left[g(\frak{L}^\mathbb{S}_\infty(U)-\min_{[0,1]} \frak{L}^\mathbb{S}_\infty)\right]=\langle \mathcal{I}^\mathbb{S}, g \rangle$.

We stop here our considerations on random maps leaving, in particular, the question of Gromov-Haussdorf tightness, convergence, and the study of Gromov-Hausdorff limits, to the forthcoming work~\cite{BettinelliChapuyDolega2015}.

\renewcommand{\M}{\mathcal{M}}

\section{Extension: the Miermont/Ambjørn-Budd bijection for general surfaces}
\label{sec:Miermont}

\subsection{Miermont bijection for general surfaces}
\label{subsect:Miermont}

In this section we are extending \cref{theo:MS2} to the case of multi-pointed quadrangulations. This construction extends Miermont's one \cite{Miermont} to the case of \emph{all} surfaces, that is orientable and non-orientable. Before we state our theorem we need some additional definitions.

\begin{definition}
\label{def:MapWithSources}
For any map $\M$ and integer $k\geq 1$,  let $W=(w_1,w_2,\dots,w_k)$ be a $k$-tuple of distinct vertices of $M$, and let $d$ be a function $d : W \rightarrow \Z$ such that:
\begin{enumerate}[label=(DS\arabic*)]
\item
\label{con:Delay0th}
$\min_{w \in W}d(w) = 0$;
\item
$|d(w_i) - d(w_j)| < d_\M(w_i,w_j)$ for any $1\leq i < j \leq k$;
\label{con:Delay1st}
\item
$d(w_i) - d(w_j) + d_\M(w_i,w_j) \in 2\N$ for any $1\leq i < j \leq k$.
\label{con:Delay2nd}
\end{enumerate} 
Then we will call elements of $W$ \emph{sources} and we say that  the triple $(\M,W, d)$ defines a \emph{map with $k$ delayed sources}, where $k$ is the cardinality of $W$. The values $d(w_i)$ for $w_i\in W$ are called the \emph{delays}. 
\end{definition}

Whenever we have a map with $k$ delayed sources $(\M,(w_1,\dots,w_k),d)$, we associate with it the \emph{distance from delayed sources}:
\begin{equation} 
\label{eq:DistanceOfSources}
\ell^d: V(\M) \rightarrow \Z, \quad \ell^d(v) := \min_{1 \leq i \leq k}\left( d_\M(v,w_i) + d(w_i)\right).
\end{equation}

Notice that when all delays are equal to $0$, then the function $\ell^d$ measures the distance from the given vertex to the set of all sources $\{w_1,\dots,w_k\}$, which explains our terminology.

We are going to prove the following theorem:

\begin{theorem}
\label{theo:Miermont2}
For each surface $\mathbb{S}$ and integers $n,k\geq 1$, there exists a $2$-to-$1$ correspondence between the set of rooted bipartite quadrangulations on $\mathbb{S}$ with $n$ faces and $k$ delayed sources $(W=(w_1,\dots,w_k), d)$, and labeled maps on $\mathbb{S}$ with $n$ edges and $k$ faces, numbered from $1$ to $k$.

Moreover, if for a given bipartite quadrangulation with $k$ delayed sources $(\qq, W, d)$ we denote by $N_i$ ($\tilde{N}_i$, respectively) the set of vertices $v \in V(\qq)$ ($v \in V(\qq)\setminus W$, respectively) such that $\ell^d(v) = i$, and by $E(N_i,N_{i-1})$ the set of edges between $N_i$ and $N_{i-1}$, then the associated labeled map $\M$ has $|\tilde{N}_i|$ vertices of label $i+\ell_{min}-1$ and $|E(N_i,N_{i-1})|$ corners of label $i+\ell_{min}-1$, where $\ell_{min}$ is the minimum vertex label in $\M$.
\end{theorem}

Similarly as in the case of \cref{theo:MS2}, we will first prove a different theorem, from which \cref{theo:Miermont2} follows.

\begin{theorem}
\label{theo:Miermont1}
For each surface $\mathbb{S}$ and integers $n,k \geq 1$, there exists a bijection between the set of bipartite quadrangulations on $\mathbb{S}$ with $n$ faces, $k$ delayed sources $(W=(w_1,\dots,w_k), d)$, such that each source is marked with a distinguished oriented corner incident to it, 
and well-labeled maps on $\mathbb{S}$ with $n$ edges and $k$ faces ordered from $1$ to $k$, such that each face contains a distinguished oriented corner, and such that the minimum label of distinguished corners is $1$.

Moreover, if for a given bipartite quadrangulation with $k$ delayed sources $(\qq, W, d)$ we denote by $N_i$ ($\tilde{N}_i$, respectively) the set of vertices $v \in V(\qq)$ ($v \in V(\qq)\setminus W$, respectively) such that $\ell^d(v) = i$, and by $E(N_i,N_{i-1})$ the set of edges between $N_i$ and $N_{i-1}$, then the associated well-labeled map $\M$ has $|\tilde{N}_i|$ vertices of label $i$ and $|E(N_i,N_{i-1})|$ corners of label $i$.
\end{theorem}

\begin{remark}
Notice that for $k=1$  \cref{theo:Miermont1} coincides with \cref{theo:MS1} and \cref{theo:Miermont2} coincides with \cref{theo:MS2}. Indeed, a rooted map with unique delayed source $(\M, \{w\}, d)$ coincides with the pointed rooted map $(\M,w)$. 
\end{remark}

The construction behind \Cref{theo:Miermont2} is almost identical to the one presented in \cref{sec:bij}, so instead of proving everything twice, we only present differences between both constructions with necessary proofs. 

\begin{lemma}
\label{lem:DistanceFromSources}
For a given bipartite quadrangulation with $k$ delayed sources $(\qq,$ $(w_1,\dots,w_k), d)$, and for any pair of adjacent vertices $v,w \in V(\qq)$ one has
\[ |\ell^d(v) - \ell^d(w)| =1.\]
\end{lemma}

\begin{proof}
Clearly, by bipartition of $\qq$, for any adjacent pair of vertices $v,w \in V(\qq)$, and for any $i \in [k]$ we have that $|\ell^d_i(v) - \ell^d_i(w)| = 1$, where $\ell^d_i(v) = d_\qq(w_i,v) + d(w_i)$. Taking $\ell^d = \min_{1 \leq i \leq k}\ell^d_i$, we have that $|\ell^d(v) - \ell^d(w)| \leq 1$. If there exists adjacent pair of vertices $v,w \in V(\qq)$ such that $\ell^d(v) = \ell^d(w)$, then there exists $i \neq j$ such that $\ell^d_i(v) = \ell^d_j(w)$. This means that
\[ d_\qq(v,w_i) - d_\qq(w,w_j) + d(w_i) - d(w_j) = 0,\]
and since $u$ and $w$ are adjacent we have that $d_\qq(v,w_i) - d_\qq(w,w_i) \equiv 1 \mod 2$. Using bipartition of $\qq$ again, we obtain that $d_\qq(w,w_i) - d_\qq(w,w_j) \equiv d_\qq(w_i,w_j) \mod 2$ and putting it all together we have that
\[ d_\qq(w_i,w_j) + d(w_i) - d(w_j) \equiv 1 \mod 2,\]
contradicting \cref{def:MapWithSources}--\ref{con:Delay2nd}.
\end{proof}

We are ready now to describe the construction that follows the same steps as the construction of \Cref{sec:bij}. We number these steps in the same way preceding them with the letter ``M'' for ``multipointed''. 
Notions such as labels of edges, corners, or types of faces that are not explicitly redefined keep the same meaning as in~\Cref{sec:bij}.

\subsubsection{Constructing the dual exploration graph $\DEG(\qq, W, d)$}
\label{subsubsect:SpanningTreeInStepsMiermont}

Similarly to \cref{subsubsect:SpanningTreeInSteps} for a quadrangulation with $k$ delayed sources $(\qq, W, d)$ on a surface $\mathbb{S}$ we describe how to draw a directed graph $\DEG(\qq, W, d)$ on the same surface. As before, to distinguish vertices, edges, etc.~of the quadrangulation and of the new graph we will say that the vertices, edges, etc.~of the graph $\DEG(\qq, W, d)$ are \emph{blue}, while the edges of $\qq$ are \emph{black}.

\smallskip
 
\begin{enumerate}[label=$\bullet$, ref=Step M0--\alph*, leftmargin=0cm]
\item \label{StepM0a} \textbf{Step M0--a.} We label the vertices of $\qq$ according to their distance from the delayed sources $(W, d)$.
By \cref{lem:DistanceFromSources},  the faces of $\qq$ are either of type $(i - 1, i, i - 1, i)$ or of type $(i - 1, i, i + 1, i)$ for $i > 0$.

\medskip
As in~\Cref{sec:bij}, our goal is to draw a blue graph in such a way that at the end of the construction, each edge of the quadrangulation $\qq$ is crossed by exactly one blue edge. As before, we call \emph{free} a black edge that has not yet been crossed by a blue edge in the construction.
\medskip 

\item \label{StepM0b} \textbf{Step M0--b.}
For each source $w_i$, $1\leq i\leq k$, we add a new blue vertex in each corner incident to $w_i$, and we connect these blue vertices together by a cycle of blue edges encircling $w_i$, as on \cref{fig:step0}.
We orient this cycle as in the previous construction, using the orientation of the distinguished corner of $w_i$ (call it $c_i$). There is a unique vertex of the blue graph lying in $c_i$, and this vertex has a unique corner that is separated from $w$ by the blue cycle. We declare that corner to be the $i$-th \emph{last visited corner ($LVC_i$)} of the construction and we equip it with the orientation inherited from the one of the cycle (the $LVC_i$ will be dynamically updated in the sequel). At the end of this step we constructed $k$ blue oriented cycles, one around each source, and we distinguished $k$ oriented corners $LVC_1, \dots, LVC_k$.

As in~\Cref{sec:bij}, we are going to construct the blue edges (and thus the graph $\DEG(\qq, W, d)$) by increasing label. 
We will start by drawing edges of label $1$, so we let $i := 1$. We will start our construction at the first $LVC$, so we set $LVC:=LVC_1$ (the value of $LVC$ will be dynamically updated in the sequel).
\end{enumerate}

\noindent We now proceed with the inductive part of the construction. 
\begin{enumerate}[label=$\bullet$, ref=Step M\arabic*, leftmargin=0cm]
\item \label{StepM1} \textbf{Step M1.}
If there are no more free edges in $\qq$, we stop. Otherwise, we perform the tour of the blue graph, starting from the $LVC$, which by construction is equal to $LVC_j$ for some $1\leq j \leq k$. If we reach the $LVC$ again, we skip to the next connected component of the blue graph by letting $LVC := LVC_{j+1}$ and we go back to \ref{StepM1} (here $LVC_{k+1} = LVC_1$ by convention).
We stop as soon as we visit a face $F$ of $\qq$ having the following properties: $F$ is of type $(i-1,i,i+1,i)$, and $F$ has exactly one blue vertex already placed inside it. As in the case of~\Cref{sec:bij}, such a face always exists (this is proved in \cref{prop:welldefMiermont} below).

We now apply the same construction as in~\Cref{sec:bij} to select an edge $e$ according to the rules of~\Cref{fig:chooseFe}.
Namely, if the face $F$ is incident to only one free edge, we let $e$ be that edge.
If not, the same argument as in \cref{lemma:invariant} ensures that the blue vertex $u$ contained in $F$ is incident to two blue edges of label $(i-1)$, one incoming and one outgoing. We then let $e$ be the first edge of label $i$ encountered clockwise around $F$ after that corner (where $F$ is oriented in such a way that the two oriented blue edges turn counterclockwise around the corner of label $(i-1)$). 

\item \label{StepM2} \textbf{Step M2.} This step is identical to \ref{Step2} in \Cref{sec:bij}. Namely, we 
draw a new blue vertex $v$ in the unique corner of $F$ delimited by $e$ and its neighboring edge of label $i-1$, and we let $a$ the vertex of $\qq$ incident to this corner.
Now, as on \cref{fig:step2}, we draw a path of directed blue edges starting from $v$, that starts by crossing $e$, and continues recursively by turning 
counterclockwise 
around $a$ until it reaches a face containing a corner of label $i-1$, and finally connects to the blue vertex $w$ of label $i-1$ in that face (the same argument as in \cref{lemma:invariant} ensures that this blue vertex exists).
Then, there exists a unique $1 \leq j \leq k$ such that $LVC_j$ is attached to the connected component of the blue graph containing vertex $w$ and we update the value of $LVC_j$, as the corner lying to the right of the last directed blue edge we have drawn, in the local orientation defined by the fact that the path just drawn turns counterclockwise around $a$, and we also update $LVC := LVC_j$.

\item \label{StepM3} \textbf{Step M3.} If there are  no more free edges of label $i$ in $\qq$, we set $i:=i+1$, otherwise we let $i$ unchanged. We then go back to Step 1 and continue.
\end{enumerate}

\begin{enumerate}[label=$\bullet$, ref=Termination, leftmargin=0cm]
\item \label{StepMTermination} \textbf{Termination.} We let $\DEG(\qq, W, d)$ be the blue embedded graph on $\mathbb{S}$ obtained at the end of the construction.
\end{enumerate}

\cref{subfig:exampleDEG-Miermont2} gives an example of the construction.
\begin{figure}
\centering
\subfloat[]{
	\label{subfig:exampleDEG-Miermont1}
	\includegraphics[width=0.33\linewidth]{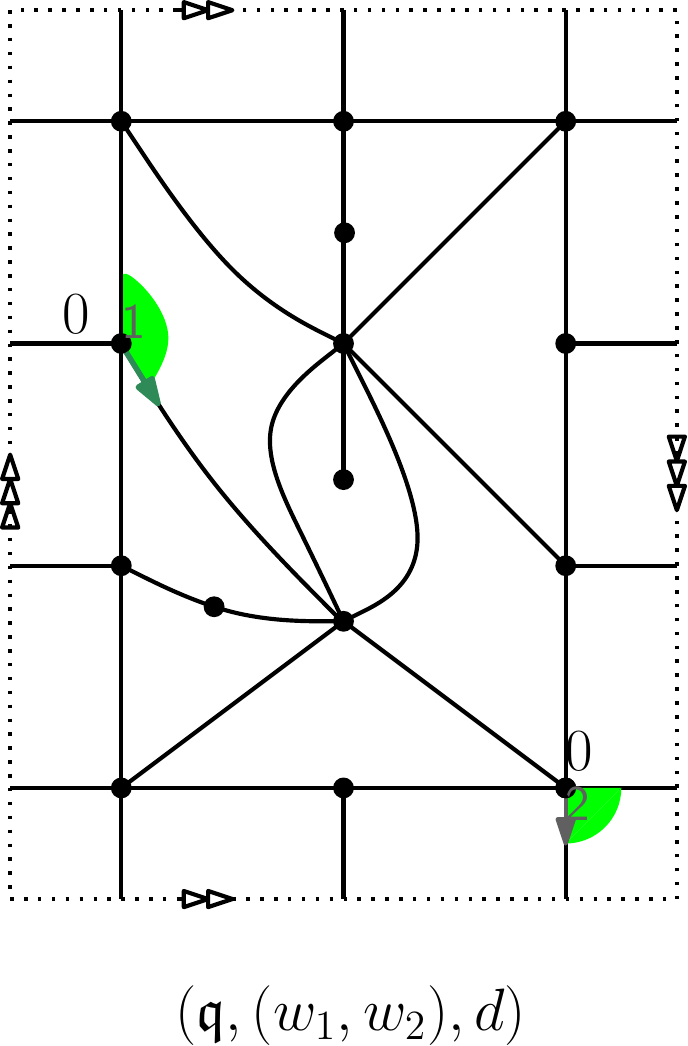}}
\subfloat[]{
	\label{subfig:exampleDEG-Miermont2}
	\includegraphics[width=0.33\linewidth]{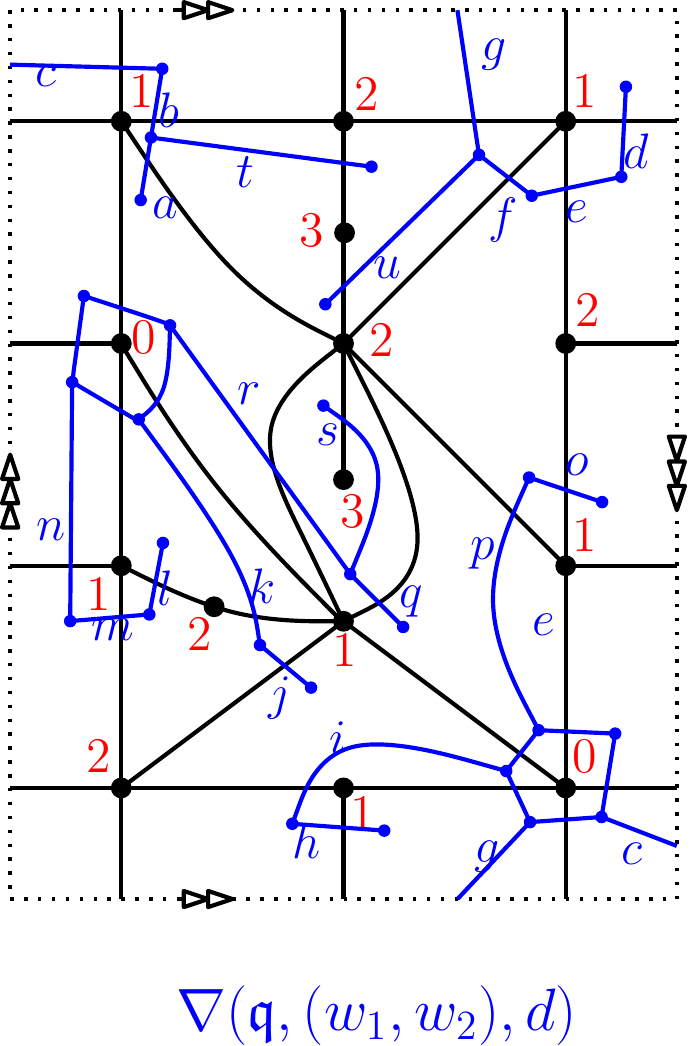}}
\subfloat[]{
	\label{subfig:exampleDEG-Miermont3}
	\includegraphics[width=0.33\linewidth]{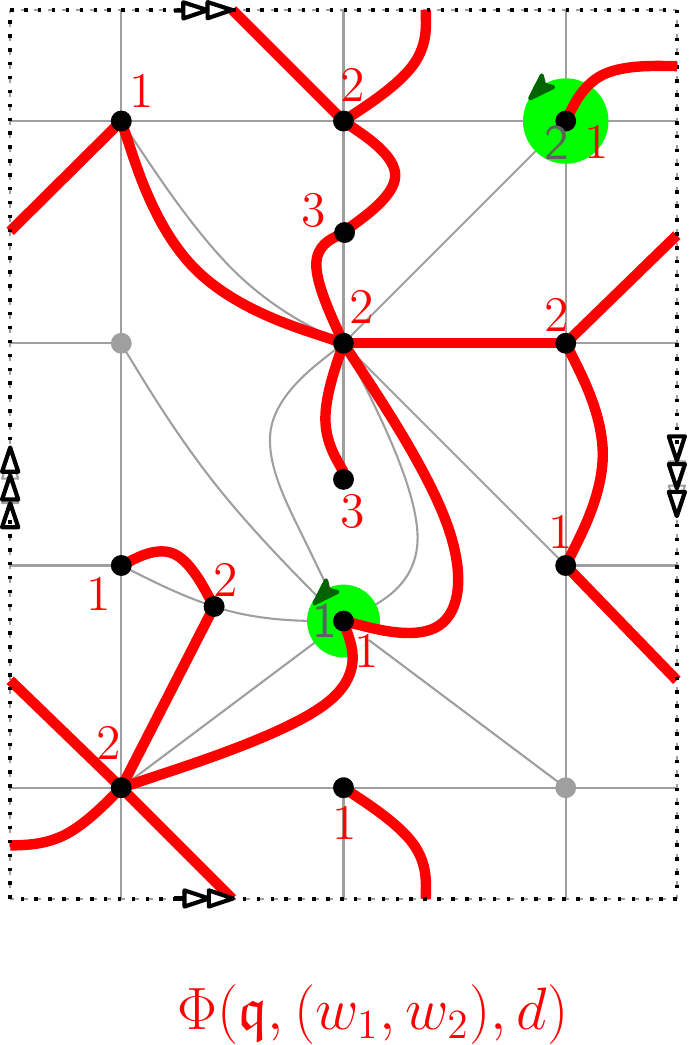}}
\caption{Construction (in red) of the labeled map $\bij(\qq, \{w_1,w_2\},d)$ associated with the bipartite quadrangulation of \cref{fig:exampleDEG} with two ordered rooted sources.}
\label{fig:exampleDEG-Miermont}
\end{figure}

\begin{proposition}\label{prop:welldefMiermont} The construction of $\DEG(\qq, W, d)$ is well-defined.
\end{proposition}

\begin{proof}
This proof is almost identical to the proof of \cref{prop:welldef}. The only slight difference is in the proof of the following claim: if there are still free edges of label $i$ in $\qq$, 
then at least one of them is incident to a face of minimum label $i-1$. In order to prove it, we assume the contrary. Let $e$ be a free edge of $\qq$ of label $i$ and let $a$ be the endpoint of label $i$ of $e$. Let $e=e_0, e_1, \dots e_r$ with $r\geq 1$ be the sequence of edges encountered clockwise around $a$ (in some conventional orientation) starting from $e$, where $e_r$ is the first edge clockwise after $e$  having label $i-1$.
\noindent Vertex $a$ is not a source, because all edges incident to sources are not free by \ref{StepM0b}. Let $w_j$ be a source such that $\ell^d(a) = d(w_j) + d_\qq(w_j,a)$. Let $a'$ be the vertex lying on some geodesic between $a$ and $w_i$ such that $d_\qq(w_j,a') = d_\qq(w_j,a) - 1$. Then $\ell^d(a') < i$ and $a'$ is linked to $a$ by an edge, thus $\ell^d(a) = i-1$ by \cref{lem:DistanceFromSources}, which proves that the edge $e_r$ exists. The edge $e_{r-1}$ has label $i$ and is incident to a face of minimum label $i-1$, so by the assumption we have made, $e_{r-1}$ is crossed by a blue edge. But according to the construction rules, this blue edge is part of a path of edges labeled $i$ turning around the vertex $a$ and originating in a face containing a corner of label $(i-1)$. This path must cross the edge $e$ (although we do not know in which direction). This is a contradiction and proves the claim.
\end{proof}

\subsubsection{Constructing the labeled map $\bij(\qq, W, d)$.}
Exactly as in the case of rooted, non-pointed quadrangulations, there are two possible types of faces of $\qq$ at the end of the construction of $\DEG(\qq, W, d)$ and we add one red edge in each face of $\qq$ according to the rule of~\cref{fig:redBlueRule}.
We let $\bij(\qq, W, d)$ be the map on $\mathbb{S}$ consisting of all the red edges, and of all the vertices of $\qq$ except $W$. For every root corner $c_i$ of $(\qq,W,d)$ there is a unique corner of $\bij(\qq,W,d)$ of label~$d(c_i)+1$  incident to the corresponding root edge of $(\qq,W,d)$. We define $c_i'$ to be that corner and we equip it with the local orientation inherited from the one of the root corner $c_i$.  \cref{subfig:exampleDEG-Miermont3} gives an example of the construction.

\begin{lemma}
$\bij(\qq, W, d)$ is a well-defined well-labeled map with $k$ faces numbered from $1$ to $k$, and one distinguished oriented corner in each face, where $k = |W|$.
\end{lemma}
\begin{proof}
We first notice that the number of connected components of $H=\mathbb{S}\setminus \bij(\qq, W, d)$ is equal to the number of connected components of $\DEG(\qq, W, d)$, that is equal to $k$ which follows from the construction of $\DEG(\qq, W, d)$. Indeed, in \ref{StepM0b} we constructed $k$ disjoint blue loops and in the following steps of the construction of $\DEG(\qq, W, d)$ we didn't change the number of the connected components of the already constructed blue graph. Moreover, every connected component $C$ of $H$ is simply connected since any loop in $C$ can be retracted to a point along the tree-like structure of the connected component of $\DEG(\qq, W, d)$ lying in $C$. This proves both that $\bij(\qq, W, d)$ is a valid map on $\mathbb{S}$, and that it has $k$ faces.

The fact that each root corner $c_i'$ lies in the different face of $\bij(\qq, W, d)$ is a consequence of the fact that each corner $c_i'$ belongs to the corresponding root edge of $\qq$ that is crossed by different connected component of $\DEG(\qq, W, d)$. Finally, it is clear that $\bij(\qq, W, d)$ is well-labeled.
\end{proof}

\subsubsection{From well-labeled maps to quadrangulations}
\label{sec:reversebijMiermont}

We now describe the reverse construction.
Again, the construction is almost identical to the construction from \cref{sec:reversebij}, and we will go rapidly through the parts that are truly identical.
 We let $\M$ be a $k$-rooted well-labeled map on a surface $\mathbb{S}$ with $n$ edges and $k$ faces, where $k$-rooted means that there are $k$ ordered roots in $\M$ such that each root corner lies in a different face of $\M$. From $\M$ we are going to reconstruct simultaneously two graphs: the quadrangulation with $k$ rooted delayed sources $(\qq, W, d)$ associated with $\M$, and the dual exploration graph $\DEG(\qq, W, d)$. As before, edges of the map $\M$, of the quadrangulation, and of the dual exploration graph will be called \emph{red}, \emph{black}, and \emph{blue}, respectively, vertices of the dual exploration graph will also be refereed to as \emph{blue}, and black edges will also be referred to as \emph{internal edges}.

\smallskip

\begin{enumerate}[label=$\bullet$, ref=Step MR0, leftmargin=0cm]
\item \label{StepMR0} \textbf{Step MR0}.
We consider the representation of $\M$ as a collection of $k$ ordered labeled polygons, each with one marked oriented corner, and with identified pairs. This representation easily generalizes the one of \cref{subsec:RepresentationOfMap}, so we don't give the details\footnote{The only difference is that now an edge is said \emph{twisted} iff the corresponding gluing is not compatible with the orientations (fixed by the orientations of the polygons) of the two edge-sides thus glued.}. We denote by $\pp$ the associated collection of polygons $\pp = \{\pp_j\}_{j \in [k]}$, by $c_j'$ the  marked oriented corner in $\pp_j$, and by $E_s(\M)$ and $E_t(\M)$ the associated sets of matchings. Inside each polygon $\pp_j$ we draw a vertex $w_j$ labeled by $\min_{v \in V(\pp_j)}\ell(v)-1$ and we connect each corner of $\pp_j$ labeled by $\min_{v \in V(\pp_j)}\ell(v)$ to $w_j$ by a new (black) edge. In this way we dissect each polygon $\pp_j$ into $k_j$ areas, where $k_j$ is the number of corners of $\pp_j$ labeled by $\min_{v \in V(\pp_j)}\ell(v)$. In each such area we now draw one blue vertex and we connect them by a directed (counterclockwise with respect to the orientation given by the corner $c_j'$) blue loop encircling $w_j$. There is a unique internal edge incident to the root corner $c_j'$, and a unique blue edge crossing it. We declare the blue corner lying at the extremity of that edge (and exterior to the cycle) to be the \emph{$j$-th last visited corner ($LVC_j$)} of the construction. We set $LVC := LVC_1$ and $i:=1$, and we continue.

\smallskip

We now proceed with the inductive step of the construction that will be totally similar to the easier case described in \cref{sec:reversebij}, that is we are going to construct simultaneously and recursively a planar graph $\rplan(\M)$ drawn in $\pp$ and a blue directed graph $\rDEG(\M)$ drawn in $\pp$ such that each connected component of $\rplan(\M)$ and of $\rDEG(\M)$ is drawn in some corresponding polygon $\pp_j$. 
 The vertices of that planar graph will be $V(\pp) \cup \{w_1,\dots,w_k\}$ and its edges (called internal edges because they are lying in the interior of polygons from $\pp$) will have the property that for every polygon $\pp_j$ each vertex $w \in V(\pp_j)$ labeled by $i$ will be connected to a unique vertex $w' \in V(\pp_j) \cup \{w_j\}$ labeled by $i-1$. Here, we will use the same notation and notions as we did in \cref{sec:reversebij} (such as area, types of areas, labels, etc.).
Again, the structure of the construction is analogous to the one of \cref{sec:reversebij} and we use the same indexing for the different steps, with a prefix letter ``M''.
\end{enumerate}

\smallskip

\begin{enumerate}[label=$\bullet$, ref=Step MR\arabic*, leftmargin=0cm]
\item \label{StepMR1} \textbf{Step MR1}.
If there are no vertices of label $i+1$ in $\pp$, go to the termination step.
Otherwise, walk along the blue graph, starting from the $LVC$ which is equal to $LVC_j$ for some $1\leq j\leq k$. As in \ref{StepM1}, if we visit the $LVC$ twice, we change its position to $LVC := LVC_{j+1}$ and we continue our tour from the new position. We let $F$ be the first visited area having the following properties:
\begin{itemize}
\item[$\bullet$] The minimum corner label in $F$ is $i-1$.
\item[$\bullet$] $F$ is of type \ref{T2} or \ref{T3}.
\item[$\bullet$] If $F$ is of type \ref{T3}, let $e$ be the unique (red) edge of $\pp$ bordering $F$. If $F$ is of type \ref{T2}, let $e$ be the last (red) edge of $\pp$ bordering $F$, having extremities labeled by $i$, and $i+1$, in the counterclockwise orientation induced by the blue graph on $F$ (see \cref{fig:stepR1}).
Let $\tilde{e}$ be the unique edge of $\pp$ that is matched with $e$ in the map structure inherited from $\M$, and let $\tilde{F}$ be the area containing $\tilde{e}$. Then $\tilde{F}$  is of type \ref{T2}.
\end{itemize}
Note that the rules for choosing $F$ are exactly the same as in \ref{StepR1}. \cref{prop:RwelldefMiermont} ensures that such an area $F$ always exists.

\item \label{StepMR2} \textbf{Step MR2}.
We follow the same rules as in \ref{StepR2} in \cref{sec:reversebij}.
Namely, we let $F, e, \tilde{e}$ and $\tilde{F}$ be defined as above, and $v$ be the vertex of $\pp_j$ of label $i$ incident to $\tilde{e}$. We first link $v$ by new internal edges to all the corners of $\tilde{F}$ having label $i+1$, without crossing any existing blue edge, thus subdividing $\tilde{F}$ into several new areas $f_1,f_2,\dots,f_k$.
We now add a new blue vertex $v_i$ in the new area $f_i$ for $1\leq i \leq k-1$, and we connect the vertices $v_1,v_2,\dots,v_k$ by a blue directed path as in \ref{StepR2} described in \cref{sec:reversebij}, see \cref{fig:stepR2} again.
\noindent We declare the corner of $v_k$ incident to the last drawn blue edge and exterior to $v$, as on~\cref{fig:stepR2}, to be the new $LVC_j$, and we set $LVC = LVC_j$.

\item \label{StepMR3} \textbf{Step MR3}. If each vertex of label $i+1$ of $\pp$ is linked to an internal edge, we set $i:=i+1$, otherwise we let $i$ unchanged. We then go back to \ref{StepMR1} and continue.
\end{enumerate}

\begin{enumerate}[label=$\bullet$, ref=Termination, leftmargin=0cm]
\item \label{StepMRTermination} \textbf{Termination.} We perform the identifications of edges of $\pp$ according to the map structure of $\M$, thus reconstructing the surface $\mathbb{S}$. We call $\rbij(\M)$ the map on $\mathbb{S}$ consisting of all the internal edges, with vertex set $V(\M)\cup\{w_1,\dots,w_k\}$.
\end{enumerate}

\smallskip
\cref{fig:exampleReverse-Miermont} gives an example of the construction.

\begin{figure}
\includegraphics[width=\linewidth]{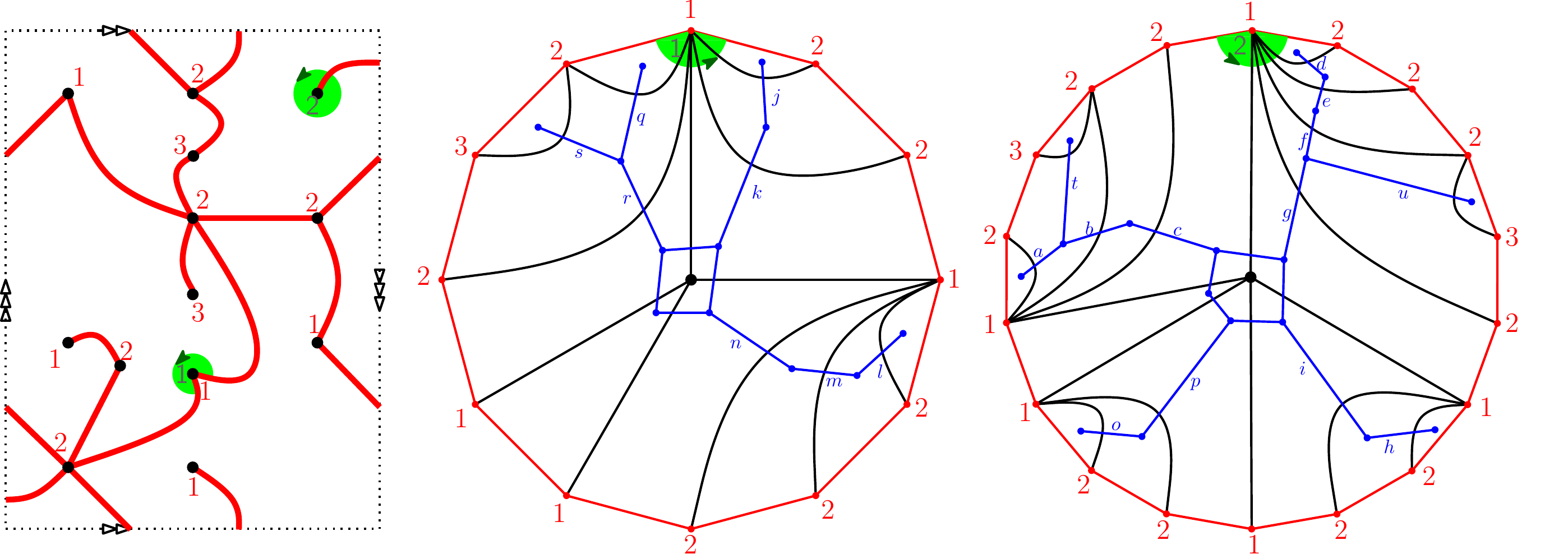}
\caption{Left: A labeled map $\M$ with two rooted and ordered faces on the Klein bottle. Right: The same map displayed as a gluing of two polygons, and the construction of the associated quadrangulation (black edges). To help the reader visualize the construction, blue edges added after \ref{StepMR0} have been numbered $a,b,\dots,z$ in their order of appearance in the construction. To make the picture lighter we have not drawn the orientation of the tree edges (they are all oriented towards the inner cycles).}
\label{fig:exampleReverse-Miermont}
\end{figure}

\begin{proposition}\label{prop:RwelldefMiermont}
The map $\rbij(\M)$ is a well-defined bipartite quadrangulation.
\end{proposition}

The proof of this statement is exactly the same as the proof of \cref{prop:Rwelldef} (in particular, \cref{lemma:PropertiesOfR} holds true for this more general construction), and we skip it. We conclude with the following observation.
Let us define $d: \{w_1,\dots,w_k\} \rightarrow \Z$ by setting $d(w_j) = \min_{v \in V(\pp_j)}\ell(v)-1$. In $\qq$, we distinguish the corner $c_j$ incident to $w_j$ defined as the unique oriented corner whose associated edge is incident to the corner $c_j'$ and such that the orientations of $c_j$ and $c_j'$ are the same. We have:

\begin{lemma}\label{lemma:labelsDistancesMiermont}
The map $(\rbij(\M), (w_1,\dots,w_k), d)$ is a quadrangulation with $k$ delayed sources and $k$ ordered roots $c_1,\dots,c_k$ attached to them. The label $\ell$ of any vertex in $V(\M)$ is equal to its distance $\ell^d$ from the delayed sources.
\end{lemma}

\begin{proof}
First, we have that
\[ \min_{1 \leq j \leq k}d(w_j) = \min_{v \in V(\pp)}\ell(v)-1 = 0,\]
thus condition~\ref{con:Delay0th} from \cref{def:MapWithSources} holds true. Since every two vertices of $\rbij(\M)$ are linked by an edge only if their labels differ by $1$, notice that~\ref{con:Delay2nd} from \cref{def:MapWithSources} is clearly satisfied. Moreover, from the same reason, for any $1 \leq i < j \leq k$ we have that $|d(w_i)-d(w_j)| \leq d_{\rbij(\M)}(w_i,w_j)$, with equality iff there exists a geodesic path from $w_i$ to $w_j$ with strictly monotone labels. But such a geodesic path does not exist since for every $1 \leq i \leq k$ each vertex linked to $w_i$ by an internal edge has greater label than $w_i$. This proves~\ref{con:Delay1st} from \cref{def:MapWithSources}.

It remains to prove that labelings $\ell$ and $\ell^d$ coincide. First, we show that for every vertex $v \in V(\pp_j)$, we have that $\ell(v) = d(w_j) + d_{\rbij(\M)}(v,w_j)$. It is a consequence of the fact that every corner of label $i$ is linked to some corner of label $i-1$ for all $i > d(w_j)$. Together with the fact that $d_{\rbij(\M)}(v,v') \leq |\ell(v) - \ell(v')|$ for any pair of vertices $v,v' \in V(\M)$ it shows that for a fixed $v \in V(\pp_j)$ one has
\[ \ell(v) \leq d(w_i) + d_{\rbij(\M)}(v,w_i)\]
for any $1 \leq i \leq k$ with an equality for $i=j$, and hence
\[ \ell(v) = \min_{1 \leq j \leq k}\left(d(w_i) + d_{\rbij(\M)}(v,w_i)\right) =\ell^d(v), \]
which finishes the proof.
\end{proof}

\subsubsection{Bijection}
We are now ready to reformulate \cref{theo:Miermont1} in the following more precise form.
\begin{theorem}
For each $n\geq 1$ and each surface $\mathbb{S}$, the mapping $\bij$ and $\rbij$ are reverse bijections between the set of bipartite quadrangulations on $\mathbb{S}$ with $n$ faces and $k$ ordered rooted delayed sources, and the set of well-labeled maps on $\mathbb{S}$ with $n$ edges and $k$ ordered rooted faces.
\end{theorem}

The proof of this theorem is 
similar to the one of \cref{theo:MS1} and we skip it. 

We now deduce \cref{theo:Miermont2} from \cref{theo:Miermont1}.
The argument is a direct adaptation of the one used to deduce \cref{theo:MS2} from \cref{theo:MS1} in \cref{sec:bijection}, so we just indicate the main differences.
The simultaneous choice, for each rooted bipartite quadrangulation $(\qq,W,d)$ with $k$ delayed sources of an oriented corner $\rho_i(\qq,W)$ incident to the source $w_i$ in $\qq$ for each $1\leq i \leq k$ is called an \emph{oracle}.
Once an oracle is fixed, we can view $\qq$ as a quadrangulation carrying $k$ \emph{rooted} delayed sources, call it $\qq'$ (with $w_i$ rooted at $\rho_i(\qq, W)$). This quadrangulation is equipped with an additional marked corner (the original root corner of $\qq$). We can then apply the bijection $\bij$ to $(\qq',W,d)$ and we obtain a well-labeled map $\M$ with $k$ rooted, ordered faces. Since $\M$ has $2n$ corners and $\qq'$ has $4n$ corners, we can use the marked corner of $\qq'$ to mark a corner of $\M$ and to get an additional sign $\epsilon\in\{+,-\}$ as in \cref{sec:bijection}. Declaring this corner to be the new root of $\M$, we can now shift all the labels of $\M$ by the same integer, in such a way that this corner receives the label $1$. We thus have obtained a labeled map $\M'$ with $k$ rooted, ordered faces, that carries a marked oriented corner $c$ of label $1$ (the root corner of $\M$), together with a  sign $\epsilon$, and we denote by $\Lambda_\rho(\qq,W,d):=(\M',c,\epsilon)$.

We thus obtain \cref{theo:Miermont2} in the following more precise form:
\begin{theorem}
\label{theo:Miermont2'}
For each $n\geq 1$ and each surface $\mathbb{S}$, there exists a choice of the oracle $\rho$ that makes $\Lambda_\rho$ a bijection between the set of rooted bipartite quadrangulations on $\mathbb{S}$ with $n$ faces and $k$ delayed sources, and the set of rooted labeled maps on $\mathbb{S}$ with $n$ edges and $k$ ordered faces equipped with a sign $\epsilon\in\{+,-\}$.
\end{theorem}
\begin{proof}
This is a straightforward adaptation of the proof of \cref{theo:MS2}.
For fixed $d_1,d_2,\dots d_k \geq 1$, consider the set $\mathcal{Q}_{\mathbb{S},n,d_1,\dots,d_k}$ of all rooted quadrangulations on $\mathbb{S}$ with $n$ faces and $k$ delayed sources such that the $i-th$ source has degree $d_i$, and the set $\mathcal{U}_{\mathbb{S},n,d_1,\dots,d_k}$ of all rooted labeled maps on $\mathbb{S}$ with $n$ edges and $k$ faces carrying a sign $\epsilon\in\{+,-\}$, such that the numbers of corners of minimum label in the $i-th$ face is given by the numbers $d_i$. Then the bijection of \cref{theo:Miermont1} endows the  disjoint union of these two sets with a regular bipartite graph structure, and applying Hall's marriage theorem, we obtain the existence of the wanted oracle. Details are similar to the proof of~\cref{theo:MS2}.
\end{proof}

\subsection{Ambjørn-Budd bijection for general surfaces}
\label{subsect:AB}

In the case of orientable surfaces, Ambjørn and Budd~\cite{AmbjornBudd} designed yet another variant of Schaeffer's bijection, that proved to be useful in the study of scaling limits of maps (see e.g. \cite{BJM}). The Ambjørn-Budd bijection is actually a corollary of Miermont's bijection, in the sense that it can be deduced from it.
As we will see, the same is true for non-orientable surfaces. In this section, using the generalization of Miermont's bijection done in the previous section, we will generalize Ambjørn and Budd's construction to all surfaces.

Let $(\M,v_0)$ be a map with pointed vertex $v_0$. We will say that a vertex $u \in V(\M)$ is \emph{extremal} in the pointed map $\M$ if 
\[ d_\M(u,v_0) > d_\M(v,v_0)\]
for any vertex $v$ that is adjacent to $u$. Otherwise, we will say that the vertex $u$ is \emph{non-extremal}. 

\begin{lemma}
\label{lem:distances}
Let $(\M,v_0)$ be a pointed map. Then, for any vertex $v \in V(\M)$, there exists an extremal vertex $u \in V(\M)$ such that
\[ d_\M(v_0,v) + d_\M(v,u) = d_\M(v_0,u).\]
\end{lemma}

\begin{proof}
If $v$ is an extremal vertex, we can choose $u = v$. Otherwise, there exists a vertex $u_1$ adjacent to $u_0:=v$ such that $d_\M(v_0,u_1) = d_\M(v_0,u_0) + 1$. 
Iterating as long as we can, we find an integer $I$ and vertices $u_{i+1}$ for $0\leq i \leq I$ such that a vertex $u_{i+1}$ is adjacent to $u_i$, $d_\M(v_0,u_{i+1}) = d_\M(v_0,u_i) + 1$, and $I$ is the largest index such that the vertex $u_{I+1}$ exists. Then $u := u_{I+1}$ is the desired extremal vertex.
\end{proof}

The notion of extremal vertices is important in the following theorem.
\begin{theorem}[Ambjørn-Budd bijection for general surfaces]
\label{theo:AB}
For each surface $\mathbb{S}$ and integer $n\geq 1$, there exists a $2$-to-$1$ correspondence between the set of rooted bipartite quadrangulations on $\mathbb{S}$ with $n$ faces carrying a pointed vertex $v_0$, and rooted maps on $\mathbb{S}$ with $n$ edges carrying a pointed vertex $\widetilde{v_0}$. Moreover, if a bipartite quadrangulation has $n_i$ non-extremal vertices at distance $i$ from the pointed vertex $v_0$ for some $i\geq 1$ and $k$ extremal vertices, then its associated map has $n_i$ vertices at distance $i$ from the pointed vertex $\widetilde{v_0}$ and $k$ faces.
\end{theorem}

In order to prove the theorem we first state the following lemma:
\begin{lemma}
\label{lem:LabelsAreDistances}
Let $(\qq,W, d)$ be a quadrangulation on $\mathbb{S}$ with $n$ faces and $k$ delayed sources and let $\Lambda_\rho(\qq,W, d) = (\M,c,\epsilon)$ be the associated labeled map, via the bijection of \cref{theo:Miermont2'}. Then, for every $v \in V(\M)$ $v$ is a local maximum in $\Lambda_\rho(\qq,W, d)$ iff $v$ is a local maximum in $(\qq,W, d)$.
\end{lemma}

\begin{proof}
If $v \in V(\M)$ is a local maximum in $(\qq,W, d)$ then by construction it is clearly a local maximum in $\Lambda_\rho(\qq,W, d)$. 

To prove the converse, let $v \in V(\M)$ be a vertex labeled by $i = \ell(v)$, which is not a local maximum in $(\qq,W, d)$. We need to prove that there exists a vertex $w \in V(\M)$ adjacent to $v$ in $\M$ and labeled by $i+1$. From \cref{lemma:labelsDistancesMiermont} and strictly from the construction of the map $\M$ we know that there exists a face $f \in F(\qq)$ of type $(i - 1, i, i + 1, i)$ containing the vertex $v$. Let $e \in E(\M)$ be the edge of $\M$ lying inside $f$ that connects vertices labeled by $i$ and $i+1$ (it exists from the rules of the construction: see \cref{fig:redBlueRule}). If this edge contains the vertex $v$, there is nothing to prove, so let us assume that the vertex $w$ labeled by $i+1$ is connected by $e$ to the second vertex $v' \neq v \in V(\qq)$ labeled by $i$ belonging to the face $f$ (see \cref{subfig:GeodesicInMiermont1}). \begin{figure}[h!]
\centering
\subfloat[]{
	\label{subfig:GeodesicInMiermont1}
	\includegraphics[scale=0.6]{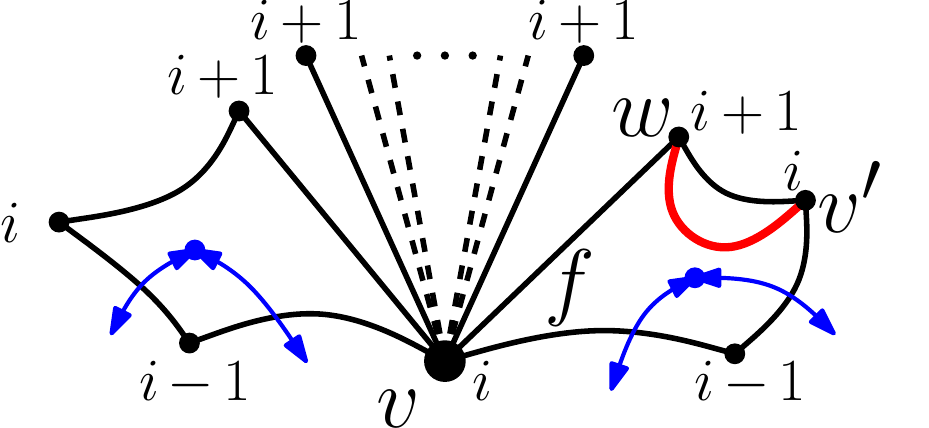}}
\subfloat[]{
	\label{subfig:GeodesicInMiermont2}
	\includegraphics[scale=0.6]{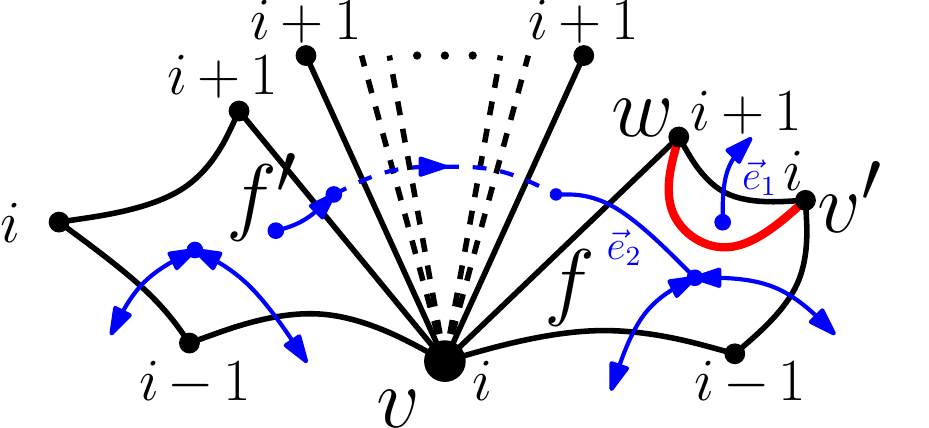}}
	
\subfloat[]{
	\label{subfig:GeodesicInMiermont3}
	\includegraphics[scale=0.6]{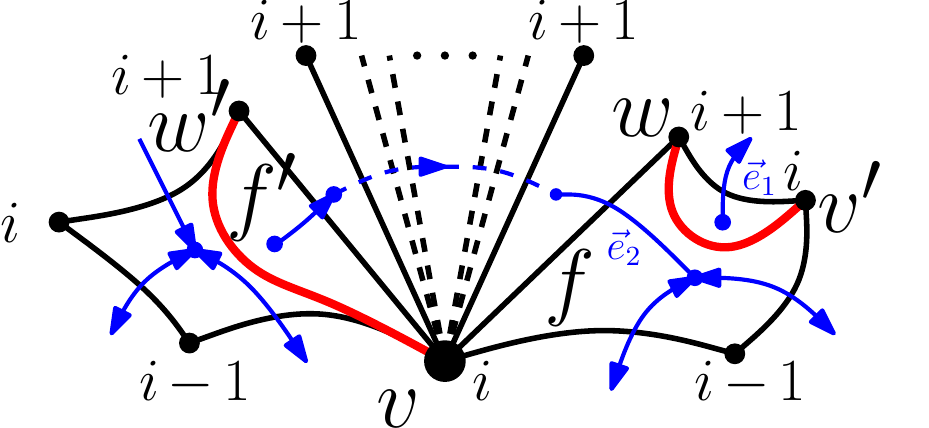}}
\caption{Illustration for the proof of \cref{lem:LabelsAreDistances}.} 
\label{fig:basicFaceRulesMMM}
\end{figure}
It means that the blue directed edge $\vec{e_1}$ crossing the edge of the face $f$ that connects $w$ with $v'$ is outgoing, hence the blue directed edge $\vec{e_2}$ crossing the edge of the face $f$ that connects $w$ with $v$ is incoming, because it is forced by the rules of the construction (see \cref{fig:redBlueRule}). Therefore, at some point of the construction of the labeled map $\M$ (\ref{StepM1}) there was chosen a face $f' \in F(\qq)$ of type $(i - 1, i, i + 1, i)$ from where a blue directed branch labeled by $i+1$ and prolonging the directed edge $\vec{e_2}$ was going out (see \cref{subfig:GeodesicInMiermont2}). But this means that the edge of the map $\M$ that is drawn inside the face $f'$ connects $v$ to some vertex $w' \in V(\qq)$ labeled by $i+1$ (see \cref{subfig:GeodesicInMiermont3}), which finishes the proof.
\end{proof}

Using the lemma above, one can restate \cref{theo:Miermont2'} in a different form. Then \cref{theo:AB} will be an obvious consequence of this reformulation.

\begin{definition}
Let $\qq$ be a rooted bipartite quadrangulation, and let $\ell : V(\qq) \to \N$ be a coloring of its vertices. We say that a triple $(\qq,\ell,\prec)$ is \emph{properly labeled} if every edge is labeled by $(i,i+1)$ for some $i \in \N$, at least 
one vertex of $\qq$ is labeled by $0$, and $\prec$ is a linear order on the set $V_{min}(\qq) \subset V(\qq)$ of all local minima of $\qq$ with respect to $\ell$.
\end{definition}

\begin{corollary}
\label{cor:theoMiermont}
For each $n\geq 1$ and each surface $\mathbb{S}$, there exists a choice of the oracle $\rho$ that makes $\Lambda_\rho$ a bijection between the set of rooted bipartite properly labeled quadrangulations on $\mathbb{S}$ with $n$ faces, and the set of rooted labeled maps on $\mathbb{S}$ with $n$ edges, ordered faces, and equipped with a sign $\epsilon\in\{+,-\}$.

Moreover, 
\begin{itemize}
\item there is a one-to-one correspondence between local minima $V_{min}(\qq)$ and faces of $\Lambda_\rho(\qq,\ell,\prec)$ such that if $f$ is a face of $\Lambda_\rho(\qq,\ell,\prec)$ associated with a vertex $v \in V_{min}(\qq)$ then $\ell(v)$ is given by the minimum of the labels around $f$ minus $1$;
\item labeled vertices of $\Lambda_\rho(\qq,\ell,\prec)$ correspond to labeled vertices $V(\qq)\setminus V_{min}(\qq)$;
\item local maxima of $(\qq,\ell,\prec)$ correspond to local maxima of $\Lambda_\rho(\qq,\ell,\prec)$.
\end{itemize}
\end{corollary}

\begin{proof}
It is clear from \cref{def:MapWithSources} and from \cref{lem:DistanceFromSources} that for any rooted bipartite quadrangulation with $k$ delayed sources $(\qq,W,d)$ the triple $(\qq,\ell^d,\prec)$ is a rooted bipartite properly labeled quadrangulation with $k$ local minima, where $\prec$ is a linear order of the set $W  = V_{min}(\qq)$. On the other hand, given a rooted bipartite properly labeled quadrangulation $(\qq,\ell,\prec)$ with $k$ local minima one can associate with it a rooted bipartite quadrangulation with $k$ delayed sources given by $(\qq,V_{min}(\qq),d)$, where $d(v) = \ell(v)$, and it is easy to see that this establishes a one-to-one correspondence between those sets. Then, the first two properties of the corollary hold trivially by construction, while the last property is a reformulation of \cref{lem:LabelsAreDistances}.
\end{proof}

\begin{proof}[Proof of \cref{theo:AB}]
Let $(\qq,v_0)$ be a rooted pointed bipartite quadrangulation with $k$ extremal vertices $(w_1,\dots,w_k)$. We associate with it a rooted bipartite properly labeled quadrangulation $(\qq,\ell,\prec)$ with $k$ local minima $V_{min}(\qq,\prec) = (w_1,\dots,w_k)$ by setting
\[\ell(v) = \max_{w \in V(\qq)}d_{\qq}(w,v_0)-d_{\qq}(v,v_0).\]
It is evident from the construction that we established a bijection between the set of rooted pointed bipartite quadrangulations with $k$ extremal vertices $w_1,\dots,w_k$ and the set of rooted bipartite properly labeled quadrangulation with $k$ local minima. Hence, by \cref{cor:theoMiermont} composing this bijection with the bijection $\Lambda_\rho$ of \cref{theo:Miermont2'},
we obtain a bijection between rooted pointed bipartite quadrangulations with $k$ extremal vertices $(\qq,v_0)$, and labeled, rooted, maps with $k$ ordered faces, and additional sign $(\M,\epsilon)$, where the labels have some special property. Namely, observe that $(\qq,\ell,\prec)$ has a unique local maximum at the vertex $v_0$, labeled by $\ell_{max} := \max_{w \in V(\qq)}d_{\qq}(w,v_0)$. Thus, by \cref{cor:theoMiermont}, the corresponding map $(\M,\ell,\epsilon) := \Lambda_\rho(\qq,\ell,\prec)$ has a unique local maximum at the vertex $\widetilde{v_0}  = v_0 \in V(\M)$. This means that the labels $\ell$ in the map $\M$ are encoded uniquely by the pointed map $(\M,\widetilde{v_0})$, which establishes desired bijection (indeed, for any vertex $v \in \M$ one has $\ell(v) = \ell(v) - d_\M(v,\widetilde{v_0})$, since $\widetilde{v_0}$ is the unique local maximum of $(\M,\ell)$).
Now, notice that we have the following equality between sets of vertices:
\[ \{v \in V(\qq)\setminus\{w_1,\dots,w_k\}: d_\qq(v_0,v) = i\} = \{v \in V(\M): d_\M(\widetilde{v_0},v) = i\},\]
because they both coincide with the multiset
\[ \{v \in V(\qq)\setminus\{v_1',\dots,v_k'\}:  \ell(v_0) - \ell(v) = i\} =  \{v \in V(\M): \ell(\widetilde{v_0}) - \ell(v) = i\}.\]
cIn this way we proved that there is a bijection between the set of rooted pointed bipartite quadrangulations on $\mathbb{S}$ with $n$ faces and $k$ ordered extremal vertices, and between the set of rooted pointed labeled maps on $\mathbb{S}$ with $n$ edges and $k$ ordered faces. To conclude the proof we need to show that one can omit the word ''ordered'' in the previous statement, and one can do it, again, by building an appropriate bipartite graph and applying Hall's marriage theorem. As we already used this technique several times in this paper, we let it as an easy exercise to the reader.
\end{proof}

\section*{Acknowledgments}

G.C. thanks Gilles Schaeffer for mentioning the problem to him, back in 2007.

\bibliographystyle{alpha}

\bibliography{nonorientable}

\end{document}